\documentclass[11pt,reqno]{amsart} 
\usepackage[margin=0.9in]{geometry}
\usepackage[dvipsnames]{xcolor}
\usepackage{amsthm, amsfonts, amssymb, amsmath,enumitem, bbm, mathabx}
\usepackage{mathtools}
\mathtoolsset{showonlyrefs}
\usepackage{mathrsfs} 
\usepackage{pgf,tikz,pgfplots,graphicx}
\usetikzlibrary{shapes.geometric}
\usetikzlibrary{arrows}
\usepackage[
            CJKbookmarks=true,
            bookmarksnumbered=true,
            bookmarksopen=true, 
            colorlinks=true,
            citecolor=red,
            linkcolor=blue,
            anchorcolor=red,
            urlcolor=blue
            ]{hyperref}
\usepackage{enumitem}
\usepackage{comment}
\usepackage{chngcntr}
\usepackage{mathabx}
\usepackage{bm}
\usepackage{etoolbox}
\usepackage{enumitem}
\usepackage{subcaption}
\usepackage{floatrow}
\usepackage{float}
\usepackage{pkgfile}
\usepackage{parskip}

\usepackage{abraces} 
\usepackage{eufrak}
\usepackage{upgreek} 
\usepackage{booktabs} 
\usepackage{array} 
\usepackage{verbatim}

\usepackage[numbers, sort, square]{natbib}

\newtheorem{thm}{Theorem}[section]

\newtheorem{cor}[thm]{Corollary}

\newtheorem{prop}[thm]{Proposition}

\newtheorem{lem}[thm]{Lemma}

\newtheorem{assumption}[thm]{Assumption}

\theoremstyle{definition} 
\newtheorem{defn}[thm]{Definition}

\newcommand{\kerK}{\mathsf{K}}
\newcommand{\kerR}{\mathsf{R}}
\newcommand{\kerA}{\mathsf{A}}
\newcommand{\kerB}{\mathsf{B}}
\newcommand{\kerC}{\mathsf{C}}

\newcommand{\revsag}[1]{\textcolor{black}{#1}}

\newcommand{\nbsag}[1]{\textcolor{blue}{\texttt{[#1-sagnik]}}}

\allowdisplaybreaks

\begin{document}

\title[Extremal Eigenvalues of Random Kernel Matrices]{Extremal Eigenvalues of Random Kernel Matrices with Polynomial Scaling}
\author[D. Kogan, S. Nandy, J. Huang]{David Kogan, Sagnik Nandy and Jiaoyang Huang}
\address{Department of Mathematics\\ Yale University\\ New Haven\\ CT 06520\\ United States}
\email{d.kogan@yale.edu}
\address{Department of Statistics\\ University of Chicago\\ Chicago\\ IL 60637\\ United States}
\email{sagnik@uchicago.edu}
\address{Department of Statistics and Data Science\\ University of Pennsylvania\\ Philadelphia\\ PA 19104\\ United States}
\email{huangjy@wharton.upenn.edu}

\begin{abstract} 
We study the spectral norm of random kernel matrices with polynomial scaling, where the number of samples scales polynomially with the data dimension. In this regime, Lu and Yau (2022) proved that the empirical spectral distribution converges to the additive free convolution of a semicircle law and a Marcenko-Pastur law. We demonstrate that the random kernel matrix can be decomposed into a ``bulk" part and a low-rank part. The spectral norm of the ``bulk" part almost surely converges to the edge of the limiting spectrum.

In the special case where the random kernel matrices correspond to the inner products of random tensors, the empirical spectral distribution converges to the Marcenko-Pastur law. We prove that the largest and smallest eigenvalues converge to the corresponding spectral edges of the Marcenko-Pastur law.
\end{abstract}


\maketitle

\section{Introduction}
Consider random vectors $X_1,\ldots,X_n \in \R^p$, where for all $i \in \{1,\ldots,n\}$, $X_i=(X_{i1},\ldots,X_{ip})$. For any kernel function $k(x)$, we consider the random kernel matrix $\kerK(X) \in \R^{n \times n}$ given by
\begin{align}
    \label{eq:def_kernel}
    [\kerK(X)]_{ij}&:=\begin{cases}
    \frac{1}{\sqrt{n}}k\left(\frac{\langle X_i,X_j\rangle}{\sqrt{p}}\right), & \mbox{if $i \neq j$}\\
    0, & \mbox{otherwise.}
    \end{cases}
\end{align}
 Random kernel matrices \eqref{eq:def_kernel} naturally arise in various contexts such as nonlinear dimension reduction \cite{belkin2003laplacian}, the sparse PCA problem in statistics \cite{deshp2016sparse}, and the random feature model \cite{hastie2022surprises, louart2018random, pennington2017nonlinear, rahimi2007random}. In all these examples, the spectral properties of \eqref{eq:def_kernel}  are crucial. In this work, we study the spectral norm of the random kernel matrix $\kerK(X)$ with polynomial scaling, where the number of samples scales polynomially with the data dimension: $n/p^\ell\rightarrow \gamma\in (0,\infty)$ for some positive integer $\ell$.

In the linear regime ($\ell=1$), Cheng and Singer \cite{doi:10.1142/S201032631350010X} showed that the empirical spectral distribution of 
$\kerK(X)$ converges to a deterministic limit distribution, characterized by its Stieltjes transform. Subsequently, Fan and Montanari \cite{Fan2019} discovered that the limiting empirical spectral distribution is described by the additive free convolution of a semicircle law and a Marcenko-Pastur law. They also demonstrated that for odd kernel functions, the spectral norm of $\kerK(X)$ almost surely converges to the edge of the limiting spectrum measure. For the polynomial regime ($\ell\geq 2$), Lu, Yau and their collaborators \cite{2022arXiv220506308L,dubova2023universality} recently proved a similar decomposition of the limiting empirical density. In this work, we investigate the extreme eigenvalues of the random kernel matrix 
$\kerK(X)$ in the polynomial regime $\ell\geq 2$.

\subsection{Empirical Eigenvalue Distribution of $\kerK(X)$}
Let us consider the sequence of normalized Hermite polynomials $\{H_k(x): k \ge 0\}$, where $H_0(x):=1$ and for all $k \ge 1$,
 \begin{align}
   \label{eq:def_hermite}
    H_k(x):=\frac{(-1)^k}{\sqrt{k!}}e^{\frac{x^2}{2}}\left(\frac{d^k}{dx^k}e^{-\frac{x^2}{2}}\right).  
 \end{align}
Observe that the sequence $\{H_k(\cdot): k \ge 0\}$ constitutes a collection of orthonormal bases in the inner product space of polynomials where the inner product between two polynomials $f(\cdot)$ and $g(\cdot)$ is given by $\mathbb E[f(\xi)g(\xi)]$ for $\xi \sim \dN(0,1)$. 

Next, consider the semicircle distribution supported on $[-2, 2]$, with the probability density function given by
\begin{align}
\label{eq:def_semicircle}
\mu_{sc}(x) = \frac{1}{2\pi}\sqrt{4-x^2}\mathbbm 1\{x \in [-2,2]\},
\end{align}
and the standard Marcenko-Pastur law $\mu_{mp, \gamma}$, the limiting empirical spectral measure of $\frac{1}{p}YY^T$ when $Y\in \mathbb R^{n\times p}$ with $Y_{ij} \diid \dN(0,1)$ and $n/p \to \gamma$ for $n,p \to \infty$. Let us define:
\begin{align}
\label{eq:end_distrn}
\mu_{a,b,\gamma} = \frac{a}{\sqrt{\gamma \ell!}}(\mu_{mp,\gamma\ell!}-1)\boxplus \sqrt{b}\mu_{sc}
 \end{align}
where $a = \mathbb E[k(\xi)H_\ell(\xi)]$ and $b = \mathbb E[k(\xi)^2] - \sum_{m=0}^\ell\mathbb E[k(\xi)H_m(\xi)]$ for $\xi \sim \dN(0,1)$. Here, $\frac{a}{\sqrt{\gamma \ell!}}(\mu_{mp,\gamma\ell!}-1)$ is the distribution of $\frac{a}{\sqrt{\gamma \ell!}}(y-1)$ with $y \sim \mu_{mp,\gamma\ell!}$ and  $\sqrt{ b}\,\mu_{sc}$ is the distribution of $\sqrt{b}
\,\bar y$ with $\bar y \sim \mu_{sc}$. Furthermore, $\boxplus$ denotes the free additive convolution of the two measures. In other words, if $W\in \mathbb R^{n\times n}$ is a GOE matrix, i.e. $W_{ij} \sim \dN(0,1)$ for $i \neq j$ and $W_{ii} \sim \dN(0,2)$ with independent entries and $Z\in \mathbb R^{n\times {p\choose \ell}}$ be a matrix with i.i.d. $\dN(0,1)$ entries, then $\mu_{a,b,\gamma}$ is the limiting spectral measure of
\begin{align}
    \label{eq:bulk_matrix}
         M = \frac{a\sqrt{\ell!}}{\sqrt{np^\ell}}V+\sqrt{\frac{b}{n}} W,
    \end{align}
where $V=ZZ^T-{p \choose \ell}I_{n}$. In this paper, we shall consider the setting where  
\begin{align}
    \label{eq:master_assumption}
    \mathbb EX_{ij}=0,\quad \mathbb EX_{ij}^2=1,\quad \mbox{and}\quad \mathbb E|X_{ij}|^k\leq k^{\beta k},\quad \mbox{for all $k \ge 0$ and some $\beta>0$.}
\end{align}
Now, we recall the following theorem from  \cite[Theorem 2]{2022arXiv220506308L} and \cite[Theorem 2.7]{dubova2023universality}, characterizing the bulk behavior of the spectrum of $\kerK(X)$.
\begin{thm}[{\cite[Theorem 2]{2022arXiv220506308L},\cite[Theorem 2.7]{dubova2023universality}}]
\label{thm:bulk}
Adopt assumptions in \eqref{eq:master_assumption}.
    Let us denote the Steiljes transform of the empirical eigenvalue of $\kerK(X)$ by $s_{\kerK(X)}(z)$. Then, for $z=E+i\eta \in \mathbb C^+$ satisfying $|E| \le \tau$ and $\tau \le \eta \le \tau^{-1}$ for a constant $\tau>0$, we almost surely have $s_{\kerK(X)}(z) \rightarrow m(z)$, where $m(z)$ is the unique solution in $\mathbb C^+$ to the equation 
    \begin{equation}
    \label{eq:self_consistent}
        m(z)\bigg(z+\frac{a^2 m(z)}{1+\sqrt{\ell!\gamma}\,a\, m(z)}+b^2 m(z)\bigg)+1=0.
    \end{equation}
\end{thm}
From \cite{doi:10.1142/S201032631350010X}, we can conclude that the solution $m(z)$ of equation \eqref{eq:self_consistent}, is the Steiljes transform of the distribution $\mu_{a,b,\gamma}$ defined in \eqref{eq:end_distrn}.

\subsection{Eigenvalues of the Simplified Tensor Matrix}

A special case, when the kernel function $k(x)=x^d$ is a monomial,  is inspired by the question: \emph{Are random tensors well conditioned?} Suppose $X_1, X_2,\cdots, X_n$ are independent copies of a random vector in ${\mathbb R}^p$ with independent coefficients. How large can $n$ be so that these random tensors $X_1^{\otimes d},X_2^{\otimes d},\cdots,X_n^{\otimes d}$ are linearly independent with high probability? Moreover, instead of linear independence,
we may ask for a stronger, more quantitative property of being well conditioned. We would
like to have a uniform bound
\begin{align}
\label{e:small_eign}\left\|\sum_{i=1}^n a_i {\rm vec}[X_i^{\otimes d}]\right\|_2\geq \sigma  \|\bm a\|_2,\quad \text{ for all } {\bm a}=(a_1, a_2,\cdots, a_n)\in {\mathbb R}^n,
\end{align}
with $\sigma$ as large as possible. Equivalently, we can understand \eqref{e:small_eign} as a lower bound on the
smallest singular value of the $n\times p^d$ matrix $X^{\odot d}$
(called the Khatri-Sidak product) whose
columns are vectorized tensors $X_1^{\otimes d},X_2^{\otimes d},\cdots,X_n^{\otimes d}$:
\begin{align*}
    \sigma_{\rm min}(X^{\odot d})\geq \sigma.
\end{align*}

Problems of this type were studied in theoretical computer science community
in the context of  learning Gaussian mixtures \cite{anderson2014more}, and
tensor decompositions
\cite{bhaskara2014smoothed,anari2018smoothed} (for the non-symmetric version, $X^{(1)}\odot X^{(2)}\odot \cdots \odot X^{(d)}$).

For $d=1$, this problem, called the singularity problem of random matrices, has been extensively studied (see \cite{rudelson2009smallest,vershynin2010introduction}). In this case, known optimal results yield that
\begin{align}\label{e:singular_value}
    \sigma_{\min} (X)\gtrsim \varepsilon \sqrt{n}
\end{align}
if $n=(1-\varepsilon)d$ and $0<\varepsilon \leq1$.

For $d\geq 2$, optimal results on $\sigma=\sigma_{\min}(X^{\odot d})$ are yet unknown. For the nonsymmetric version, \citet{rudelson2012row} showed that $\sigma\gtrsim c p^{d/2}$
for fixed $d$ and $n\ll p^d$ using chaining. Furthermore,
it was proven by Vershynin in \cite{vershynin2020concentration} that $\sigma \gtrsim \varepsilon$ using the concentration for random tensors for $d=o(\sqrt{n/\log n})$ and $n=(1-\varepsilon)p^d$.

The random kernel matrix \eqref{eq:def_kernel}, when the kernel function $k(x)=x^d$ is related to $(X^{\odot d})^\top X^{\odot d}$. As a corollary of our main result, we prove an optimal bound on the smallest singular value of $\sigma_{\rm}(X^{\odot d} )$
for fixed $d$, which matches the optimal results \eqref{e:singular_value} for random matrices.

 We prove the following theorem regarding the eigenvalues of a 
``simplified" tensor matrix. For $d\in \Z^+$, define $\kerR_d(X)\in \R^{n\times n}$ as follows:
\begin{align}   \label{e:def_tensor_matrix} 
[\kerR_d(X)]_{ij}= \frac{1}{\sqrt{np^d}}\sum\limits_{\substack{k_1<k_2<\ldots< k_d\\k_1,\ldots,k_d \in [p]}}(X_{ik_1}X_{ik_2}\cdots X_{ik_d})(X_{jk_1}X_{jk_2}\cdots X_{jk_d}).
\end{align}

Let $\mathscr I_d=\left\{\mathscr S_{d,1},\ldots,\mathscr S_{d,{p \choose d}}\right\}$ be an enumeration of the set of all $d$-element subsets of $\{1,\ldots,p\}$ containing all distinct elements. Let us define the matrices $X_d \in \R^{n \times {p \choose d}}$ as follows:
\[
[X_d]_{ij} = \prod_{k \in \mathscr S_{d,j}}X_{ik}.
\]
Then we have $\kerR_d(X)=\frac{1}{\sqrt{np^d}}X_d X_d^\top$.
Furthermore, denote $\bar \kerR_d(X):=\frac{1}{\sqrt{np^d}}X_d^\top X_d$.
Note that $\kerR_d(X)$ and $\bar{\kerR}_d(X)$ has the same set of non-zero eigenvalues. In the following theorem, we describe the asymptotic properties of the extremal non-trivial eigenvalues of $\kerR_d(X)$. 

\begin{thm}\label{thm:tensor_matrix}
Adopt assumptions in \eqref{eq:master_assumption}. Consider the matrices $\kerR_d(X)$ defined in \eqref{e:def_tensor_matrix}. As $n,p\to\infty$, let $n/p^\ell \rightarrow \gamma \in (0,\infty)$. If $n > {p \choose d}$ for all $n,p$ large, then almost surely $\lim_{n\to\infty}\lambda_{\min} (\kerR_d(X))=0$. Let $\lambda_{\min}=\lambda_{(1)}<\lambda_{(2)}<\ldots<\lambda_{(n)}=\lambda_{\max}$ be the eigenvalues of $\kerR_d(X)$ in increasing order and let $k=n-{p \choose d}$. Then we have the following:

\begin{enumerate}
    \item If $1 \le d<\ell$, then for any given $\varepsilon>0$, there exists a $n_0(\varepsilon,d)$ such that for all $n \ge n_0(\varepsilon,d)$
\begin{equation}
\label{eq:eig_r_d_bar/-1}
\sqrt{\frac{n}{p^d}}-\frac{2}{\sqrt{d!}}(1+\varepsilon) \le  \lambda_{(k+1)}(\kerR_d(X)) \le \lambda_{\max}(\kerR_d(X))\le \sqrt{\frac{n}{p^d}}+\frac{2}{\sqrt{d!}}(1+\varepsilon), 
\end{equation}
almost surely. In this case, the rank of $\kerR_d(X)$ is almost surely equal to ${p \choose d}$ and the smallest $k$ eigenvalues of $\kerR_d(X)$ are zero. 
\item If $d=\ell$, then we have two cases:
\begin{itemize}
    \item[(a)] If $n \le {p \choose \ell}$ for all $n,p$ large, then almost surely
    \begin{align}
\lim_{n\to\infty}\lambda_{\min} (\kerR_d(X)-\diag \kerR_d(X))&= -\frac{2}{\sqrt {d!}}+\sqrt{\gamma},\nonumber\\ \lim_{n\to\infty} \lambda_{\max} (\kerR_d(X)-\diag \kerR_d(X))&= \frac{2}{\sqrt{d!}}+\sqrt{\gamma}.\nonumber
\end{align}
\item[(b)]If $n > {p \choose \ell}$ for all $n,p$ large, then almost surely
    \begin{align}
\lim_{n\to\infty}\lambda_{(k+1)}(\kerR_d(X)-\diag \kerR_d(X))&= -\frac{2}{\sqrt {d!}}+\sqrt{\gamma},\nonumber\\ \lim_{n\to\infty} \lambda_{\max} (\kerR_d(X)-\diag \kerR_d(X))&= \frac{2}{\sqrt{d!}}+\sqrt{\gamma}.\nonumber
\end{align}
\end{itemize}
\item If $d > \ell$, then almost surely
\begin{align*}
\lim_{n \rightarrow \infty}\lambda_{\min}(\kerR_d(X)-\diag \kerR_d(X)) &= -\frac{2}{\sqrt{d!}},\\
\lim_{n \rightarrow \infty}\lambda_{\max}(\kerR_d(X)-\diag \kerR_d(X)) &= \frac{2}{\sqrt{d!}}.\nonumber
\end{align*}
\end{enumerate}
\end{thm}

\begin{cor}
   For fixed $d$ pick $0<\varepsilon<1$ and set $n=(1-\varepsilon) p^d/d!$. Consider $n$ independent random vectors $X_1, X_2,\cdots, X_n$ with 
   $\mathbb EX_{ij}=0$, $\mathbb EX_{ij}^2=1$, and $\mathbb E|X_{ij}|^k\leq k^{\beta k}$. Then the random tensors $X_1^{\otimes d},X_2^{\otimes d},\cdots,X_n^{\otimes d}$ are almost surely well conditioned. In other words, 
   \begin{align*}
\left\|\sum_{i=1}^n a_i {\rm vec}[X_i^{\otimes d}]\right\|_2\geq C_{\varepsilon, d} \sqrt{n} \|\bm a\|_2,\quad \text{ for all } {\bm a}=(a_1, a_2,\cdots, a_n)\in {\mathbb R}^n.
\end{align*}

\begin{proof}
Let $X\in \mathbb R^{n\times p^d}$ be the matrix whose $i$th row is $X_i^{\otimes d}$. Let $(p)_d:=p(p-1)\cdots(p-d+1)$ denote the falling factorial of order $d$. Furthermore, let $Y\in \mathbb R^{n\times (p)_d}$ denote the matrix obtained by retaining the columns of $X$ corresponding to products $X_{ij_1}X_{ij_2}\dots X_{ij_d}$ where $j_1\ne j_2\ne \dots\ne j_d$. Then, since $YY^\top=XX^\top-E$, where $E$ is a positive semi-definite matrix, we have $\lambda_{\min}(YY^\top)\leq \lambda_{\min}(XX^\top)$. It is not hard to show that $YY^\top=d!\sqrt{np^d}\cdot \kerR_d(X)$. Therefore, by Theorem \ref{thm:tensor_matrix} and \eqref{eq:asymp_diag}, we can conclude that for all $\delta>0$, there exists $n_0(\delta)$ such that for all $n \ge n_0(\delta)$, \begin{align*}
\lambda_{\min}(YY^\top)&\geq d!(\sqrt{np^d})\left(\sqrt{\gamma}-\frac{2}{\sqrt{d!}}+\frac{1}{\sqrt{\gamma} d!}\right) \\
&= \sqrt{np^d d!}\left(\sqrt{1-\varepsilon}-2+\frac{1}{\sqrt{1-\varepsilon}}\right)>C_{\varepsilon,d}\; n, \quad \mbox{almost surely}.
\end{align*}
In the last line, we have used the AM-GM inequality and the fact that $\sqrt{1-\varepsilon}\ne \frac{1}{\sqrt{1-\varepsilon}}$.
\end{proof}
\end{cor}

\subsection{Asymptotics of the Spectral Norm of $\kerK(X)$}

As in \cite{Fan2019}, we can conclude that the measure $\mu_{a,b,\gamma}$ is compactly supported. It has one interval of support for $\gamma\le1$ and at most two intervals of support (with a point mass at $0$) for $\gamma>1$. The support of $\mu_{a,b,\gamma}$ can be numerically computed from \eqref{eq:self_consistent}. The support points of $\mu_{a,b,\gamma}$ are given by the set of $z \in \mathbb C^+$ where the cubic equation \eqref{eq:self_consistent} admits an imaginary root. Let us consider the edge of the support of $\mu_{a,b,\gamma}$ defined as follows:
\begin{align}
\label{eq:edge_of_spectrum}
    \|\mu_{a,b,\gamma}\|:=\max\{|x|:x \in \mathrm{supp}(\mu_{a,b,\gamma})\},
\end{align}
where $\mathrm{supp}(\mu_{a,b,\gamma})$ denotes the support of $\mu_{a,b,\gamma}$.
Our main result states that the random kernel matrix $\kerK(X)$ from \eqref{eq:def_kernel} can be decomposed into a ``bulk" part and a low-rank part. The spectral norm of the ``bulk" part almost surely converges to $\|\mu_{a,b,\gamma}\|$.

In this section, we further assume that the random vectors $X_i$
  have Bernoulli entries. In this context, 
$\kerK(X)$ can be viewed as a random matrix over the finite field 
$\mathbb F_2$. Its singularity properties have been studied in \cite{abbe2015reed} to estimate the capacity of error-correcting codes.
 
\begin{assumption}\label{a:Bernoulli}
    $X_{ij}$ are indepedent Bernoulli random variables: ${\mathbb P}(X_{ij}=\pm 1)=1/2$.
\end{assumption}

Under Assumption \ref{a:Bernoulli}, the random kernel matrix $\kerK(X)$ decomposes as a linear combination of the ``simplified" tensor matrix $\kerR_d(X)$ from \eqref{e:def_tensor_matrix} 
\begin{align}\label{e:decomposition}
    \kerK(X)=\sum_{d=1}^p a_d\sqrt{d!}(\kerR_d(X)-\diag \kerR_d(X))
\end{align}
Now, let us consider the following theorem, characterizing the spectral norm of matrices $\kerK(X)$ satisfying \eqref{e:decomposition}. It is worth mentioning that the following theorem holds under certain moment assumptions on the entries $X_{ij}$. It is not necessary to assume Assumption \ref{a:Bernoulli} for this theorem.
\begin{thm}
\label{thm:main_1}
Adopt assumptions in \eqref{eq:master_assumption}. Suppose there exist a positive integer $\tilde L>\ell$ and $0<\alpha \le e^{-C(\ell)}$, where $C(\ell)>0$ is an absolute constant depending only on $\ell$, such that the matrix $\kerK(X) \in \R^{n \times n}$ satisfy
 \[
 \kerK(X)=\sum_{d=1}^\infty a_d\sqrt{d!}(\kerR_d(X)-\diag \kerR_d(X))
 \]
 with $|a_d| \le C_3\sqrt{d+1}\alpha^d$ for an absolute constant $C_3>0$ and all $d \ge \tilde L$. Then, we have
 \[
\kerK(X) = \tilde{\kerK}(X) + \sum_{d=1}^{\ell-1} a_d\,\sqrt{d!}(\kerR_d(X)-\diag \kerR_d(X)),
\]
where $\|\tilde{\kerK}(X)\| \asto \|\mu_{a,b,\gamma}\|$, as $n,p \rightarrow \infty$ and $n/p^\ell \rightarrow \gamma \in (0,\infty)$.
\end{thm}

Additionally, under Assumption \ref{a:Bernoulli}, we can get the following equivalent representation of the matrix $\kerK(X)$ defined in \eqref{eq:def_kernel}.
\begin{lem}
    \label{lem:lemma_k_x_expansion}
    For any analytic $k(x)$, under Assumption \ref{a:Bernoulli}, the matrix $\kerK(X)$ defined in \eqref{eq:def_kernel} satisfies
    \begin{align}
        \label{eq:kernel_expansion}
[\kerK(X)]_{ij}=\begin{cases}\sum\limits_{d=1}^{p}\frac{a_d}{\sqrt{np^dd!}}\sum\limits_{\substack{k_1\neq k_2\neq\ldots\neq k_d\\k_1,\ldots,k_d \in [p]}}(X_{ik_1}X_{ik_2}\cdots X_{ik_d})(X_{jk_1}X_{j  k_2}\cdots X_{jk_d}), & \mbox{if $i \neq j$}\\
        0, & \mbox{otherwise,}
        \end{cases}
    \end{align}
    for all $(i,j) \in [n] \times [n]$. In the above equation, the constants $\{a_d:d\in [p]\}$ satisfies $a_d=\mathbb E[k(\xi)H_d(\xi)]+o(1)$, where $\xi \sim \dN(0,1)$ and for all $k \in [p]$, $H_k(\cdot)$ is defined by \eqref{eq:def_hermite}. Furthermore, if there exists constants $\al \in (0,e^{-1})$ and $A>0$ such that  
\begin{equation}
\label{eq:kernel_bound}
|k(x)|\leq A e^{\al |x|} \quad \mbox{when $|x| \ge t$ for some $t>\ell/\alpha$,},
\end{equation}
then, for all $d \ge \lceil t\alpha\rceil$, the coefficients $a_d$ are bounded as follows:
  \begin{align}
      |a_d|\leq (3A\sqrt{2e\pi})\sqrt{d+1}\al^d.
  \end{align}
\end{lem}
From the above decomposition, we get the following theorem describing the asymptotic behaviour of the spectral norm of $\kerK(X)$ defined in \eqref{eq:def_kernel}.
\begin{thm}
\label{thm:main}
Let us adopt Assumption \ref{a:Bernoulli} and consider an analytic kernel $k(x)$ satisfying  
$
|k(x)|\leq A e^{\al |x|} 
$
when $|x| \ge t, t>\ell/\alpha, A>0, \al \in (0,e^{-C(\ell)})$ and $C(\ell)$ is defined in Theorem \ref{thm:main_1}.
If $n/p^\ell \rightarrow \gamma \in (0,\infty)$, then the matrix $\kerK(X)$ defined in \eqref{eq:def_kernel} satisfies
\[
\kerK(X) = \tilde{\kerK}(X) + \sum_{d=1}^{\ell-1} a_d\,\sqrt{d!}(\kerR_d(X)-\diag \kerR_d(X)),
\]
where $\|\tilde{\kerK}(X)\| \asto \|\mu_{a,b,\gamma}\|$.
\end{thm}
\begin{proof}
Observe that under Assumption \ref{a:Bernoulli}, the entries $X_{ij}$ satisfies the conditions of Theorem \ref{thm:main_1}. Furthermore, from Lemma \ref{lem:lemma_k_x_expansion} we can conclude that $\kerK(X)$ satisfies decomposition \eqref{e:decomposition}, with the coefficients $a_d$ satisfying
\[
|a_d|\leq (3A\sqrt{2e\pi})\sqrt{d+1}\al^d \; \mbox{for $\lceil t\alpha\rceil \le d \le p$ and $a_d=0$ for all $d>p$.}
\]
Observing that $\lceil t\alpha\rceil>\ell$, from Theorem \ref{thm:main_1}, the result follows.
\end{proof}
Finally, from Theorem \ref{thm:main}, we immediately get the following corollary.
\begin{cor} \label{cor:spectral_norm}
Adopt Assumption \ref{a:Bernoulli}.    Let us consider an analytic kernel $k(x)$ that satisfies \eqref{eq:kernel_bound}. If for all $0 \le k \le \ell-1$, $\mathbb E[k(\xi)H_k(\xi)]=0$, then as $n/p^\ell \rightarrow \gamma \in (0,\infty)$, we have
    \[
    \|\kerK(X)\| \asto \|\mu_{a,b,\gamma}\|.
    \]
\end{cor}

\subsection{Proof ideas and challenges} 
The two primary results presented in this paper are the precise asymptotic limits of the extremal eigenvalues of simplified tensor matrices (Theorem \ref{thm:tensor_matrix}) and the sharp asymptotic limit of the spectral norm of matrices which are linear combinations of simplified kernel matrices (Theorem \ref{thm:main_1}) under moment conditions on the entries. These two theorems help us to characterize the spectral norm of the matrices of the form \eqref{eq:def_kernel} for kernels $k(\cdot)$ satisfying certain regularity conditions.

To prove Theorem \ref{thm:main}, we establish upper bounds on the extremal eigenvalues of the simplified tensor matrices using a method of moments argument. A common approach in the existing literature (see, \cite{Fan2019}) for this type of problem involves computing the expectation of $\tr(\{\kerR_d(X) - \diag(\kerR_d(X))\}^L)$ for some $L \in \mathbb{N}$. Observing that this trace can be expressed as $\sum_{i=1}^{n} \lambda_i^L$, where $\lambda_1 \le \ldots \le \lambda_n$ are the eigenvalues of $\kerR_d(X) - \diag(\kerR_d(X))$ arranged in increasing order, for large $L$, the term corresponding to the largest eigenvalue, $\lambda_n^L$, should dominate the sum. Therefore, we anticipate that $\mathbb{E}[\tr(\{\kerR_d(X) - \diag(\kerR_d(X))\}^L)]^{1/L}$ will provide a good approximation of the largest eigenvalue of $\tr(\{\kerR_d(X) - \diag(\kerR_d(X))\}^L)$. 

However, a straightforward application of this argument fails to provide accurate approximations for the largest eigenvalues in many cases, particularly when $d < \ell$, and completely breaks down when approximating the smallest eigenvalue. \revsag{A key reason for this limitation is the dependence within the entries of the matrix $\kerR_d(X) - \diag(\kerR_d(X))$, which hampers the characterization of the extremal eigenvalues in our setup.} To address this, we employ a more sophisticated analysis involving non-backtracking walks on suitably defined bipartite cycles. For $d < \ell$, we first characterize the extremal eigenvalues of $\bar{\kerR}_d(X) - \diag(\bar{\kerR}_d(X))$ and then use the fact that the non-zero eigenvalues of $\bar{\kerR}_d(X)$ and $\kerR_d(X)$ are identical to determine both the largest and smallest eigenvalues of $\kerR_d(X)$. To analyze the eigenvalues of $\bar{\kerR}_d(X) - \diag(\bar{\kerR}_d(X))$, we introduce a version of a non-backtracking matrix $T_d(L)$ for all $d\leq \ell$ and $L\geq 1$ as below, where the $(i,j)$-th entry corresponds to a weighted count of bipartite cycles of length $2L$ between vertices $i$ and $j$, where $i,j \in \mathscr I_d$ (defined in \eqref{e:def_tensor_matrix}). The labels of the vertices in the cycle alternate between the sets $\mathscr I_d$ (row labels) and $\{1,\ldots,n\}$ (column labels). Additionally, these cycles must satisfy non-backtracking properties for both row and column labels.

\begin{defn} \label{def:non_backtracking_matrix}
For all $d\leq \ell$ and $L \in \N$, let us define the matrix $T_d(L) \in \R^{{p\choose d}\times {p\choose d}}$ elementwise as follows:

\[
[T_d(L)]_{ij}=\frac{(d!)^{L/2}}{\sqrt{n^L p^{dL}}} \sum_{\mathcal L(L)} \prod_{s=1}^d X_{u_1 k_s^{(0)}}X_{u_1 k_s^{(1)}} X_{u_2 k_s^{(1)}}X_{u_2 k_s^{(2)}} \dots X_{u_{L}, k_s^{(L-1)}}X_{u_L k_s^{(L)}}
\]

In the above definition, the summation $\sum_{\mathcal L(\ell)}$ is over all selections of integers $u_1, \dots, u_L\in [n]$ and tuples of integers $(k_1^{(0)}, \dots, k_d^{(0)})$, $\dots$, $(k_1^{(L)}, \dots, k_d^{(L)})\in [{p\choose d}]$ with $k_1^{(r)}< k_2^{(r)}< \dots< k_d^{(r)} \textrm{ for } 0\leq r \leq L$, such that $u_1\ne u_2, u_2\ne u_3 \dots, u_{L-1}\ne u_L$, 
\[
(k_1^{(0)}, \dots, k_d^{(0)}) = i, \quad (k_1^{(L)}, \dots, k_d^{(L)})=j,
\]
and for all $0 \le r \le L-1$
\[
(k_1^{(r)}, \dots, k_d^{(r)}) \neq (k_1^{(r+1)}, \dots, k_d^{(r+1)}).
\]
Observe here that the indices on the matrices correspond to $d$-tuples of numbers between 1 and $p$. By convention, we take $T_d(0)=I_{{p \choose d}}$ and $T_d(1)= \frac{\sqrt{d!}}{\sqrt{n p^{d}}}\left[(X^{(d)})^\top X^{(d)}-\mathsf{diag}((X^{(d)})^\top X^{(d)})\right]$.
\end{defn}

We derive sharp upper bounds on $\mathbb{E}[\tr(T_d(L)^m)]$ for large $m$, and use these to provide tight upper bounds on $\limsup_{n \to \infty} \|T_d(L)\|$. Together with Theorem \ref{thm:bulk}, these upper bounds, in turn, allow us to obtain sharp approximations of the extremal eigenvalues of $\bar{\kerR}_d(X) - \diag(\bar{\kerR}_d(X))$. When $d = \ell$, we apply a similar argument, using a non-backtracking matrix where the row labels are selected from $\{1,\ldots,n\}$ and the column labels from $\mathcal I_\ell$. In other words, these arguments help us to establish the first two assertions of Theorem \ref{thm:main}. It is noteworthy that such precise analysis of the extremal eigenvalues using non-backtracking cycles is not currently available in the literature when $n/p^\ell \rightarrow \gamma$ and $\ell>1$. Although a similar analysis was conducted for $\ell=1$ in \cite{bai_yin}, extending it to $\ell>1$ necessitates sophisticated combinatorial techniques that represent a novel contribution to the field. 

To characterize the extremal eigenvalues of $\kerR_d(X)$ for $d>\ell$, we
begin by characterizing the spectral norm of $\sum_{d=\ell}^{p}a_d\sqrt{d!}(\kerR_d(X)-\diag(\kerR_d(X)))$, where the constants $a_d$ decay at an exponential rate for all large $d$. Given any $\varepsilon>0$, we can choose $L_1(\varepsilon)$ large enough, such that if we truncate the sum $\sum_{d=\ell}^{p}a_d\sqrt{d!}(\kerR_d(X)-\diag(\kerR_d(X)))$ at $L_1$, the error introduced in estimating the spectral norm is at most a constant factor of $\varepsilon$. Therefore, it suffices to study the spectral norm of $B(X)=\sum_{d=\ell}^{L_1}a_d\sqrt{d!}(\kerR_d(X)-\diag(\kerR_d(X)))$. To study the spectral norm of $B(X)$, we use the method of moments by providing an upper bound for $\mathbb{E}[\tr(B(X)^{2L})]$. This quantity can be characterized by counting the number of bipartite cycles of length $2L$ where the labels of the vertices in the cycle alternate between the sets $\{1,\ldots,n\}$ (row labels) and $\{\mathscr I_d:\ell \le d \le L_1\}$ (column labels). Importantly, the length of the column labeling can vary between $\ell$ and $L_1$, distinguishing this graph from the previously studied graphs. Furthermore, these types of graphs with tuples of numbers as column labeling have not been studied in the literature. Furthermore, unlike earlier graphs with non-backtracking conditions for both row and column labels, these new graphs are constrained to satisfy non-backtracking conditions only for the row labels. 

Counting these special cycles allows us to derive a sharp upper bound for  $\limsup_{n \rightarrow \infty}\|B(X)\|$. Combined with Theorem \ref{thm:bulk}, this enables us to characterize the spectral norm of the original sum precisely. In particular, this argument leads to Theorem \ref{thm:main_1}. Furthermore, applying this theorem, with $a_d=1/\sqrt{d!}$ and $a_k=0$ for all $k \neq d$, immediately implies the third assertion in Theorem \ref{thm:main}.

It is crucial to highlight that our derivation of the spectral norm of $\sum_{d=\ell}^{p}a_d\sqrt{d!}(\kerR_d(X)-\diag(\kerR_d(X)))$ differs significantly from the analysis presented in \cite{Fan2019}. In the aforementioned paper, the authors expand the kernel into the Hermite basis and analyze the spectral norm of the truncated series. The remainder is shown to be of smaller order. However, this approach is not applicable when $\ell>1$ as the remainder is hard to control. \revsag{To address the case of polynomial scaling in $n$ and $p$, we have employed a more refined analysis that leverages the exponential decay of the coefficients $a_d$ and incorporates substantially different combinatorial arguments involving counting of non-backtracking walks on the carefully labeled bi-partite cycles as described above. Our techniques allow for the handling of short-range dependencies among the entries of random matrices while computing their moments. These techniques might be extended to handle similar problems involving random matrices with dependent entries naturally arising in several disciplines including time series analysis, econometrics, and statistical genomics.}

\subsection{Paper Organization} The paper is organized as follows: Section \ref{s:proof_main_1} provides the proof of Theorem \ref{thm:tensor_matrix}, which replies on estimates of non-backtracking matrices.  Section \ref{s:proof_main_2} outlines the major steps in proving Theorem \ref{thm:main_1}. The proof of Theorem \ref{thm:main_1} employs the moment methods, encoding the walks using multi-labelling graphs. The combinatorial properties of these mulit-labelling graphs are collected in Section \ref{tree}, with their proofs deferred to Section \ref{b_x} and Section \ref{sec:combo_row}. In Section \ref{a_x}, we examine the properties of non-backtracking matrices $T_d(L)$, which are used in the proof of Theorem \ref{thm:tensor_matrix}.
The combinatorial properties of non-backtracking walks are presented in Section \ref{sec:combo_col}, and their proofs are deferred to Section \ref{s:proof_comb_nonbacktracking}.
Finally,  Section \ref{proof_kernel_expansion} discusses the properties of analytic kernels and proves Lemma \ref{lem:lemma_k_x_expansion}.

\subsection{Notations}The set of natural numbers will be denoted by $\mathbb N$ and the set of real numbers by $\mathbb R$. We shall henceforth denote the set $\{1,\ldots,m\}$, for any $m \in \N$, by the notation $[m]$. The matrices will be denoted by uppercase letters and vectors by lower-case letters. The set $\R^{m \times n}$ refers to the set of $m \times n$ matrices and $\R^m$ refers to the set of $m$ dimensional vectors. For any vector $v \in \R^m$, we denote by $\|v\|$, the Euclidean norm of $v$. Furthermore, for any matrix $M \in \R^{m \times m}$, we denote its spectral norm by $\|M\|$ and its trace by $\mathrm{Tr}(M):=\sum_{i=1}^{m}M_{ii}$. For two sequences, $\{a_n\}$ and $\{b_n\}$, the notation $a_n \asymp b_n$ implies there exists constants $C_1,C_2>0$ such that $C_1a_n \le b_n \le C_2a_n$ for all $n$. For any natural number $n$, its factorial given by $n(n-1)(n-2)\cdots 1$ will be denoted by $n!$ and the $d$-th order falling factorial defined as $n(n-1)\cdots(n-d+1)$ will be denoted by by $(n)_d$. Finally, for any $n \in \N$, the double factorial is defined as $n!!:=\prod_{k=0}^{\lceil \frac{n}{2}\rceil-1}(n-2k)$.

\subsection*{Acknowledgement}
The research of J.H. was supported by NSF grant DMS-2331096, NSF CAREER grant DMS-2337795, and the Sloan Research Fellowship. S.N. extends his gratitude to Zongming Ma for his generous support during the course of this research. This project originated from the ``Undergraduate Research in Probability and Statistics" program at the University of Pennsylvania. The authors would like to express their deep gratitude for the opportunity provided by this program.

\section{Proof of Theorem \ref{thm:tensor_matrix}}
\label{s:proof_main_1}
We recall that $\kerR_d(X)=\frac{1}{\sqrt{np^d}}X_dX_d^\top$ and $\bar \kerR_d(X)=\frac{1}{\sqrt{np^d}}X_d^\top X_d$, which are related to the non-backtracking matrix $T_d(1)$ from \eqref{def:non_backtracking_matrix}
\begin{align}\label{e:Td1}
    T_d(1)=\frac{\sqrt{d!}}{\sqrt{np^d}}\times [X_d^\top X_d-\diag(X_d^\top X_d)]
\end{align}
The following two propositions give the estimates related to the non-backtracking matrix $T_d(1)$. We postpone their proofs to Section \ref{a_x}.
\begin{prop} \label{prop:Td_norm}
For all $d < \ell$, we almost surely have:
    \[ \limsup_{n\to\infty}\|T_d(1)\|\leq 2
    \]
\end{prop}
\begin{prop} \label{prop:Td_norm_d=l}
Suppose that $d = \ell$ and ${p\choose d}\leq n$. Then we almost surely have:
    \[ \limsup_{n\to\infty}\|T_d(1)-(\gamma d!)^{-\frac{1}{2}}I_{p\choose d}\|\leq 2
    \]
\end{prop}

To prove Theorem \ref{thm:tensor_matrix}, let us consider the following cases.
\paragraph{{\bf Case 1: $\bm{1 \le d < \ell}$.}} Using the Central Limit Theorem, we can conclude that as $n,p \rightarrow \infty$, 
\[
\left\|\frac{1}{\sqrt{np^d}}\diag X_d^\top X_d-I_{p\choose d}\sqrt{\frac{n}{p^d}}\right\|\xrightarrow{a.s.}0.
\]
By Proposition \ref{prop:Td_norm} and \eqref{e:Td1}, we have  \[\limsup_{n\to\infty} \left\|\frac{1}{\sqrt{np^d}}X_d^\top X_d-\frac{1}{\sqrt{np^d}}\diag X_d^\top X_d\right\|\leq \frac{2}{\sqrt{d!}} \textrm{ a.s.}
\]
Therefore 
\[
\limsup_{n\to\infty}\left\|\frac{1}{\sqrt{np^d}}X_d^\top X_d-I_{p\choose d}\sqrt{\frac{n}{p^d}}\right\|\leq \frac{2}{\sqrt{d!}} \textrm{ a.s.}
\]
Finally, using Weyl's inequalities \cite{Weyl1912}, we can almost surely conclude that for any given $\varepsilon>0$, there exists a $n_0(\varepsilon)$ such that for all $n \ge n_0(\varepsilon)$
\begin{equation}
\label{eq:eig_r_d_bar}
\sqrt{\frac{n}{p^d}}-\frac{2}{\sqrt{d!}}(1+\varepsilon) \le  \lambda_{\min}(\bar \kerR_d(X)) \le \lambda_{\max}(\bar \kerR_d(X))\le \sqrt{\frac{n}{p^d}}+\frac{2}{\sqrt{d!}}(1+\varepsilon). 
\end{equation}
Furthermore, the non-zero eigenvalues of $\bar \kerR_d(X)$ are equal to the non-zero eigenvalues of $\kerR_d(X)$. Since $\bar \kerR_d(X)$ has ${p\choose d}$ nonzero eigenvalues from \eqref{eq:eig_r_d_bar}, we can also almost surely conclude that if $k=n-{p\choose d}$ then for all $n \ge n_0(\varepsilon, d)$,
\[
\sqrt{\frac{n}{p^d}}-\frac{2}{\sqrt{d!}}(1+\varepsilon) \le  \lambda_{k+1}(\kerR_d(X)) \le \lambda_{\max}(\kerR_d(X))\le \sqrt{\frac{n}{p^d}}+\frac{2}{\sqrt{d!}}(1+\varepsilon). 
\]
\paragraph{{\bf Case 2: $\bm d=\bm \ell$.}}
Using Proposition \ref{prop:Td_norm_d=l}, for $n \ge {p \choose \ell}$, we get,
\[
\limsup_{n\to\infty}\left\|T_\ell(1)-(\gamma \ell!)^{-1/2}I_{p\choose \ell}\right\|\leq 2 \textrm{ a.s.}
\]
Thus we have,
\[
(\gamma \ell!)^{-1/2}-2\leq \liminf_{n\to\infty}\lambda_{\min} (T_\ell(1))\leq \limsup_{n\to\infty} \lambda_{\max} (T_\ell(1))\leq (\gamma \ell!)^{-1/2}+2 \textrm{ a.s.}
\]
Since $\bar \kerR_\ell(X) - \diag \bar \kerR_\ell(X) = \frac{T_\ell(1)}{\sqrt{\ell!}}$, we obtain:
\begin{align} \label{eq:Rdbounds}
\frac{1}{\sqrt \gamma \ell!}-\frac{2}{\sqrt{\ell!}}&\leq \liminf_{n\to\infty}\lambda_{\min} (\bar \kerR_\ell(X) - \diag \bar \kerR_\ell(X))\\
&\leq \limsup_{n\to\infty} \lambda_{\max} (\bar \kerR_\ell(X) - \diag \bar \kerR_\ell(X))\leq \frac{1}{\sqrt \gamma \ell!}+\frac{2}{\sqrt{\ell!}} \textrm{ a.s.}
\end{align}

Reversing the role of rows and columns in the multi-labeling of the non-backtracking matrix $T_\ell(1)$ from Definition \ref{def:non_backtracking_matrix}, and obtaining an analogue to Proposition \ref{prop:Td_norm_d=l}, we can similarly show that for $n \le {p \choose \ell}$, almost surely
\begin{align}
\label{eq:low_bound_supp_l}
\liminf_{n\to\infty}\lambda_{\min} ( \kerR_\ell(X) - \diag \kerR_\ell(X)) \ge \sqrt{\gamma}-\frac{2}{\sqrt{\ell!}}\nonumber\\
\limsup_{n\to\infty} \lambda_{\max} (\kerR_\ell(X) - \diag \kerR_\ell(X))\le \sqrt{\gamma}+\frac{2}{\sqrt{\ell!}}.
\end{align}
Furthermore, from Theorem \ref{thm:bulk}, the almost sure weak limit of the empirical spectral measure of $\kerR_\ell(X)-\diag \kerR_\ell(X)$ converges to $\frac{1}{\sqrt{\gamma} \ell !}(\mu_{mp, \gamma \ell!}-1)$ as defined in \eqref{eq:end_distrn} with support contained in \[\left[\sqrt{\gamma}-2/\sqrt{\ell!}, \sqrt{\gamma}+2/\sqrt{\ell!}\right].\] This almost surely implies, 
\begin{align}
\label{eq:up_bound_supp_l}
\limsup_{n\to\infty}\lambda_{\min} (\kerR_\ell(X) - \diag \kerR_\ell(X)) \le \sqrt{\gamma}-\frac{2}{\sqrt{\ell!}}\nonumber\\
\liminf_{n\to\infty} \lambda_{\max} (\kerR_\ell(X) - \diag \kerR_\ell(X)) \ge \sqrt{\gamma}+\frac{2}{\sqrt{\ell!}}
\end{align}
Therefore, from \eqref{eq:up_bound_supp_l} and \eqref{eq:low_bound_supp_l}, we can almost surely conclude that when $n\leq {p\choose \ell}$,
\begin{align}
\label{eq:lim_supp_l}
\lim_{n\to\infty}\lambda_{\min} (\kerR_\ell(X)-\diag \kerR_\ell(X))&= -\frac{2}{\sqrt{\ell!}}+\sqrt{\gamma},\nonumber\\ \lim_{n\to\infty} \lambda_{\max} (\kerR_\ell(X)-\diag \kerR_\ell(X))&= \frac{2}{\sqrt{\ell!}}+\sqrt{\gamma}.
\end{align}
Next, for $n \ge {p \choose \ell}$, by the Central Limit Theorem, we have
\begin{align}
\label{eq:asymp_diag}
\left\|\diag \bar\kerR_\ell(X)-\sqrt{\gamma}I_{p\choose \ell}\right\|\xrightarrow[]{a.s.}0, \quad \mbox{and} \quad \left\|\diag \kerR_\ell(X)-\frac{1}{\sqrt{\gamma} \ell!}I_{n}\right\|\xrightarrow[]{a.s.}0,\end{align} 
as $n,p \rightarrow \infty$.
Also, observe that the smallest $k=n-{p\choose \ell}$ many eigenvalues of $\kerR_\ell(X)$ are zero. From Theorem \ref{thm:bulk}, this observation implies that the almost sure weak limit of the empirical spectral measure of $\kerR_\ell(X)-\diag \kerR_\ell(X)$ converges to $\frac{1}{\sqrt{\gamma} \ell !}(\mu_{mp, \gamma \ell!}-1)$. Where $\mu_{mp, \gamma \ell!}$ is defined in \eqref{eq:end_distrn}. Note that since $n\geq {p\choose \ell}$, $\mu_{mp, \gamma \ell!}$ has a point mass $(1-\frac{1}{\gamma \ell!})\delta_0$ corresponding to the $n-{p\choose \ell}$ eigenvalues that are zero. Therefore,
\begin{align}
\label{eq:n_gret_p_l_sup}
 \limsup_{n\to\infty}\lambda_{(k+1)} ( \kerR_\ell(X) - \diag \kerR_\ell(X)) \le \sqrt{\gamma}-\frac{2}{\sqrt{\ell!}}.   
\end{align}
\begin{align}
\label{eq:n_larg_p_l_inf}
 \liminf_{n\to\infty}\lambda_{\max} ( \kerR_\ell(X) - \diag \kerR_\ell(X)) \ge \sqrt{\gamma}+\frac{2}{\sqrt{\ell!}}.   
\end{align}
Since, the set of non-zero eigenvalues of $\kerR_\ell(X)$ and $\bar \kerR_\ell(X)$ are equal, therefore
\begin{align}
\liminf_{n\to\infty}\lambda_{\min} (\bar \kerR_\ell(X)) = \liminf_{n\to\infty}\lambda_{(k+1)} (\kerR_\ell(X)).
\end{align}
Hence, using \eqref{eq:Rdbounds} and \eqref{eq:asymp_diag}, we can conclude 
\begin{align}
\label{eq:n_gret_p_l_inf}
\liminf_{n\to\infty}\lambda_{(k+1)} ( \kerR_\ell(X) - \diag \kerR_\ell(X)) \ge \sqrt{\gamma}-\frac{2}{\sqrt{\ell!}}.
\end{align}
Using \eqref{eq:n_gret_p_l_sup} and \eqref{eq:n_gret_p_l_inf}, we get that almost surely,
\[
\lim_{n\to\infty}\lambda_{(k+1)} ( \kerR_\ell(X) - \diag \kerR_\ell(X)) = \sqrt{\gamma}-\frac{2}{\sqrt{\ell!}}.
\]
Now, consider the largest eigenvalue of $\kerR_\ell(X)-\diag \kerR_\ell(X)$. From \eqref{eq:Rdbounds}, we can almost surely conclude that
\[
\limsup_{n\to\infty} \lambda_{\max} (\bar \kerR_\ell(X))\leq \frac{1}{\sqrt \gamma \ell!}+\frac{2}{\sqrt{\ell!}}+\sqrt{\gamma}.
\]
Again, since $\bar \kerR_\ell(X)$ and $\kerR_\ell(X)$ have the same set of nonzero eigenvalues, we can also conclude,
\begin{align*}
\limsup_{n\to\infty} \lambda_{\max} (\kerR_\ell(X))\leq \frac{1}{\sqrt{\gamma} \ell!}+\frac{2}{\sqrt{\ell!}}+\sqrt{\gamma} \quad\textrm{ a.s.}
\end{align*}
This implies
\[
\limsup_{n\to\infty} \lambda_{\max} (\kerR_\ell(X)-\diag \kerR_\ell(X))\leq \frac{2}{\sqrt{\ell!}}+\sqrt{\gamma}
\]
Finally, using \eqref{eq:n_larg_p_l_inf}, we get
\[
\lim_{n\to\infty} \lambda_{\max} (\kerR_\ell(X)-\diag \kerR_\ell(X))=\frac{2}{\sqrt{\ell!}}+\sqrt{\gamma}.
\]
\paragraph{{\bf Case 3: $\bm{\ell<d}$.}}
For this case, note that $\frac{1}{\sqrt{np^d}}X_dX_d^\top-\frac{1}{\sqrt{np^d}}\diag (X_dX_d^\top)=B(X)$, where $B(X)$ is defined in \eqref{eq:b_x} with $a_d=1/\sqrt{d!}$ and $a_k=0$ for all $k \neq d$. From Corollary \ref{cor:spectral_norm}, we know that $\|B(X)\|\xrightarrow[]{a.s.}\|\mu_{a, b, \gamma}\|$ with $a=0$, $b=1/d!$. This implies that almost surely
\[
-\frac{2}{\sqrt{d!}} \le \liminf_{n \rightarrow \infty}\lambda_{\min}(\kerR_d(X)-\diag \kerR_d(X)), \quad \mbox{and} \quad \limsup_{n \rightarrow \infty}\lambda_{\max}(\kerR_d(X)-\diag \kerR_d(X)) \le \frac{2}{\sqrt{d!}}.
\]

Furthermore, from Theorem \ref{thm:bulk}, we can conclude that the empirical spectral distribution of $\kerR_d(X)-\diag \kerR_d(X)$ converges to $\sqrt{1/d!}\mu_{sc}$, where $\mu_{sc}$ is the standard semicircle distribution defined in \eqref{eq:def_semicircle}. This in turn implies that almost surely,
\[
 \limsup_{n \rightarrow \infty}\lambda_{\min}(\kerR_d(X)-\diag \kerR_d(X)) \le -\frac{2}{\sqrt{d!}}, \quad \mbox{and} \quad \liminf_{n \rightarrow \infty}\lambda_{\max}(\kerR_d(X)-\diag \kerR_d(X)) \ge \frac{2}{\sqrt{d!}}
\]
Combining the above equations, we get that almost surely
\[
\lim_{n \rightarrow \infty}\lambda_{\min}(\kerR_d(X)-\diag \kerR_d(X)) = -\frac{2}{\sqrt{d!}}, \quad \mbox{and} \quad \lim_{n \rightarrow \infty}\lambda_{\max}(\kerR_d(X)-\diag \kerR_d(X)) = \frac{2}{\sqrt{d!}}
\]

\section{Outline of the proof of Theorem \ref{thm:main_1}}
\label{s:proof_main_2}
For any fixed $\varepsilon>0$, consider $L_2>0$ large enough such that
$
\sum_{d \ge L_2}\sqrt{d+1}e^{-d} \le \varepsilon.
$
Taking $L_1=\max\{L_2,\tilde L\}$, by Lemma \ref{lem:lemma_k_x_expansion}, we can decompose $\kerK(X)$ into matrices $\kerA(X), \kerB(X)$ and $\kerC(X)$, as follows:
\[
[\kerA(X)]_{ij}= \begin{cases}\sum\limits_{d=1}^{\ell-1}\frac{a_d}{\sqrt{np^dd!}}\sum\limits_{\substack{k_1\neq k_2\neq\ldots\neq k_d\\k_1,\ldots,k_d \in [p]}}(X_{ik_1}X_{ik_2}\cdots X_{ik_d})(X_{jk_1}X_{jk_2}\cdots X_{jk_d}), & \mbox{if $i \neq j$}\\
        0, & \mbox{otherwise,}
        \end{cases}
\]
\begin{align}
\label{eq:b_x}
[\kerB(X)]_{ij}= \begin{cases}\sum\limits_{d=\ell}^{L_1}\frac{a_d}{\sqrt{np^dd!}}\sum\limits_{\substack{k_1\neq k_2\neq\ldots\neq k_d\\k_1,\ldots,k_d \in [p]}}(X_{ik_1}X_{ik_2}\cdots X_{ik_d})(X_{jk_1}X_{jk_2}\cdots X_{jk_d}), & \mbox{if $i \neq j$}\\
        0, & \mbox{otherwise,}
        \end{cases}
\end{align}
and
\[
[\kerC(X)]_{ij}= \begin{cases}\sum\limits_{d=L_1+1}^{\infty}\frac{a_d}{\sqrt{np^dd!}}\sum\limits_{\substack{k_1\neq k_2\neq\ldots\neq k_d\\k_1,\ldots,k_d \in [p]}}(X_{ik_1}X_{ik_2}\cdots X_{ik_d})(X_{jk_1}X_{jk_2}\cdots X_{jk_d}), & \mbox{if $i \neq j$}\\
        0, & \mbox{otherwise.}
        \end{cases}
\]
The spectrums of these three components behave differently. Observe that, from Theorem \ref{thm:tensor_matrix}, the eigenvalues of $\kerA(X)$ are typically diverging. In particular, for any given $1 \le d <\ell$ and $\varepsilon>0$, there exists a $n_0(\varepsilon,d)$ such that for all $n \ge n_0(\varepsilon,d)$
\begin{equation*}
\sqrt{\frac{n}{p^d}}-\frac{2}{\sqrt{d!}}(1+\varepsilon) \le  \lambda_{\min}(\kerR_d(X)) \le \lambda_{\max}(\kerR_d(X))\le \sqrt{\frac{n}{p^d}}+\frac{2}{\sqrt{d!}}(1+\varepsilon), 
\end{equation*}
almost surely.  Therefore, it suffices to study the spectrum of $\kerB(X)$ and $\kerC(X)$, as $n,p \rightarrow \infty$.

In this section we prove the following two theorems which characterize the asymptotic behavior of the spectral norm of $\kerB(X)$ and $\kerC(X)$ as in \eqref{eq:b_x}. Their proofs are given in Section \ref{s:KerB} and Section \ref{s:KerC} respectively. Theorem \ref{thm:main_1} follows from them easily. 
\begin{thm}
\label{thm:spctrum_b}
\revsag{Consider the matrix $\kerB(X)$ defined in \eqref{eq:b_x} where the entries of $X$ satisfy \eqref{eq:master_assumption}. If $n/p^\ell \rightarrow \gamma \in (0,\infty)$ then, we almost surely have
    \[
    \lim_{n \rightarrow \infty}\|\kerB(X)\| \xrightarrow{a.s.} \|\mu_{a,b,\gamma}\|.
    \]}
\end{thm}

\begin{thm}
    \label{thm:specturm_of_remainder}
    Consider the matrix $\kerC(X)$ defined in \eqref{eq:b_x} obtained from $\kerK(X)$ satisfying \eqref{e:decomposition} with $|a_d| \le C_3\sqrt{d+1}\alpha^d$ where $C_3>0$, $0<\alpha \le e^{-C(\ell)}$ for an absolute constant $C(\ell)>0$ and $d \ge L_1$ ($L_1$ is defined before \eqref{eq:b_x}).  Furthermore, let the entries of $X$ satisfy \eqref{eq:master_assumption}. If $n/p^\ell \rightarrow \gamma \in (0,\infty)$, then
    \[
    \limsup_{n \rightarrow \infty}\|\kerC(X)\| \le 4\cdot \max\{1,\gamma^{-1/2}\}\varepsilon,
    \]
    almost surely.
\end{thm}

\begin{proof}[Proof of Theorem \ref{thm:main_1}]
For any fixed $\varepsilon>0$, consider $L_2>0$ large enough such that
$
\sum_{d \ge L_2}\sqrt{d+1}e^{-d} \le \varepsilon.
$
Taking $L_1=\max\{L_2,\tilde L\}$, the matrix $\kerK(X)$ adopts the decomposition \eqref{eq:b_x}. Then taking $\varepsilon \rightarrow 0$, the theorem follows using Theorems \ref{thm:spctrum_b} and \ref{thm:specturm_of_remainder}.
\end{proof}

\subsection{Multi-labelling and simple-labelling graphs}
\label{sec:spec_b}
To prove Theorem \ref{thm:spctrum_b} and \ref{thm:specturm_of_remainder}, let us first define the following combinatorial concepts.

\begin{figure}[t!]
    \centering
    \includegraphics[width=0.7\textwidth]{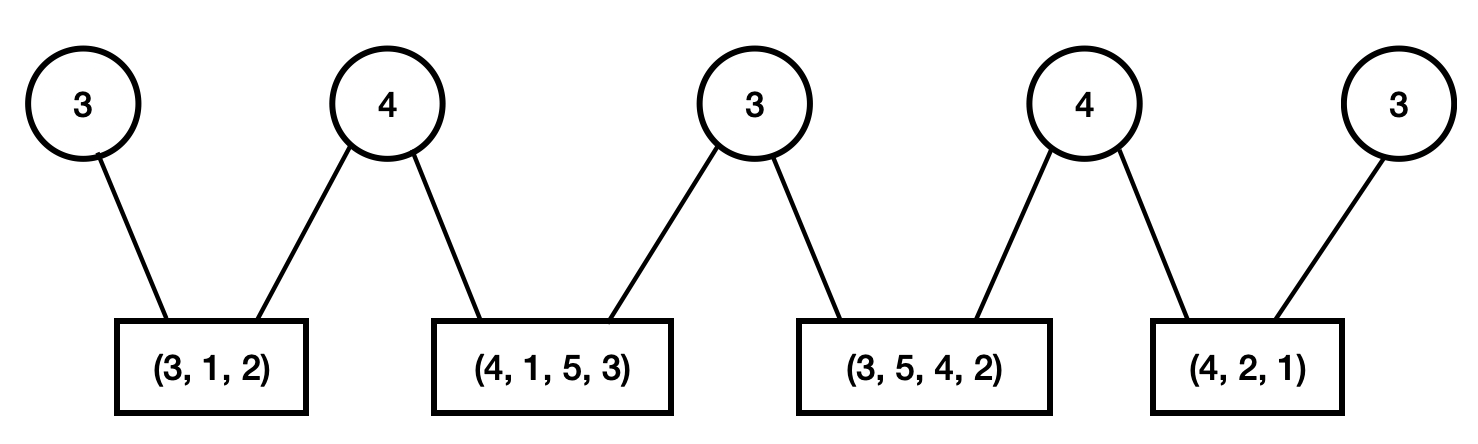}
    \caption{Illustration of an $4$-{\bf multi-labelling} of a $4$-graph with $\ell=3$. Here $d_1=3,d_2=4$, $d_3=4$ and $d_4=3$. The excess of this {\bf multi-labelling} is $5/3$.}
    \label{fig:example_image}
\end{figure}

\begin{defn}
\label{def:l_graph}
    For any integer $L\geq 2$, an \textbf{$L$-graph} is a graph consisting of a single cycle with $2L$ vertices and $2L$ edges with the vertices alternatingly denoted as \textbf{$n$-vertices} and \textbf{$p$-vertices}.
\end{defn}

\begin{defn} \label{def:multi-labelling}
    An $L_1$-\textbf{multi-labelling} of an $L$-graph is an assignment of labels $(i_1,\ldots,i_L) \in [n]^L$ for the $n$-vertices and an assignment of tuples of size between $\ell$ and $L_1$ for each of the $p$-vertices where the elements of each tuple belong to $[p]$. Further, the $n$ and the $p$ vertices satisfy the following conditions:
    \begin{itemize}
        \item[$(i)$] The $n$-label of each $n$-vertex is distinct from those of the two $n$-vertices immediately preceding and following it in the cycle.
        \item[$(ii)$] The number $d_r$ of $p$-labels in the tuple for the $r$th $p$-vertex satisfies $\ell\leq d_r\leq L_1$, and these $d_r$ $p$-labels are distinct.
        \item[$(iii)$] For each distinct $n$-label $i$ and distinct $p$-label $j$ there exists either $0$ or at least two edges in the cycle such that edge $n$-vertex endpoint is labelled $i$ and its $p$-vertex endpoint has the label $j$ in its tuple.
    \end{itemize}
A \textbf{$(p, n, L_1)$-multi-labelling} is a $L_1$-\textbf{multi-labelling} of a $L$ graph where all the $p$-labels are in $[p]$ and all the $n$-labels are in $[n]$.
\end{defn}

Next, let us consider the notion of isomorphic graphs referring to a pair of graphs that are the same up to the assignment of labels to the $n$ and $p$ vertices.

\begin{defn} \label{def:isomorphic}
    Consider two $L$-graphs $\mathcal G_1$ and $\mathcal G_2$ with \textbf{$(p, n, L_1)$-multi-labelling} $\mathcal L_1$ and $\mathcal L_2$, respectively. We say that these graphs are $\bf{isomorphic}$ if there exist permutations $\pi_n$ of $\{1, \dots, n\}$ and $\pi_p$ of $\{1, \dots, p\}$ so that if we apply $\pi_n$ to the $n$-labels and $\pi_p$ to the $p$-labels of $\mathcal L_1$ we obtain $\mathcal L_2$.
\end{defn}

\begin{defn}
    An $\bf{isomorphism\ class}$ of $L$-graphs is the largest set $S$ of $L$-graphs which are pairwise isomorphic to each other. Note that isomorphism is an equivalence relation.
\end{defn}

\begin{defn} \label{def:delta}
    Consider an $L$-graph $\mathcal G$ with a specified \textbf{$(p, n, L_1)$-multi-labelling} $\mathcal L$. Assume that $\mathcal L$ assigns $p$-labels of sizes $d_1, \dots, d_L$ respectively to the $p$-vertices of the $L$-graph. Additionally, let $r(\mathcal L)$ and $c(\mathcal L)$ denote the number of total distinct $n$-labels and $p$-labels, respectively, in $\mathcal L$. Then define the \emph{excess} of $\mathcal L$ to be  
    \[
    \Delta(\mathcal L):= 1+\frac{L}{2}+\sum_{i=1}^L \frac{d_i}{2\ell}-r(\mathcal L)-\frac{c(\mathcal L)}{\ell}.
    \]
\end{defn}
The dependence of the excess on the \textbf{$(p, n, L_1)$-multi-labelling} $\mathcal L$, will sometimes be suppressed to make the notation less cumbersome.

The isomorphism class of an $L$-graph is uniquely determined by specifying which $n$ and $p$ labels of the vertices in the graph are the same. From Definition \ref{def:isomorphic} (assuming we have selected $d_1, \dots, d_L$), it is easy to observe that the number of graphs in a single isomorphism class is $n(n-1)\dots (n-r(\mathcal L))p(p-1)\dots (p-c(\mathcal L))\asymp n^{r(\mathcal L)}p^{c(\mathcal L)}$, since a single isomorphism class fixes all properties of the graph except the indices used for the labels of the vertices of the graph. Further, observe that any two \textbf{$(p, n, L_1)$-multi-labelling} of $L$-graphs $\mathcal L$ and $\mathcal L'$ in the same isomorphism class satisfy $\Delta(\mathcal L)=\Delta(\mathcal L')$. 

The following lemma states that $\Delta(\cL)\geq 0$. We postpone it to Section \ref{s:B3}.
\begin{lem}
    \label{lem:row_positive_excess}
    For a \textbf{$(p, n, L_1)$-multi-labelling} $\mathcal L$ of an $L$-graph we have
    $
    \Delta(\mathcal L) \ge 0.
    $
\end{lem}
We can characterize the trace of the $L$-th power of $\kerB(X)$ in terms of $\Delta(\mathcal L)$ as follows:

\begin{thm}
\label{thm:upper_bound_trace}
Consider the matrix $\kerB(X)$ defined in \eqref{eq:b_x} in the setting where the entries of $X$ satisfy \eqref{eq:master_assumption}. Then for all $L \ge 1$, if $d_1(\mathcal M),\ldots,d_L(\mathcal M)$ are the number of $p$-labels on the $p$-vertices of the $L$ graph, then we have
    \begin{equation}
    \label{eq:sum_multilabeling_1}
        \mathbb E[\mathrm{Tr}(\kerB(X)^L)] \le \frac{p^\ell}{\gamma^{L/2}}\cdot  \sum_{\mathcal L\in \mathcal M} \left(\frac{\max\{1, (405\ell\Delta(\mathcal L))^{405\beta}\}}{p}\right)^{\ell\Delta(\mathcal L)}\left(\frac{n}{p^\ell}\right)^{r(\mathcal L)}\left(\prod_{s=1}^{L}\frac{|a_{d_s(\mathcal L)}|}{(d_s(\mathcal L)!)^{1/2}}\right),\nonumber
    \end{equation}
    where $\mathcal M$ is the set of isomorphism classes of $(p, n, L_1)$-multi-labellings of $L$ graphs.
\end{thm}

\begin{proof}
     For notational convenience let $i_{L+1}:=i_1$,
\begin{align*}
    \mathbb E [\tr(\kerB(X)^L)] &= \sum_{\substack{i_1, \dots, i_L=1\\ i_1\ne i_2, i_2\ne i_3, \dots, i_L\ne i_1}}^n \mathbb E\left[\prod_{s=1}^L b_{i_s i_{s+1}}\right] \\ &= \sum_{\substack{i_1, \dots, i_L=1\\ i_1\ne i_2, i_2\ne i_3, \dots, i_L\ne i_1}}^n \frac{p^{-\frac{\ell L}{2}}}{\gamma^{L/2}}\mathbb E\left[\prod_{s=1}^L \left(\sum_{d_s=\ell}^{L_1}a_{d_s}\sqrt{\frac{1}{p^{d_s} {d_s}!}}\sum_{\substack{j_1, \dots, j_{d_s}=1\\ j_1\ne j_2\ne, \dots, \ne j_{d_s}}}\prod_{a=1}^{d_s} X_{i_s j_a}X_{i_{s+1} j_a}\right)\right]\\
&= \frac{1}{\gamma^{L/2}}\sum_{d_1, \dots, d_L=l}^{L_1}\sum_{\substack{i_1, \dots, i_L=1 \\
    i_1\ne i_2, i_2\ne i_3, \dots, i_{L-1}\ne i_L}}^n \sum_{\substack{j_1^{1}, \dots, j_{d_1}^{1} = 1\\
    j_1^{1}\ne j_2^{1}\ne\dots\ne j_{d_1}^1}}^p \dots \sum_{\substack{j_1^{L}, \dots, j_{d_L}^L = 1\\
    j_1^{L}\ne j_2^{L}\ne\dots\ne j_{d_L}^{L}}}^p \\
    &\qquad \qquad p^{-\frac{\ell L+\sum_{s=1}^L d_s}{2}} \left(\prod_{s=1}^L\frac{a_{d_s}}{(d_s!)^{1/2}}\right) \mathbb E\left [\prod_{s=1}^L\prod_{a=1}^{d_s}X_{i_sj_a^s}X_{i_{s+1}j_a^s} \right]
\end{align*}
    
    Following a similar argument to Lemma 5.9 in \cite{Fan2019} we conclude that the sum above is equivalent to 
    \begin{align} \label{eq:sum_Lgraph_1}
        \frac{1}{\gamma^{L/2}}\sum_{\mathcal L\in \mathcal M} p^{-\frac{\ell L+\sum_{s=1}^L d_s(\mathcal L)}{2}} \left(\prod_{s=1}^L\frac{|a_{d_s(\mathcal L)
        }|}{(d_s(\mathcal L))!)^{1/2}}\right) \mathbb E\left [\prod_{s=1}^L\prod_{a=1}^{d_s}X_{i_sj_a^s}X_{i_{s+1}  j_a^s} \right],
    \end{align} 
    where $\mathcal M$ is the collection of all $(p,n,L_1)$ multilabellings of $L$ graphs. Let $b_{ij}$ be the number of times $X_{ij}$ appears in the product $\prod_{s=1}^\ell\prod_{a=1}^{d_s}X_{i_sj_a^s}X_{i_{s+1}  j_a^s}$. Furthermore, since $\mathbb E[X^2_{ij}]=1$ and $\mathbb E[|X_{ij}|^{k}] \le k^{\alpha k}$, and the entries $X_{ij}$ are independent, for any $(p,n,L_1)$ multilabellings $\mathcal L$ of an $L$ graph, we have,
    \begin{align}
        \mathbb E\left[\prod_{s=1}^L\prod_{a=1}^{d_s}X_{i_sj_a^s}X_{i_{s+1}  j_a^s}\right] = \prod\limits_{i,j:b_{ij}>2}\mathbb E[X^{b_{ij}}_{ij}] &\le \prod\limits_{i,j:b_{ij}>2}b^{\beta b_{ij}}_{ij}\nonumber\\
        &\le \left(\sum_{i,j:b_{ij}>2}b_{ij}\right)^{\beta\sum_{i,j:b_{ij}>2}b_{ij}} \nonumber\\
        &\le \max\left\{1,(405 \Delta(\mathcal L) \ell)^{405\beta \Delta(\mathcal L) \ell}\right\}.
    \end{align}
    Here, the last line follows using the following lemma. Its proof involves delicate combinatorics, so we postpone the  proof to Section \ref{tree}.
    \begin{lem}\label{l:bsum}
        The following holds: 
        \begin{align}
            \sum_{i,j:b_{ij}>2}b_{ij}\leq 405 \Delta(\mathcal L) \ell.
        \end{align}
    \end{lem}
    The number of graph multi-labellings belonging a single equivalence class is upper bounded by $n^{r(\mathcal L)}p^{c(\mathcal L)}$ where $r(\mathcal L)$ and $c(\mathcal L)$ denote the number of distinct $n$-labels and $p$-labels, respectively, of a multi-labeling $\mathcal L$. Further, let $d_s(\mathcal L)$ denote the number of $p$-labels of the $s$th $p$-vertex of $\mathcal L$. Since multi-labellings within the same equivalence class give the same value of $\Delta$ we can upper bound \eqref{eq:sum_Lgraph_1} by: 
    \begin{align*}
        &\frac{1}{\gamma^{L/2}}\sum_{\mathcal L\in \mathcal M} p^{-\frac{\ell L+\sum_{s=1}^L d_s(\mathcal L)}{2}} \left(\prod_{s=1}^L\frac{|a_{d_s(\mathcal L)
        }|}{(d_s(\mathcal L))!)^{1/2}}\right)\times \max\left\{1,(405 \Delta(\mathcal L) \ell)^{405\beta \Delta(\mathcal L) \ell}\right\}\\
        &\leq \frac{p^\ell}{\gamma^{L/2}}\cdot  \sum_{\mathcal L\in \mathcal M} \left(\frac{\max\{1, 405\ell\Delta(\mathcal L))^{405\beta}\}}{p}\right)^{\ell\Delta(\mathcal L)} \left(\frac{n}{p^\ell}\right)^{r(\mathcal L)}\left(\prod_{s=1}^L\frac{|a_{d_s(\mathcal L)
        }|}{(d_s(\mathcal L))!)^{1/2}}\right).
    \end{align*}
    The above observation gives us the theorem.
\end{proof}


Next, we consider a simple labelling of a $L$-graph.

\begin{defn} \label{def:simple_labelling}
    An $(p, n, \ell)$-\textbf{simple-labelling} of an $L$-graph is an assignment of an $n$-label in $[n]$ to each $n$-vertex and either a $\ell$-tuple with elements in $[p]$ or an empty label $\emptyset$ to each $p$-vertex such that the following properties hold:
    \begin{itemize}
        \item[$(i)$] $n$-labels of adjacent $n$-vertices are distinct
        \item[$(ii)$] $p$-labels in a $\ell$-tuple are distinct
        \item[$(iii)$] There are an even number of edges $(i, j)$ between each $n$-label $i$ and $p$-label $j$.
        \item[$(iv)$] For any two consecutive $n$-labels $i$ and $i'$, the number of occurrences of three consecutive vertices with labels $i, \emptyset, i'$ equals the number of occurrences of three consecutive vertices with labels $i', \emptyset, i$.
    \end{itemize}
\end{defn}

Two $(p, n, \ell)$ simple-labellings of an $L$-graph are \textbf{equivalent} if there is a permutation $\pi_p$ of $\{1,2,\ldots,p\}$ and a permutation $\pi_n$ of $\{1,2,\ldots,n\}$ such that one labelling is the image of the other upon applying $\pi_p$ to all of its $p$-labels and $\pi_n$ to all of its $n$-labels. (The empty $p$-label remains empty under any such permutation $\pi_p$.) For any fixed $L$, the equivalence classes under this relation will be called \textbf{simple-labelling equivalence classes}. 

\begin{defn} \label{def:canonical_simple}
    The \textbf{canonical simple-labelling} in a simple labelling equivalence class $\tilde{\mathcal C}$ is the one in which the label of $i$th new $n$-vertex that appears in the cyclic traversal is $i$, and the label of $j$th new non-empty $n$-vertex label is $j$. The \textbf{canonical multi-labelling} in a multi-labelling equivalence class $\mathcal C$ is the one in which the label of each $i$th new $n$-vertex is $i$ and the label of $j$th new $p$-vertex label is $j$. Note that the new $p$-vertex labels in the label-tuple for each $p$-vertex appear in sorted order.
\end{defn}

Consider a $(p, n, \ell)$-simple-labeling $\tilde M$ of an $L$-graph and let $\tilde k$ denote the number of $p$-vertices with non-empty labels. Then define the \textbf{excess} of the simple-labeling to be:
    \begin{align}
    \label{eq:excess_simple_label}
    \tilde \Delta(\tilde M) := \frac{L + \tilde k}{2}+1-\tilde r -\frac{\tilde c}{\ell}
    \end{align}

Here, $\tilde r$ is the number of distinct $n$-labels and $\tilde c$ is the number of distinct $p$-labels in the simple labeling. The following lemma states that the above-defined \textbf{excess} is always positive. 

The following lemma states that $ \tilde \Delta(\tilde M) \geq 0$. We postpone it to Section \ref{s:B3}.
\begin{lem} \label{lem:simple_positive_excess}
     Consider a $(p, n, \ell)$-simple-labelling of an $L$-graph. Then $ \tilde \Delta(\tilde M) \geq 0$.
\end{lem}

The proof of this lemma follows using Lemma 5.13 of \cite{Fan2019}. Next, consider the following lemma providing a lower bound to $\mathbb E[\mathrm{Tr}(M^L)]$ for $M$ is defined in \eqref{eq:bulk_matrix} in terms of the simple labelings of an $L$-graph.

\begin{lem} \label{lem:bound_trace_M}
    Let $M$ be the matrix defined in \eqref{eq:bulk_matrix} and let $\tilde{\mathcal C}$ denote the set of equivalence classes $(p, n, \ell)$-simple-labellings, with $L\ge 2$ and even. For each $\tilde{\mathcal M}\in \tilde{\mathcal C}$, let $\tilde \Delta(\tilde{\mathcal M})$ denote the excess of $\tilde{\mathcal M}$, $\tilde k(\tilde{\mathcal M})$ denote the number of $p$-vertices with non-empty label, and $\tilde r(\tilde{\mathcal M})$ be the number of distinct $n$-labels. Then,
    \begin{align}
    \mathbb E[\tr(M^{L})]\geq \frac{p^\ell}{\gamma^{L/2}} \left(\frac{p  - L}{p}\right)^{L} \left(\frac{n - L}{n}\right)^{L} \sum_{\tilde{\mathcal M}\in \tilde{\mathcal C}} \left(\frac{1}{ p^\ell}\right)^{\tilde \Delta(\tilde{\mathcal M})} \left(\frac{n}{p^\ell}\right)^{r(\tilde{\mathcal M})-\frac{L-\tilde k(\tilde{\mathcal M})}{2}} \frac{|a|^{\tilde k(\tilde{\mathcal M})}}{(\ell!)^{\frac{\tilde k(\tilde{\mathcal M})}{2}}} \left(\gamma b\right)^{\frac{L - \tilde k(\tilde{\mathcal M})}{2}} \nonumber
    \end{align}
\end{lem}
The proof of this theorem follows by some minor modifications to Lemma 5.16 of \cite{Fan2019}. Hence we omit this proof for the sake of avoiding repetition.

We shall show by a careful combinatorial argument in Section \ref{tree} that in the sum described in \eqref{eq:sum_multilabeling_1}, it is enough to consider the multi-labelings with excess $0$. Next, let us construct a map $\phi:\mathcal L(p, n, L_1)\to \tilde {\mathcal L}(p, n, L_1)$ (i.e. from $(p, n, \ell)$-multi-labellings with excess $0$ to $(p, n, L_1)$-simple-labellings with excess $0$) as follows: 
    \begin{itemize}
        \item[(1)] It maps each $n$-vertex in $\mathcal M$  to the same $n$-vertex in $\tilde{\mathcal M}=\phi(\mathcal M)$
        \item[(2)] It maps  each $p$-vertex in $\mathcal M$ with exactly $\ell$ $p$-labels, to a $p$-vertex with the same set of $\ell$ labels in $\tilde{\mathcal M}$.
        \item[(3)] It maps each $p$-vertex in $\mathcal M$ with more than $\ell$ $p$-labels to a $p$-vertex with an empty label in $\tilde{\mathcal M}$
    \end{itemize}
With the notations above, the following proposition holds. Its proof uses the results from Section \ref{tree}, so we postpone it to Section \ref{s:B3}.
\begin{prop} \label{prop:invert_mapping_2}
    Suppose $a_\ell\ne 0$ and $L \geq 2$. Let $\mathcal C_0$ and $\tilde{\mathcal C_0}$ denote the set of all $(p,n,L_1)$-multi-labelling and $(p,n,L_1)$-simple labelling equivalence classes, respectively, such that for all $\mathcal M\in \mathcal C_0$ and all $\tilde{\mathcal M}\in \tilde {\mathcal C_0}$ we have $\Delta(\mathcal M)=\tilde \Delta(\tilde {\mathcal M})=0$. For $\mathcal M \in \mathcal C$, let $r(\mathcal M)$ be the number of distinct $n$-labels and $d_s(\mathcal M)$ denote the number of $p$-labels in the $s$th $p$-vertex of the $L$-graph. Next, for $\tilde{\mathcal M}\in \tilde{\mathcal C_0}$, let $\tilde \Delta(\tilde{\mathcal M})$ be its excess, $\tilde r(\tilde{\mathcal M})$ be the number of distinct $n$-labels, and $\tilde k(\tilde{\mathcal M})$ be the number of $p$-vertices with non-empty label. Then there exists a map $\phi:\mathcal C_0\to\tilde{\mathcal C_0}$ constructed as above, such that,
    \begin{itemize}
        \item[$(1)$] For all $\mathcal M\in \mathcal C_0$, $r(\mathcal M) = \tilde r (\tilde{\mathcal M})$,
        \item[$(3)$] For any $\tilde {\mathcal M}\in \tilde{\mathcal C_0}$,
        \[
        \sum_{\mathcal M\in\phi^{-1}(\tilde{\mathcal M})}\prod_{s=1}^L\frac{|a_{d_s}(\mathcal M)|}{(d_s(\mathcal M)!)^{1/2}} \le\left(\frac{|a_\ell|}{(\ell!)^{1/2}}\right)^{\tilde k(\tilde {\mathcal M})}b^{\frac{L-\tilde k(\tilde {\mathcal M})}{2}}
        \]
    \end{itemize}
\end{prop}

\subsection{Analysis of the spectrum of $\kerB(X)$}
\label{s:KerB}

\begin{proof}[Proof of Theorem \ref{thm:spctrum_b}]
Since $\kerA(X)$ is of rank bounded by $C\cdot p^{\ell-1}$ for a fixed constant $C$, the random matrices $\kerB(X)+\kerC(X)$ and $\kerK(X)$ have the same limiting spectral distribution. Consequently, using Theorem \ref{thm:bulk}, we can conclude that
\[
\liminf_{n \rightarrow \infty}\|\kerB(X)+\kerC(X)\| \ge \|\mu_{a,b,\gamma}\|.
\]
Furthermore, using Theorem \ref{thm:specturm_of_remainder}, we can also conclude
\[
\liminf_{n \rightarrow \infty} \|\kerB(X)\| \ge \|\mu_{a,b,\gamma}\|,
\]
where $\|\mu_{a,b,\gamma}\|$ refers to the edge of the spectrum of $\mu_{a,b,\gamma}$ as defined in \eqref{eq:edge_of_spectrum}. Now let us take any $\varepsilon>0$ and an even integer $L\geq 2$. Then,
\[
\mathbb P[\|\kerB(X)\|>(1+\varepsilon)\|\mu_{a, b, \gamma}\|]\leq \mathbb P[\tr(\kerB(X)^L)>(1+\varepsilon)^L\|\mu_{a, b, \gamma}\|^L] \leq \frac{\mathbb E[\tr(\kerB(X)^L)]}{(1+\varepsilon)^L\|\mu_{a, b, \gamma}\|^L}
\]
Using Theorem \ref{thm:upper_bound_trace}, we get
\[
\mathbb E[\mathrm{Tr}(\kerB(X)^L)] \le \frac{p^\ell}{\gamma^{L/2}}\cdot  \sum_{\mathcal L\in \mathcal M} \left(\frac{\max\{1, (405\ell\Delta(\mathcal L))^{405\beta}\}}{p}\right)^{\ell\Delta(\mathcal L)}\left(\frac{n}{p^\ell}\right)^{r(\mathcal L)}\left(\prod_{s=1}^{L}\frac{|a_{d_s(\mathcal L)}|}{(d_s(\mathcal L)!)^{1/2}}\right).
\]
Now using the trivial bound $\Delta(\mathcal L)\leq \frac{L(1+L_1)}{2}$ we have,
\begin{align}
&\mathbb E[\mathrm{Tr}(\kerB(X)^L)] \nonumber\\
&\le \frac{p^\ell}{\gamma^{L/2}}\sum_{\Delta_0=0}^{\frac{L(1+L_1)}{2}}\sum_{\substack{\mathcal L \in \mathcal M\\\Delta(\mathcal L)=\Delta_0}} \left(\frac{\max\{1, (405\ell\Delta(\mathcal L))^{405\beta}\}}{p}\right)^{\ell\Delta(\mathcal L)}\left(\frac{n}{p^\ell}\right)^{r(\mathcal L)}\left(\prod_{s=1}^{L}\frac{|a_{d_s}(\mathcal L)|}{(d_s(\mathcal L)!)^{1/2}}\right) \nonumber
\end{align}
Now, we shall split the sum into terms over $\Delta(\mathcal L)=0$ and $\Delta(\mathcal L)>0$. For $\Delta(\mathcal L)=0$ terms we use Proposition \ref{prop:invert_mapping_2} to get
\begin{align*}
\mathcal S_0 = \frac{1}{\gamma^{L/2}}\sum_{\substack{\mathcal L \in \mathcal M\\\Delta(\mathcal L)=0}} p^\ell\left(\frac{n}{p^\ell}\right)^{r(\mathcal L)}\left(\prod_{s=1}^{L}\frac{|a_{d_s}(\mathcal L)|}{(d_s(\mathcal M)!)^{1/2}}\right) \le  \frac{p^\ell}{\gamma^{L/2}}\sum_{\substack{\tilde{\mathcal L}\in \tilde{\mathcal M}\\ \tilde \Delta(\tilde{\mathcal L}) = 0}}\left(\frac{n}{p^\ell}\right)^{\tilde r(\tilde{\mathcal L})} \frac{|a|^{\tilde k(\tilde{\mathcal L})}}{(\ell !)^{\tilde k(\tilde{\mathcal L}) / 2}}b^{\frac{L-\tilde k(\tilde{\mathcal L})}{2}}
\end{align*}
As long as $L=o(p)$, for sufficiently large $n, p$, we have $\left(\frac{p-L}{p}\right)\left(\frac{n-L}{n}\right)>1-\varepsilon/4$ for all $\varepsilon>0$. So, by Lemma \ref{lem:bound_trace_M}, we can conclude $(1-\frac{\varepsilon}{4})^L\mathcal S_0\leq \mathbb E[\tr(M^L)]$.

Now, we consider the sum of the terms with $\Delta(\mathcal L)>0$. Let $\mathcal N_{\Delta_0,L,L_1}$ be the set of isomorphism classes for $L$-graphs with \textbf{$(p, n, L_1)$-multi-labelling}-multi $\mathcal M$ that satisfy $\Delta(\mathcal M)=\Delta_0$. Then define,
\begin{align*}
\mathcal S_1 &= \frac{p^\ell}{\gamma^{L/2}}\sum_{\Delta_0=1}^{\frac{L(1+L_1)}{2}}\sum_{\mathcal L \in \mathcal N_{\Delta_0,L,L_1}} \left(\frac{(405\ell\Delta(\mathcal L))^{405\beta}}{p}\right)^{\ell\Delta(\mathcal L)}\left(\frac{n}{p^\ell}\right)^{r(\mathcal L)}\left(\prod_{s=1}^{L}\frac{|a_{d_s}(\mathcal L)|}{(d_s(\mathcal L)!)^{1/2}}\right)
\\
&= \frac{p^\ell}{\gamma^{L/2}}\sum_{\Delta_0=1}^{\frac{L(1+L_1)}{2}}\sum_{\mathcal L_0\in \mathcal N_{0, L,L_1}}\sum_{\substack{\mathcal L\in \mathcal N_{\Delta_0,L,L_1}\\ \Psi_k(\mathcal L)=\mathcal L_0}} \left(\frac{(405\ell\Delta(\mathcal L))^{405\beta}}{p}\right)^{\ell\Delta(\mathcal L)}\left(\frac{n}{p^\ell}\right)^{r(\mathcal L)}\left(\prod_{s=1}^{L}\frac{|a_{d_s}(\mathcal L)|}{(d_s(\mathcal L)!)^{1/2}}\right)
\end{align*}
Where $\Psi_k$ is defined in Definition \ref{def:map_positive_excess}. Since $a_\ell \ne 0$, we can bound $\max_{1\leq d\leq L_1}\frac{|a_d/a_\ell|}{(d!/\ell!)^{1/2}}$ by some factor $C_5$ dependent on $A$ and $\beta$ using Proposition \ref{lem:bd_on_coeff}. Then, by Proposition \ref{prop:deltapositivebound} and Lemma \ref{lem:psi_property}, we have constants $C_1,C_2,C_3>0$ such that
\begin{align}
\mathcal S_1 & \le \frac{p^\ell}{\gamma^{L/2}}\sum_{\Delta_0=1}^{\frac{L(1+L_1)}{2}} \sum_{\mathcal L_0\in \mathcal N_{0,L,L_1}}(C_1L)^{C_2\Delta_0L_1}\cdot p^{-\ell(\Delta_0-1)} \left(\frac{n}{p^\ell}\right)^{r(\mathcal L_0)}\left(\prod_{s=1}^{L}\frac{|a_{d_s}(\mathcal L_0)|}{(d_s(\mathcal L_0)!)^{1/2}}\right)\times\nonumber\\
&  \hskip 25em \max(1, \gamma^{C_3\Delta_0\ell}) C_5^{C_4\Delta_0\ell+L_1}
\nonumber\\
& \label{eq:final_bound}= \sum_{\Delta_0=1}^{\frac{L(1+L_1)}{2}} \mathcal S_0 \cdot \left(\frac{C_1L^{C_2L_1}\max(1, \gamma^{C_3\ell})C_5^{C_4\ell}}{p^\ell}\right)^{\Delta_0} C_5^{L_1}
\end{align}

Since we have chosen $L\asymp B\log n$, we can select $n$ large enough so that the sum in \eqref{eq:final_bound} is not greater than $\mathcal S_0\cdot \varepsilon_2$ for any $\varepsilon_2>0$. By Proposition 5.11 of \cite{Fan2019}, we have $\mathbb E[\tr(M^L)] \leq n\mathbb E[\|M\|^L] \leq n(\|\mu_{a, b, \gamma}\|(1+\frac{\varepsilon}{4}))^L$, for all large $n$. Thus $\mathbb E[\tr \,\kerB(X)^L] \leq n(\|\mu_{a, b, \gamma}\|(1+\frac{\varepsilon}{4}))^L (1+\varepsilon_2)$. Hence,
\[
\mathbb P[\|\kerB(X)\|>(1+\varepsilon)\|\mu_{a, b, \gamma}\|] \leq \frac{n(1+\varepsilon/4)^L(1+\varepsilon_2)}{(1+\varepsilon)^L}
\]
Now, let us pick sufficiently large $B$ so that in addition to $L\asymp B\log n$, we also have $B\log((1+\varepsilon/4)(1+\varepsilon)^{-1})<-3$. Then we have $\mathbb P[\|\kerB(X)\|>(1+\varepsilon)\|\mu_{a, b, \gamma}\|] <n^{-2}(1+\varepsilon_2)$ for sufficiently large $n$. This implies $\lim\sup_{n, p\to\infty}\|\kerB(X)\|\leq (1+\varepsilon)\|\mu_{a, b, \gamma}\|$ a.s., and letting $\varepsilon\to 0$ we get the result.
\end{proof}

\subsection{Analysis of the spectrum of $\kerC(X)$}
\label{s:KerC}
\begin{proof}[Proof of Theorem \ref{thm:specturm_of_remainder}]
Let us consider $L\in\N$ and denote by $\mathcal N_{k,L,d}$, all the $(n,p,d)$-multi-labelings of $L-$graphs satisfying $\Delta(\mathcal L)=k$, so that each $p$-vertex has $d$ $p$-labels. For notation simplicity, let $i_{L+1}:=i_1$. 
Then, by expanding the product, we have
    \begin{align}
    \label{eq:bound_c_firts_sum}
        &\E[\tr(\kerR_d(X)-\diag \kerR_d(X))^L]\nonumber\\
        &= \frac{1}{n^{L/2}p^{Ld/2}}\sum_{\substack{i_1,\ldots,i_L\\i_1 \neq i_2,i_2 \neq i_3,\ldots,i_L\neq i_1}}^n\;\sum_{k^{(1)}_1<\ldots<k^{(1)}_d}^p\cdots\sum_{k^{(L)}_1<\ldots<k^{(L)}_d}^p\E\left[\prod_{s=1}^{L}\prod_{t=1}^{d}X_{i_sk_{t}^{(s)}}X_{i_{s+1}k_t^{(s)}}\right]
    \end{align}
Furthermore, since $\mathbb E[X^2_{ij}]=1$ and $\mathbb E[|X_{ij}|^{k}] \le k^{\beta k}$, and the entries $X_{ij}$ are independent, we have,
    \begin{align*}
        \mathbb E\left[\prod_{s=1}^L\prod_{a=1}^{d}X_{i_sj_a^s}X_{i_{s+1}  j_a^s}\right] = \prod\limits_{i,j:b_{ij}>2}\mathbb E[X^{b_{ij}}_{ij}] &\le \prod\limits_{i,j:b_{ij}>2}b^{\beta b_{ij}}_{ij}\nonumber\\
        &\le \left(\sum_{i,j:b_{ij}>2}b_{ij}\right)^{\beta\sum_{i,j:b_{ij}>2}b_{ij}}\nonumber\\
        &\le \max\left\{(405 \Delta \ell)^{405\beta \Delta \ell},1\right\}
    \end{align*}
    The last line of the above equation follows using Lemma \ref{l:bsum}. The number of graph multi-labellings belonging a single equivalence class is upper bounded by $n^{r(\mathcal L)}p^{c(\mathcal L)}$ where $r(\mathcal L)$ and $c(\mathcal L)$ denote the number of distinct $n$-labels and $p$-labels, respectively, of a multi-labeling $\mathcal L$. Furthermore, since multi-labellings within the same equivalence class give the same value of $\Delta$, using Definition \ref{def:delta} we can upper bound \eqref{eq:bound_c_firts_sum} by: 
    \begin{equation}
        \sum_{\substack{\mbox{\scriptsize{$\mathcal L$: where $\mathcal L$ is a $(p,n,d)$-}}\\\mbox{\scriptsize{multilabelling}\,}\mbox{\scriptsize{of a $L$ graph}}}}\gamma^{r(\mathcal L)-L/2}p^{\ell(1-\Delta(\mathcal L))}\max\left\{(405 \Delta \ell)^{405\beta \Delta \ell},1\right\} \nonumber
    \end{equation}  
Note $\Delta\leq \frac{d(L+1)}{2}$. Therefore, for all $d \ge L_1$, we have
    \begin{align*}
        &\E[\tr(\kerR_d(X)-\diag \kerR_d(X))^L] \le  \max\{1,\gamma^{-L/2}\}\sum_{\Delta=0}^{\frac{d(L+1)}{2}}\sum_{\mathcal L\in \mathcal N_{\Delta,L,d}}p^{\ell(1-\Delta)}\max\left\{(405 \Delta \ell)^{405\beta \Delta \ell},1\right\}
    \end{align*}
Note $|\mathcal N_{0,L,d}|\leq{2L\choose L}$, as by Lemma \ref{lem:tree_excess_0} the resulting graph is a tree. By Lemma \ref{prop:deltapositivebound}, we can get constants $K_1,K_2>0$ that depend only on $\ell$, such that
\begin{align*}
   &\E[\tr(\kerR_d(X)-\diag \kerR_d(X))^L]\nonumber\\
   &\le \max\{1,\gamma^{-L/2}\}\sum_{\Delta=0}^{\frac{d(L+1)}{2}}|\mathcal N_{0,L,d}|\cdot p^{\ell(1-\Delta)} (K_1 L)^{K_2 \Delta d} \quad 
   \\
   & \le \max\{1,\gamma^{-L/2}\}\sum_{\Delta=0}^{\frac{d(L+1)}{2}}{2L \choose L}\cdot p^{\ell(1-\Delta)} (K_1 L)^{K_2 \Delta d},
 \end{align*}
Let us take $L$ satisfying 
$
L=p^{\ell/(2\, K_2d)}/K_1.
$
Then, we can get
\begin{align*}
   \E[\tr(\kerR_d(X)-\diag \kerR_d(X))^L] 
   & \le \max\{1,\gamma^{-L/2}\}\sum_{\Delta=0}^{\frac{d(L+1)}{2}}|\mathcal N_{0,L,d}|\cdot p^{\ell(1-\Delta)} (K_1 L)^{K_2 \Delta d}\nonumber\\
   & \le \max\{1,\gamma^{-L/2}\}p^\ell{2L \choose L}(1+o(1))\nonumber\\
   &\le \max\{1,\gamma^{-L/2}\}4^L\,p^\ell(1+o(1)),
     \end{align*}
where the last inequality follows using ${2L\choose L}\leq 2^{2L}$. This implies
\[
\left(\E[\tr(\kerR_d(X)-\diag \kerR_d(X))^L]\right)^{1/L} \le 4\cdot\max\{1,\gamma^{-1/2}\}\,p^{\ell/L}(1+o(1)).
\]
Using Borel Cantelli Theorem, we can conclude that for any $\delta \in (0,1)$
\begin{align*}
\|\kerR_d(X)-\diag \kerR_d(X)\| &\le (1+\delta)\left(\E[\tr(\kerR_d(X)-\diag \kerR_d(X))^L]\right)^{1/L} \nonumber\\
&\le 4\cdot \max\{1,\gamma^{-1/2}\}\,p^{\ell/L}(1+o(1)),
\end{align*}
almost surely. 
For all $\ell \in \N$, define $C(\ell):=(2K_1K_2)/e+1$. Next, for all $p,\ell \in \N$, define the function $g_{p,\ell}:\R \rightarrow \R$ as $g_{p,\ell}(x):=x\,p^{\ell/(2K_2x)}\log(1/\alpha e)$. Then, $g_{p,\ell}$ is minimized at
\[
x_*=\frac{\ell\log p}{2K_2}, \quad \mbox{and} \quad g_{p,\ell}(x_*)=\frac{e\ell\log p}{2K_2}\log(1/\alpha e).
\]
Consequently, if $\alpha \le e^{-C(\ell)}$, then for all $d \in \N$, 
\[
K_1\ell\log p \le p^{\ell/(2K_2d)}d\log(1/\alpha e).
\]
Therefore, for all $d \ge L_1$ 
\[
d\log \alpha + \frac{\ell\log p}{L} \le d\log e^{-1}.
\]
Consequently, there exists absolute constant $\tilde C_3>0$ such that
\begin{align}
\label{eq:high_high_1}
    \sum_{d \ge L_1}a_d\|\kerR_d(X)-\diag \kerR_d(X)\| & \le 4\cdot \max\{1,\gamma^{-1/2}\}\sum_{d \ge L_1}\sqrt{d+1}\alpha^dp^{\ell/L}\nonumber\\
    & \le 4\cdot \max\{1,\gamma^{-1/2}\}\sum_{d \ge L_1}\sqrt{d+1} e^{-d} \le 4\cdot \max\{1,\gamma^{-1/2}\}\varepsilon,
\end{align}
by the definition of $L_1$. 

\end{proof}

\section{Combinatorial Properties of Multi-labelings}
\label{tree}
In this section, we will gather some combinatorial properties of multi-labellings as defined in Definition \ref{def:multi-labelling}. All the proofs are deferred to Section \ref{sec:combo_row}. To that end, let us consider the following classification of the edges in the $L$-graph with a \textbf{$(p, n, L_1)$-multi-labelling}.

\begin{defn} \label{def:edge_types}
    Given a $L$-graph $\mathcal G$ with a \textbf{$(p, n, L_1)$-multi-labelling}, we pick some $p$-vertex to be the start of the cycle of edges of $\mathcal G$ and additionally give the edges in $\mathcal G$ an orientation. In this way, $\mathcal G$ becomes a directed cycle starting with some $p$-vertex. We also obtain an order for the vertices of $\mathcal G$. Consider an edge $e$ between $p$-vertex $V$ with label with tuple of labels $(j_1, \dots, j_a)$ and $n$-vertex $U$ with label $i$. Then a $\bf{sub-edge}$ of $e$ is a pair of labels, which we denote by $(i, j_k)_e$, with $k\in d_e$. (Note if $U$ came before $V$ our sub-edge would be $(j_k, i)_e$)
    
    \begin{itemize}
        \item A sub-edge $(i, j)_e$ or $(j, i)_e$ is of sub-type $\bf{t_1}$ if the second term in the tuple has not previously appeared as a label of a $p$-vertex or $n$-vertex, respectively.

        \item A sub-edge $(i, j)_e$ or $(j, i)_e$ is of sub-type $\bf{t_3}$ if there has been a single occurrence of the sub-edge $(i, j)_e$ or $(j, i)_e$ previously, and this previous occurrence was of sub-type $\bf{t_1}$.
        \item A sub-edge $(i, j)_e$ is of sub-type $\bf{t_4}$ if it is not of sub-type $\bf{t_1}$ or $\bf{t_3}$.

        \item An edge $e$ in $\mathcal G$ is of type $\bf{T_1}$ if all its sub-edges are of sub-type $\bf{t_1}$
        \item An edge $e$ in $\mathcal G$ is of type $\bf{T_3}$ if all its sub-edges are of sub-type $\bf{t_3}$
        \item An edge $e$ in $\mathcal G$ is of type $\bf{T_4}$ if it is not of type $\bf{T_1}$ or $\bf{T_3}$.
        \item A \textbf{good pair of edges} are a $\bf{T_1}$ edge and a $\bf{T_3}$ edge between vertices with the same labels.
        \item A \textbf{bad edge} is an edge that does not belong to a good pair.
    \end{itemize}
\end{defn}

Note that, we can similarly define $\bf{T_1}, \bf{T_3}$ or $\bf{T_4}$ edges for $(p,n,d)$-non-backtracking multi-labeling. We shall bound the number of bad edges and the number of $\bf{T_4}$ edges. In that direction, consider the following propositions.

\begin{prop}\label{prop:unpaired_edges}
    The maximum number of bad edges in an $L$-graph $\mathcal G$ with \textbf{$(p, n, L_1)$-multi-labelling} $\mathcal M$ is no greater than $120\Delta(\mathcal M)\ell$. 
\end{prop}

\begin{prop} \label{prop:t4bound}
    The number of $\bf{T_4}$ edges in an $L$-graph $\mathcal G$ with \textbf{$(p, n, L_1)$-multi-labelling} $\mathcal M$ is bounded $135 \Delta(\mathcal M) \ell$.
\end{prop}

Lemma \ref{l:bsum} is an easy consequence of Proposition \ref{prop:t4bound}. 
 \begin{proof}[Proof of Lemma \ref{l:bsum}]
    Lemma \ref{l:bsum} follows from Proposition \ref{prop:t4bound} and the fact that $\sum_{i,j:b_{ij}>2}b_{ij}$ is bounded by three times the number of $\mathbf{T_4}$ edges. 
    \end{proof}

Also, consider the following properties of the multi-labelings with excess equal to zero. 

\begin{lem} 
\label{lem:tree_excess_0}
Consider an $L$-graph $\mathcal G$ with $(p, n, L_1)$ multi-labelling $\mathcal M$. If $\Delta(\mathcal M)=0$, then we have the following:
\begin{itemize}
    \item[(1)] For each edge $(U, V)$ in $\mathcal G$ there exists exactly one other edge $(U', V')$ in $\mathcal G$ so that the label of $U$ is the same as the label of $U'$ and the label of $V$ is the same as the label of $V'$.
    \item[(2)] If we combine such edges between vertices of equal labels the resulting graph $\mathcal G'$ is a tree that has distinct $n$-labels for each $n$-vertex and has no two tuples of $p$-labels that share a common $p$-label.
    \item[(3)] The degree of each $p$-vertex in $\mathcal G'$ is greater than $1$.
    \item[(4)] If $V$ is a $p$-vertex with more than $\ell$ $p$-labels, then the degree of $V$ in $\mathcal G'$ is $2$.
\end{itemize}
Furthermore, for any $L$-graph $\mathcal G$ with $(p, n, L_1)$ multi-labelling $\mathcal M$ that satisfies conditions $(1)$, $(2)$, $(3)$, and $(4)$, we have $\Delta(\mathcal M)=0$.
\end{lem}

Next, we shall construct a map from the set of multi-labelings to the set of multi-labelings with excess equal to zero.

\begin{defn}[Positive Excess Map] \label{def:map_positive_excess}
    Let us recall $\mathcal N_{k, L, L_1}$, the set of isomorphism classes for $L$-graphs $\mathcal G$ with \textbf{$(p, n, L_1)$-multi-labelling} $\mathcal M$ that satisfy $\Delta(\mathcal M)=k$. We shall construct a map $\Psi_k:\mathcal N_{k, L, L_1}\to\mathcal N_{0, L, L_1}$. Let us refer to edges belonging to good pairs as good edges. For this proof, we shall refer to all edges that are not good as bad edges. Now consider the following steps:

  \bigskip

  Step $1$: For each isomorphism class $G\in \mathcal N_{k, L, L_1}$, we first remove all the edges that are bad in $G$ and all the vertices in $G$ that are not incident to at least one good edge. Let $\mathcal U = \{U_{m_1}, \dots, U_{m_s}\}$ denote the set of vertices that are incident to both good and bad edges. Here, $1\leq m_1 < m_2 < \dots < m_s\leq 2L$ and $U_m$ denotes the $m$th vertex of $G$. Note that, given any $n$-label $i$ or a tuple of $p$-labels $J$, the number of elements in $\mathcal U$ with that label (or tuple of labels) is even. Otherwise, we have an odd number of good edges incident to the vertices with label $i$ or labels $J$ which we cannot pair up.
  
  \bigskip

  Step $2$: We define a chain to be a sequence of consecutive edges in $G$. Let $A_r$ for $1\leq r\leq s+1$, denote the edges of $G$ between $U_{k_{r-1}}$ and $U_{k_r}$ with the convention that $U_{k_0}$ and $U_{k_{s+1}}$ are the first and last vertex in $G$. Observe that each set of edges $A_r$ contains only good or bad edges. We will refer to chains containing only good or only bad edges as good and bad chains. Note each $U_{k_r}$ is the endpoint of a good and bad chain. Now, we would like to produce a single good chain by combining the good chains amongst $A_r$. In that direction, let us find good chains $A_{r_1}$ and $A_{r_2}$ that share an endpoint with the same vertex label (or tuple of labels). Then, we glue the two chains together at that vertex (possibly reversing the order of one of the chains) and continue this process until no more than one good chain remains. Observe that this process terminates in a single good chain which is also a cycle because each possible label of an endpoint of a chain appears an even number of times.

  \bigskip

  Step $3$: We now have a single chain $A$ with no more than $2L$ edges (since we removed some bad edges from the original graph). However, it is possible that the chain contains $p$-vertices of degree $1$ (or equivalently consecutive $n$-vertices with equal labels). Then, we remove such a $p$-vertex and the pair of edges incident to it from $A$.

  \bigskip

  Step $4$: Now, we have a good chain $A$ of length $2L'$ with $L'\leq L$ (since we removed some bad edges) which is also a cycle. We consider the set of $p$-vertices with more than $\ell$ $p$-labels and degrees more than $4$ i.e. they are incident to more than two good pairs of edges in this chain. For each such vertex, we simply remove an arbitrary subset of its $p$-labels so that it is left with precisely $\ell$ labels. We want $A$ to have at least one $p$-vertex with $\ell$ labels so if $A$ contains only $p$-vertices with more than $\ell$ $p$-labels with degree equal to $4$, then we arbitrarly select any $p$-vertex and remove an arbitrary subset of its $p$-labels so that it is left with precisely $\ell$ labels.

  \bigskip
  
  Step $5$: We want to extend the length of $A$ back to $2L$. Suppose $A$ has $2L'$ edges $L'\leq L$. Let $V$ be the first $p$-vertex in the chain $A$ with $\ell$-labels. Suppose $(U, V)$ and $(V, U')$ are consecutive edges in $A$ for $n$-vertices $U$ and $U'$. Then, we insert two edges $(V, W)$ and $(W, V)$ in between $(U, V)$ and $(V, U')$ where $W$ has distinct $n$-label from other $n$-vertices in $A$. We continue this process until we reach a cycle of length $2L$. Let this good chain of length $2L$ obtained be denoted by $\Psi_k(G)$.

  \bigskip

  We claim that $\Psi_k(G)$ produces a valid isomorphism class belonging to $\mathcal N_{0, L,L_1}$. Steps $1, 2, 5$ ensure that condition $(1)$ of Lemma \ref{lem:tree_excess_0} is satisfied. Condition $(2)$ is satisfied since $\Psi_k$ only keeps the good edges of $G$. Condition $(3)$ is satisfied by Step 3, and Condition $(4)$ is satisfied by step $(4)$. Thus we conclude that $\Psi_k$ produces an isomorphism class belonging to $\mathcal N_{0, L,L_1}$ by Lemma \ref{lem:tree_excess_0}.
\end{defn}

The following two results bound the number of valid multi-labelings with excess $\Delta$ in terms the number of number of multi-labelings with excess $0$.

\begin{prop} \label{prop:deltapositivebound}
    Consider a multi-labelling $H\in \mathcal N_{0, L, L_1}$. For any $k>0$ we have that $|\Psi_k^{-1}(H)|\leq (C_1 L)^{C_2 k L_1}$ where $C_1$ and $C_2$ are constants that depend solely on $\ell$.
\end{prop}

\begin{lem} \label{lem:psi_property}
    Given $\mathcal M\in \mathcal N_{k, L, L_1}$, define $\mathcal M_0=\Psi_k(\mathcal M)$ where $\Psi_k$ is the map from Definition \ref{def:map_positive_excess}. Let $r(\mathcal M)$ and $r(\mathcal M_0)$ denote the number of distinct $n$-labels of the two multi-labellings. Further, let $d_s(\mathcal M)$ denote the number of $p$-labels of the $s$th $p$-vertex of $\mathcal M$ for $1\le s\le L$ and  $d_r(\mathcal M_0)$ denote the number of $p$-labels of the $r$th $p$-vertex of $\mathcal M_0$. Let $C_3$ and $C_4$ be positive constants. Then,
    \begin{itemize}
        \item[(1)] $r(\mathcal M)-r(\mathcal M_0)\leq C_3k\ell$.
        \item[(2)] $\prod_{s=1}^L\frac{|a_{d_s}(\mathcal M)|}{(d_s(\mathcal M)!)^{1/2}}\leq \prod_{r=1}^{L}\frac{|a_{d_r}(\mathcal M_0)|}{(d_r(\mathcal M_0)!)^{1/2}}\cdot \max_{1\le d\le L_1}\frac{|a_d/a_\ell|^{C_4 k\ell+L_1}}{(d!/\ell!)^{C_4k\ell/2+L_1}}$.
    \end{itemize}
\end{lem}
The results above allow us to restrict the sum in \eqref{eq:sum_multilabeling_1} to multi-labellings $\mathcal M$ where $\Delta(\mathcal M)=0$.  This significantly simplifies the proof of Theorem \ref{thm:spctrum_b} compared to the approach used in \citet{Fan2019}.

\section{Non-backtracking Matrices}
\label{a_x}

In this section, we will prove Proposition \ref{prop:Td_norm} and Proposition \ref{prop:Td_norm_d=l}. Before doing so, we need to establish some basic properties of the non-backtracking matrices $T_d(L)$ for all $d \leq \ell$ and $L \in \N$, as defined in Definition \ref{def:non_backtracking_matrix}. First, these matrices satisfy the following recurrence relation:

\begin{lem} \label{lem:recursion}
    For all $L \in \N$ and $1 \le d \le \ell$, the matrices $T_d(L)$ satisfy the recurrence relation: 
    \begin{align}
    T_d(1)T_d(L)&=T_d(1)T_d(L)= T_d(L+1)+(d!)^{-1/2}\sqrt{\frac{p^d}{n}}T_d(L)+T_d(L-1)+o(1), \quad \mbox{almost surely.}\nonumber
    \end{align}
    where the error term has spectral norm bounded by $o(1)$.
\end{lem}
\begin{proof}
    By definition of $T_d(L)$, we have 
    \begin{align*}
    \left[T_d(1)T_d(L)\right]_{ij}= \frac{(d!)^{\frac{1+L}{2}}}{n^{\frac{1+L}{2}} p^{\frac{d(1+L)}{2}}}
    \sum_{\mathcal L(1+L)}^* \prod_{s=1}^d X_{u_1 k_s^{(0)}}X_{u_1 k_s^{(1)}} X_{u_2 k_s^{(1)}}X_{u_2 k_s^{(2)}} \dots X_{u_{1+L}, k_s^{(L)}}X_{u_{1+L} k_s^{(1+L)}},
    \end{align*}
    where the summation $\sum_{\mathcal L(1+L)}^*$ satisfies the same conditions as $\sum_{\mathcal L(1+L)}$ in Definition \ref{def:non_backtracking_matrix} except, it is possible that the indices $u_1$ and $u_2$ are equal. We divide the summation $\sum_{\mathcal L(1+L)}^*$ into three parts $\sum_{\mathcal L_0(1+L)}^*, \sum_{\mathcal L_1(1+L)}^*$, and $\sum_{\mathcal L_2(1+L)}^*$. The term $\sum_{\mathcal L_0(1+L)}^*$ is the sum of the subset of terms that satisfy $u_{1}\ne u_{2}$, $\sum_{\mathcal L_1(1+L)}^*$ is the sum of the subset of terms that satisfy $(k_1^{(0)}, \dots, k_d^{(0)})\ne (k_1^{(2)}, \dots, k_d^{(2)})$ and $u_{1}= u_{2}$. Finally, $\sum_{\mathcal L_2(L_1+L_2)}^*$ is the sum of the subset of terms that satisfy $u_{1}= u_{2}$ and $(k_1^{(0)}, \dots, k_d^{(0)})= (k_1^{(2)}, \dots, k_d^{(2)})$. 

   It is not hard to observe that the term $\sum_{\mathcal L_0(1+L)}^*$ gives the matrix $T_d(L+1)$. Next, observe that, the sum $\sum_{\mathcal L_1(1+L)}^*$ can be written as
   \begin{align}
   \label{eq:expansion_break_1}
   &\frac{(d!)^{\frac{1+L}{2}}}{n^{\frac{1+L}{2}} p^{\frac{d(1+L)}{2}}}
    \sum_{\mathcal L(L)}^*\sum_{\substack{(k_1^{(0)}, \dots, k_d^{(0)}) \neq (k_1^{(1)}, \dots, k_d^{(1)})\\
    (k_1^{(2)}, \dots, k_d^{(2)}) \neq (k_1^{(1)}, \dots, k_d^{(1)})}} \left\{\prod_{s=1}^dX^2_{u_1 k_s^{(1)}}\right\}\prod_{s=1}^d X_{u_1 k_s^{(0)}}X_{u_1 k_s^{(2)}} \dots X_{u_{1+L}, k_s^{(L)}}X_{u_{1+L} k_s^{(1+L)}}\nonumber\\ 
    &=\frac{(d!)^{\frac{1+L}{2}}}{n^{\frac{1+L}{2}} p^{\frac{d(1+L)}{2}}}
    \sum_{\mathcal L(L)}^*\sum_{S \subset [d]}\sum_{\substack{\{k^{(1)}_s:s \in S\}\\\mbox{\scriptsize{$k^{(1)}_s$ are distinct}}}} \prod_{s \in S}\left\{X^2_{u_1 k_s^{(1)}}-1\right\}\left\{\prod_{s=1}^d X_{u_1 k_s^{(0)}}X_{u_1 k_s^{(2)}} \dots X_{u_{1+L}, k_s^{(L)}}X_{u_{1+L} k_s^{(1+L)}}\right\}\nonumber\\
    & \hskip 20em +o(1).
   \end{align}
   In the above equation, observe that we have relaxed the restrictions $(k_1^{(0)}, \dots, k_d^{(0)}) \neq (k_1^{(1)}, \dots, k_d^{(1)})$ and $(k_1^{(2)}, \dots, k_d^{(2)}) \neq (k_1^{(1)}, \dots, k_d^{(1)})$. Since $\mathbb E[|X_{ij}|^k] \le k^{\beta k}$ for all $k>2$ and $p^\ell/n \rightarrow \gamma \in (0,\infty)$, this relaxation can be justified using the arguments similar to Lemma 6 of \citet{bai_yin}. It can be shown that the error incurred is almost surely $o(1)$. Furthermore, using similar arguments, we can also show that the only term that can possibly have a non-zero contribution to the last sum in \eqref{eq:expansion_break_1} as $n,p \rightarrow \infty$ is 
   \begin{align}
    &(d!)^{-1/2}\sqrt{\frac{p^d}{n}}\frac{(d!)^{\frac{L}{2}}}{n^{\frac{L}{2}} p^{\frac{dL}{2}}}
    \sum_{\mathcal L(L)}\prod_{s=1}^d X_{u_1 k_s^{(0)}}X_{u_1 k_s^{(2)}} \dots X_{u_{1+L}, k_s^{(L)}}X_{u_{1+L} k_s^{(1+L)}}\nonumber\\
    &=(d!)^{-1/2}\sqrt{\frac{p^d}{n}}T_d(L),
   \end{align}
   and all other terms are almost surely $o(1)$. Furthermore, using the same arguments, we can also show that the only significant term in the sum $\sum_{L_2(1+L)}^*$ is 
   \[
   (d!)^{-1/2}\sqrt{\frac{p^d}{n}}\frac{(d!)^{\frac{L}{2}}}{n^{\frac{L}{2}} p^{\frac{dL}{2}}}
    \left\{\sum_{\mathcal L(L)}\prod_{s=1}^d X^2_{u_1 k_s^{(0)}}X_{u_3 k_s^{(2)}} \dots X_{u_{1+L}, k_s^{(L)}}X_{u_{1+L} k_s^{(1+L)}}\right\}.
   \]
   Decomposing the above sum as in \eqref{eq:expansion_break_1}, we can show that as $n,p \rightarrow \infty$, the only possible non-zero contribution comes from
   \[
   \frac{(d!)^{\frac{L-1}{2}}}{n^{\frac{L-1}{2}} p^{\frac{d(L-1)}{2}}}
    \left\{\sum_{\mathcal L(L-1)}\prod_{s=1}^d X_{u_3 k_s^{(2)}} \dots X_{u_{1+L}, k_s^{(L)}}X_{u_{1+L} k_s^{(1+L)}}\right\}=T_d(L-1).
   \]
   Hence, we can conclude the proof of the lemma.
\end{proof}

As a consequence of Lemma \ref{lem:recursion}, we have the following lemma:

\begin{lem} \label{lem:Td_formula}
    For $d<\ell$ the matrix $T_d(1)^L=\sum_{m=0}^L C_d(m, L)T_d(m)(1+o_p(1))$ for constants $C_d(m, L)$ for $0\leq m\leq L$ satisfying $|C_d(m, L)|\leq 2^L$.
\end{lem}

\begin{proof}
    We proceed by induction on $L$. The base case $L=1$ is trivial. Assume the claim holds for $L=M$. Then we have $T_d(1)^M=\sum_{m=0}^M C_d(m, M)T_d(m)(1+o_p(1))$ for constants $C_d(m, M)$ satisfying $|C_d(m, M)|\leq 2^M$. This implies:
    \begin{align}
        T_d(1)^{M+1} &= T_d(1)\left[\sum_{m=0}^M C_d(m, M)T_d(m)(1+o_p(1))\right] \\
        &= \sum_{m=0}^M C_d(m, M)T_d(1)T_d(m)+T_d(1)\cdot o_p(1)\nonumber\\
        & = \sum_{m=0}^M C_d(m, M)T_d(m+1)(1+o_p(1))+\sum_{m=1}^M C_d(m, M)T_d(m-1)(1+o_p(1)).
    \end{align}

    In the above equation, we used Lemma \ref{lem:recursion} and the fact that $\sqrt{\frac{p^d}{n}}\cdot (d!)^{-1/2}$ is $o_p(1)$. If we define $C_d(m, M+1)=C_d(m-1, M)+C_d(m+1, M)$ where $C_d(-1,M)=C_d(0, M)=C_d(M+1, M)=C_d(M+2, M)=0$, by the inductive hypothesis $|C_d(m-1, M)+C_d(m+1, M)|\leq 2^{M+1}$. This proves the claim.
\end{proof}

To analyze the trace of the matrix $T_d(L)$ we define a non-backtracking multi-labelling of an $L$-graph. Summing over all non-backtracking multi-labelling will allow us to provide a bound on the trace of the matrix $T_d(L)$.

\begin{defn} \label{def:col_multi-labelling}
    A $(p, n, d)$-\textbf{non-backtracking multi-labelling} of an $L$-graph is an assignment of a $n$-label in $[n]$ to each $n$-vertex and an ordered tuple of $d$ elements of $p$-labels in $[p]$ to each $p$-vertex, such that:
    \begin{itemize}
        \item[$(i)$] The $n$-label of each $n$-vertex is distinct from those of the the two $n$-vertices immediately preceding and following it in the cycle. The tuple of $p$-labels of each $p$-vertex is distinct from those of the two $p$-vertices immediately preceding and following it in the cycle. Two tuples of $p$-labels that are reorderings of each other are considered equal.
        \item[$(ii)$] For each distinct $n$-label $i$ and $p$-label $j$ there are an even number of edges in the cycle of the form $(i, j)$.
        \item[$(iii)$] $d$ is an integer such that $d\leq \ell$.
    \end{itemize}
\end{defn}

Let $\tilde r$ be the number of distinct $n$-labels and $\tilde c$ be the number of distinct $p$-labels for a given non-backtracking multi-labeling. Furthermore, let us recall the definition of the excess of a multi-labeling from \eqref{eq:excess_simple_label}. Then we have the following lemma. We postpone its proof to Section \ref{s:proof_comb_nonbacktracking}. 

\begin{lem}
\label{lem:col_positive_excess}
    Given a $(p, n, d)$-non-backtracking multi-labeling of an $\ell$-graph with $d$ $p$-labels for each of its $p$-vertices, we have
    \[
    \tilde r+\frac{\tilde c}{\ell}\leq \frac{L}{2}+\frac{dL}{2\ell}
    \]
\end{lem}

\begin{defn}\label{def:col_excess}
Let us define the excess of a $(p, n, d)$-non-backtracking multi-labeling of a $L$-graph to be:
\[
\bar\Delta =1+\frac{L}{2}+\frac{dL}{2\ell}-\tilde r -\frac{\tilde c}{\ell}
\]
\end{defn}
By Lemma \ref{lem:col_positive_excess}, clearly, $\bar\Delta \ge 0$. 

\begin{defn} \label{def:multi_label_classes_col}
    Two $(p, n, d)$-non-backtracking multi-labelings of an $L$-graph are \textbf{equivalent} if there is a permutation $\pi_p$ of $[p]$ and a permutation $\pi_n$ of $[n]$ such that one labeling is the image of the other upon applying $\pi_p$ to all of its $p$-labels and $\pi_n$ to all of its $n$-labels. For any fixed $L$, the equivalence classes under this relation will be called \textbf{non-backtracking multi-labeling equivalence classes}.
\end{defn}

To characterize the spectral norm of $T_d(X)$, we shall need a slightly general class of non-backtracking multi-labelings, where we allow backtracking at a few possible $p$-labels. 

\begin{defn} \label{def:generalized_multi_label}
    We define a generalized $(p, n, d)$-non-backtracking multi-labeling on an $mL$-graph as follows. Starting from a vertex $V_1$, we label the vertices of the graph as $V_1, V_2, \dots, V_{2mL}$. Let $S \subset \{V_3, V_5, \dots, V_{2mL-1}\}$ be a subset of these vertices, referred to as the set of backtracking vertices.

    We modify condition (i) of Definition \ref{def:col_multi-labelling} such that backtracking is permitted between vertices $V_{k-1}$ and $V_{k+1}$ when $V_k \in S$. Specifically, the restriction that $V_{k-1}$ and $V_{k+1}$ must have different labels (or a tuple of labels) is relaxed for backtracking vertices in $S$. All other conditions of Definition \ref{def:col_multi-labelling} remain unchanged.
\end{defn}

It can be shown that this slightly generalized class of multi-labelings also has a strictly non-negative excess.

\begin{lem}\label{lem:generalized_positive_excess}
    Let $\mathcal{L}$ be a generalized $(p, n, d)$-non backtracking multi-labeling of an $L$-graph. The excess $\bar{\Delta}(\mathcal{L})$ is non-negative, i.e., $\bar{\Delta}(\mathcal{L}) \geq 0$, for any backtracking set $S$, where $\bar{\Delta}$ is defined as in Definition \ref{def:col_excess}.
\end{lem}

This result follows trivially from Lemma \ref{lem:row_positive_excess} by noting that the generalized multi-labeling is a subclass of the $(p, n, d)$-multi-labeling, with the roles of rows and columns interchanged. So we omit its proof.

Next, we present a theorem that characterizes the spectral norm of $T_d(X)$.
\begin{thm}
\label{thm:t_d_l}
    For all $d \leq \ell$, $L \in \N$, and $\bar C =(C_1L)^{C_3}$  with constants $C_1, C_3>0$, we have 
    \[
    \limsup_{n \rightarrow \infty} \|T_d(L)\| \le \bar C\; \textrm{ almost surely.}
    \]
\end{thm}
\begin{proof}
    To prove this proposition, we establish a bound on $\mathbb{E}\left[\tr\left(T_d(L)^{2m}\right)\right]$. Using the combinatorial results from Section \ref{sec:combo_col}, we will derive a bound on the trace of each matrix in the sum.

We aim to bound $\mathbb{E}\left[\tr(T_d(L))\right]$ for $L \in \mathbb{N}$. From Definition \ref{def:non_backtracking_matrix} of $T_d(L)$, we recall the following expression:
\begin{align*}
    \mathbb{E}\left[\tr\left(T_d(L)\right)\right] = \sum_{k_s^{(0)} \in \binom{p}{d}} \frac{(d!)^{L/2}}{\sqrt{n^L p^{dL}}} \sum_{\mathcal{L}(\ell)} \mathbb{E}\left[\prod_{s=1}^d X_{u_1 k_s^{(0)}} X_{u_1 k_s^{(1)}} X_{u_2 k_s^{(1)}} X_{u_2 k_s^{(2)}} \dots X_{u_L k_s^{(L-1)}} X_{u_L k_s^{(0)}}\right],
\end{align*}
where $\sum_{k_s^{(0)} \in \binom{p}{d}}$ denotes the sum over all possible $d$-tuples $u$, i.e., one for each row of $T_d(L)$. For notational simplicity, let $k_s^{(L)} := k_s^{(0)}$. Thus, following Definition \ref{def:col_multi-labelling}, we have:
\begin{align*}\label{eq:sum_Lgraph}
    \mathbb{E}\left[\tr\left(T_d(L)\right)\right] = \sum_{\text{($p,n,d$)-non-backtracking multi-labelings of $L$-graphs}} \frac{(d!)^{L/2}}{\sqrt{n^L p^{dL}}} \mathbb{E}\left[\prod_{t=1}^L \prod_{s=1}^{d} X_{u_t k_s^{(t-1)}} X_{u_t k_s^{(t)}}\right].
\end{align*}

Now, recall Definition \ref{def:generalized_multi_label}. Let $\mathcal{B}$ denote the set of generalized $(p,n,d)$-non-backtracking multi-labeling equivalence classes of $2mL$-graphs, where the backtracking vertex set is $S = \{V_{2L+1}, V_{4L+1}, \dots, V_{4(m-1)L+1}\}$. Following similar steps as for $\mathbb{E}\left[\tr\left(T_d(L)\right)\right]$, we can write:
\begin{align*}
    \mathbb{E}\left[\tr\left(T_d(L)^{2m}\right)\right] = \sum_{B \in \mathcal{B}} \frac{(d!)^{mL}}{n^{mL} p^{dmL}} \mathbb{E}\left[\prod_{t=1}^{2mL} \prod_{s=1}^{d} X_{u_t k_s^{(t-1)}} X_{u_t k_s^{(t)}}\right].
\end{align*}

We define the generalized equivalence classes in Definition \ref{def:generalized_multi_label} to account for potentially equal $n$-indices, which result in backtracking when multiplying two copies of $T_d(L)$. Let $\mathcal{L}$ be any element of $\mathcal{B}$. Recall that $\mathbb{E}[X_{ij}^2] = 1$ and $\mathbb{E}\left[|X_{ij}|^k\right] \leq k^{\alpha k}$. Let $b_{ij}$ denote the number of times $X_{ij}$ appears in the product $\prod_{t=1}^{\ell} \prod_{s=1}^{d} x_{u_t k_s^{(t-1)}} x_{u_t k_s^{(t)}}$. Since the entries $X_{ij}$ are independent, we have:
\begin{align*}
    \mathbb{E}\left[\prod_{t=1}^{\ell} \prod_{s=1}^{d} x_{u_t k_s^{(t-1)}} x_{u_t k_s^{(t)}}\right] &= \prod_{i,j : b_{ij} > 2} \mathbb{E}\left[X_{ij}^{b_{ij}}\right] \leq \prod_{i,j : b_{ij} > 2} b_{ij}^{\alpha b_{ij}} \\
    &\leq \left(\sum_{i,j : b_{ij} > 2} b_{ij}\right)^{\alpha \sum_{i,j : b_{ij} > 2} b_{ij}} \leq (36 \Delta(\mathcal{L}) \ell)^{36 \alpha \Delta(\mathcal{L}) \ell}.
\end{align*}

In the last line, we used (2) in Proposition \ref{prop:gen_back_tracking_bound} and the fact that $\sum_{i,j : b_{ij} > 2}$ is at most three times the number of $\mathbf{T_4}$ edges. Let $r(\mathcal{L})$ and $c(\mathcal{L})$ denote the number of distinct $n$-labels and $p$-labels in $B$, respectively. For a fixed equivalence class $B$, there are at most $n^{r(\mathcal{L})} p^{c(\mathcal{L})}$ ways to select the $n$-labels and $p$-labels of the graph. However, due the restriction $k_1^{(t)}<k_2^{(t)}<\dots <k_d^{(t)}$ we have overcounted by a factor of at least $(d!)^{mL-12\Delta_0\ell}$. This can be seen by (3) in  Proposition \ref{prop:gen_back_tracking_bound}, since following each $\bf{T_1}$ edge into a $p$-vertex $V$, we have overcounted the number of ways to label that vertex by $d!$. The possible excess $\Delta_0 = \Delta(\mathcal{L})$ ranges from $0$ to $\frac{d(2mL+1)}{2}$, where $\ell \Delta_0$ is an integer. Therefore, we have:
\begin{align*}
    \mathbb{E}\left[\tr\left(T_d(L)^{2m}\right)\right] &\leq \sum_{B \in \mathcal{B}} \frac{(d!)^{mL} n^{r(\mathcal{L})} p^{c(\mathcal{L})}}{\sqrt{n^L p^{dL}}} \\
    &= \sum_{\Delta_0 = 0}^{\frac{d(2mL+1)}{2}} \sum_{\substack{\mathcal{L} \in \mathcal{B} \\ \Delta(\mathcal{L}) = \Delta_0}} n^{1 - \Delta_0} \gamma^{\frac{c(\mathcal{L})}{\ell} - \frac{dmL}{\ell}} (d!)^{mL} \cdot (36 \Delta_0 \ell)^{36 \alpha \Delta_0 \ell} (d!)^{mL-12\Delta_0\ell}.
\end{align*}

We can bound $\gamma^{\frac{c(\mathcal L)}{\ell}-\frac{dmL}{\ell}}$ by $\max\{1, \gamma^{-\Delta_0}$\} since $r(\mathcal L)\leq mL$ which implies $\frac{dmL}{\ell}- \frac{c(\mathcal L)}{\ell}\leq \Delta_0$. By Proposition \ref{prop:gen_back_tracking_bound}, where the backtracking set has size $|S| = m - 1$, we conclude that for constants $C_1, C_2, C_3$ (depending on $\ell, \alpha$, and $d$), we have:
\begin{align*}
    \mathbb{E}\left[\tr\left(T_d(L)^{2m}\right)\right] \leq \sum_{\Delta_0 = 0}^{\frac{d(2mL+1)}{2}} (2L+1)^{2m} \left(C_1(2mL)\right)^{C_2 \Delta_0 + C_3} n^{1-\Delta_0}.
\end{align*}

If we choose $m$ such that $m / \log n \to 1$, and set $z = (C_1L)^{C_3+3} \cdot (1 + \varepsilon)$ for some $\varepsilon > 0$, then $\sum_m z^{-2m} \mathbb{E}\left[\tr\left(T_d(L)^{2m}\right)\right] < \infty$. Defining $\bar{C} = (C_1L)^{C_3+2}(1+\varepsilon)$, and using:
\[
\mathbb{P}\left[\|T_d(L)\| \geq z\right] \leq \frac{\mathbb{E}\left[\tr\left(T_d(L)^{2m}\right)\right]}{z^{2m}},
\]
the result follows by the Borel-Cantelli Lemma.
\end{proof}

Finally, we can give the proofs of Proposition \ref{prop:Td_norm} and Proposition \ref{prop:Td_norm_d=l}.

\begin{proof}[Proof of Proposition \ref{prop:Td_norm}]
     Using Lemma \ref{lem:Td_formula}, we obtain that $T_d(1)^{L}=\sum_{m=0}^L C_d(m, L)T_d(m)(1+o(1))$ with $C_d(m, L)\leq 2^L$. Therefore, using Theorem \ref{thm:t_d_l} we can show that $\|T_d(1)^L\|\leq L\cdot 2^{L}\cdot (C_1 L)^{C_3}$. Finally, taking the $L$-th root of both sides and letting $L\to\infty$, we obtain the result.
\end{proof}

\begin{proof}[Proof of Proposition \ref{prop:Td_norm_d=l}]
     From Lemma \ref{lem:recursion}, we obtain that $\left[T_d(1)-(\gamma d!)^{-\frac{1}{2}}I_{p\choose d}\right]T_d(L) = T_d(L+1)+T_d(L-1)+o(1)$.
    
    Furthermore, we can re-write $\left[T_d(1)- (\gamma d!)^{-1/2}I_{p\choose d}\right]^{L}$, as
    \[
    \left[T_d(1)- (\gamma d!)^{-1/2}I_{p\choose d}\right]^{L}=\sum_{m=0}^L C_d(m, L)T_d(m)(1+o(1)).
    \]
    We shall inductively show that, $|C_d(m, L)|\leq 2^L$. The base case clearly holds as \[(\gamma d!)^{-1/2} = \lim_{n\to\infty}\sqrt{{p\choose d}n^{-1}}\leq 1.\]
    Suppose the statement holds for some $L_0$. Then,    
    \begin{align*}
        \left[T_d(1)- (\gamma d!)^{-1/2}I_{p\choose d}\right]^{L_0+1}&=\left[T_d(1)- (\gamma d!)^{-1/2}I_{p\choose d}\right]\sum_{m=0}^{L_0} C_d(m, L_0)T_d(m)(1+o(1)) \\
        &=\sum_{m=0}^{L_0+1} [C_d(m-1, L_0)+C_d(m+1, L_0)]T_d(m)(1+o(1))
    \end{align*}
    
    where we have used the convention that $C_d(-1, L_0)=C_d(L_0+1, L_0)=C_d(L_0+2, L_0)=0$. Now, setting $C_d(m, L_0+1)=C_d(m-1, L_0)+C_d(m+1, L_0)$, the inductive step follows. Using $C_d(m, L)\leq 2^L$ and Theorem \ref{thm:t_d_l}, we can show that $\|T_d(1)-\sqrt{d!\gamma} I_{p\choose d}\|^{L}\leq L\cdot 2^{L}\cdot (C_1 L)^{C_3}$, almost surely. Finally, taking the $L$-th root on both sides and letting $L\to\infty$, we obtain the result.
\end{proof}

\section{Combinatorial Properties of Non-backtracking Multilabelings}
\label{sec:combo_col}
In this section, we shall provide an upper bound to the number of non-backtracking multi-labelling equivalence classes with excess $\Delta_0$ over all $(p,n,d)$-non backtracking multi-labelings. We postpone their proofs to Section \ref{s:proof_comb_nonbacktracking}. To that end, let us recall the classification of edges defined in Definition \ref{def:edge_types}. Furthermore, also consider the following definitions.
\begin{defn} \label{def:unpaired_edge}
    Consider an $L$-graph $\mathcal G$ with $(p, n, d)$-non-backtracking multi-labeling $\mathcal L$. Let $e^{(i)}$ denote the $i$th edge in $\mathcal G$. Let $e^{(r)}$ be a $\bf{T_1}$ edge in the graph between vertices $(U, V)$ with $n$-label $i$ and tuple of $p$-labels $(j_1, \dots, j_d)$. Furthermore, let $s$ be an integer between $r$ and $2L$. We say that edge $e^{(r)}$ is \textbf{unpaired at s} if at least one of the sub-edges $(i, j_k)$ has appeared only once amongst the edges $e^{(1)}, \dots, e^{(s)}$ (as a sub-edge). 
\end{defn}

\begin{defn} \label{def:regular_T3}
     Consider an $L$-graph $\mathcal G$ with $(p, n, d)$-non-backtracking multi-labeling $\mathcal L$. Let $e^{(r)}$ be a $\bf{T_3}$ edge in $\mathcal G$ from vertex $U$ to vertex $V$. We say that $e^{(r)}$ is a $\bf{T_3}$ edge of \textbf{order $\bf{\tau}$} if amongst the edges $e^{(1)}, \dots, e^{(r-1)}$ there are exactly $\tau$ unpaired $\bf{T_1}$ edges at $r$ that are incident to a vertex with the same label as $U$.
\end{defn}
The next proposition shows that the number of $\bf{T_4}$ edges in any $L$-graph $\mathcal G$ with $(p, n, d)$-non-backtracking multi-labeling is bounded above by a constant multiple of the excess $\Delta$ of the multi-labeling.
\begin{prop} \label{prop:col_t4}
    The number of $\bf{T_4}$ edges in an $L$-graph $\mathcal G$ with $(p, n, d)$-non-backtracking multi-labeling $\mathcal L$ is bounded $12 \Delta(\mathcal L)\ell$.
\end{prop}
Similarly, the order of any $\bf{T_3}$ edge is also bounded above by a constant multiple of $\Delta$.
\begin{prop} \label{prop:T3choice_col}
    Consider an $L$-graph $\mathcal G$ with $(p, n, d)$-non-backtracking multi-labeling $\mathcal L$. Then there is no $\bf{T_3}$ edge of order $\tau > 24\Delta(\mathcal L)\ell + 1$.
\end{prop}
Using Propositions \ref{prop:col_t4} and \ref{prop:T3choice_col}, we prove the following upper bound on the number of equivalence classes of $L$-graphs with excess $\Delta_0$ over all $(p,n,d)$-non backtracking multi-labelings.
\begin{prop} \label{prop:num_graphs_col}
    There are at most $(C_1L)^{C_2\Delta_0+C_3}$ equivalence classes of $L$-graphs with excess $\Delta_0$ over all $(p, n, d)$-non-backtracking multi-labelings.
\end{prop}

\begin{proof}
    Recall that by Proposition \ref{prop:col_t4}, there are at most $12\Delta_0\ell$ $\bf{T_4}$ edges in any $L$-graph with excess $\Delta_0$. The remainder of the edges are $\bf{T_1}$ and $\bf{T_3}$. We may select the positions of the $\bf{T_4}$ edges in at most $L^{12\Delta_0 \ell}$ ways. The graph may not contain a $\bf{T_1}$ directly followed by a $\bf{T_3}$ edge since this would result in two consecutive vertices with the same $n$-label or $p$-labels. Thus there are at most $2^{12\Delta_0 \ell + 1}$ ways to pick which edges are $\bf{T_1}$ and $\bf{T_3}$ amongst the edges that are not $\bf{T_4}$ (we pick $\bf{T_1}$ or $\bf{T_3}$ for a sequence of edges in between two $\bf{T_4}$ edges). This successfully assigns a type to each edge in the graph.

    We define the canonical multi-labelling of a non-backtracking multi-labeling in the exact same way as in Definition \ref{def:canonical_simple}. For the remainder of the proof, we consider the canonical non-backtracking multi-labeling for each equivalence class. For any $\bf{T_4}$ edge $e$ there are at most $L$ ways to select the first occurrence of the $n$-label incident to $e$ and at most $(dL)^d$ ways to select the first occurrence of each of the $p$-labels of the $p$-vertex incident to $e$. Since we consider canonical non-backtracking multi-labelings we have no additional choices for the labels of $\bf{T_1}$ edges.

    The equivalence class is now uniquely determined up to $\bf{T_3}$ edges with order $\tau>1$. By Proposition \ref{prop:T3choice_col}, $\tau$ is no greater than $24\Delta_0\ell+1$. Since there are at most $L$ $\bf{T_3}$ edges in this graph isomorphism class, therefore there are at most $L^{24\Delta_0\ell+1}$ ways to select a corresponding $\bf{T_1}$ edge for each $\bf{T_3}$ edge of order $\tau>1$. Multiplying the factors together gives us at most 
    \[
    L^{12\Delta_0\ell}\cdot 2^{12\Delta_0\ell+1}\cdot(dL)^{12d\Delta_0 \ell}\cdot L^{24\Delta_0+1}\leq (C_1L)^{C_2\Delta_0+C_3} 
    \]
    isomorphism classes of graphs with excess $\Delta_0$. Here, the constants $C_1$, $C_2$, and $C_3$ may depend on $d$ and $\ell$, both of which are finite by construction.
\end{proof}
Now consider a similar result for generalized $(p,n,d)$-non-backtracking multi-labelings defined in Definition \ref{def:generalized_multi_label}.
\begin{prop} \label{prop:gen_back_tracking_bound}
    Suppose $V_1$ is a $p$-vertex. Fix a backtracking set $S$ such that $s_k - s_{k-1} = 2L$ for each $1 \leq k \leq |S| - 1$, and each $s_k$ is odd. 
    \begin{itemize}
        \item[(1)] There are at most $(2L + 1)^{|S|} (C_1 (mL))^{C_2 \Delta_0 + C_3}$ equivalence classes of $mL$-graphs with excess $\Delta_0$, over all generalized $(p,n,d)$-non-backtracking multi-labelings with backtracking set $S$.
        \item[(2)] There are at most $12 \Delta_0 \ell$ $\mathbf{T_4}$ edges in each such equivalence class.
        \item[(3)] There are at least $mL-24\Delta_0\ell$ $\bf{T_1}$ edges between vertices $(V_k, V_{k+1})$ where $V_{k+1}$ is a $p$-vertex.
    \end{itemize}
\end{prop}
\begin{proof}
By defintion, $|S|\leq m$. Observe that all the vertices in $S$ are $p$-vertices. We note that at each vertex $V_{s_k}\in S$ there is a maximal backtracking sequence of length $q(k)$. That is, for the vertices $V_{s_k-q(k)}, \dots, V_{s_k+q(k)}$ each pair $V_{s_k-q}$ and $V_{s_k+q}$ have the same set of vertex labels for $0\leq q\leq q(k)$ and $V_{s_k-q(k)-1}$ and $V_{s_k+q(k)+1}$ have different vertex labels. 

We now describe a procedure for removing a backtracking sequence from some equivalence class of multi-labelings $\mathcal G$ with backtracking set $|S|$. We pick some backtracking vertex $V_{s_k}\in S$. Consider some backtracking sequence $V_{s_k-q}, \dots, V_{s_k+q}$ for some $q\leq q(k)$. Since each edge in the set $(V_{s_k-q}, V_{s_k-q+1}), \dots, (V_{s_k+q-1}, V_{s_k+q})$ has a corresponding pair if we remove all such edges from the graph, the result will still satisfy Condition $(ii)$ of \ref{def:col_multi-labelling}. Since, $V_{s_k}-q$ and $V_{s_k}+q$ have the same labels we are left with an equivalence class $\mathcal G'$ of a $(mL-q)$-graph whose vertices are given by $V_1, \dots, V_{s_k-q}, V_{s_k+q+1}, \dots, V_{2mL-1}$. Let $S'$ denote the set of backtracking vertices of $\mathcal G'$. Observe that $q=q(k)\iff V_{s_k-q}\notin S'$. So, $S'=S-V_{s_k}$ if $q=q(k)$ and $S'=S-V_{s_k}+V_{s_k-q}$ otherwise. 

Let us compare $\Delta(\mathcal G)$ and $\Delta(\mathcal G')$. Observe that $\Delta(\mathcal G)-\Delta(\mathcal G')\geq r(\mathcal G)+c(\mathcal G)/\ell$ where $r(\mathcal G)$ and $c(\mathcal G)$ denote the number of $n$-labels and $p$-labels amongst vertices $V_{s_k-q(k)}, \dots, V_{s_k+q(k)}$ that are also $n$-labels or $p$-labels in $\mathcal G'$, respectively. Further observe that the number of $\bf{T_4}$ edges that are in $\mathcal G$, but not $\mathcal G'$ is upper bounded by $r(\mathcal G)+c(\mathcal G)$. This is because all other edges that belong to $\mathcal G$, but not $\mathcal G'$ and are not incident to a vertex in $r(\mathcal G)$ or $c(\mathcal G)$ are $\mathbf{T_1}$ or $\mathbf{T_3}$.

Now let us induct on $|S|$. When $|S|=0$, the result (1) follows from Proposition \ref{prop:num_graphs_col} and (2) follows from Proposition \ref{prop:col_t4}. Consider two cases. First suppose $q(1)\leq 2L$. Then consider the generalized equivalence class $\mathcal G'$ of multi-labelings of a $mL-q(1)$ graph obtained by removing the maximal backtracking sequence between $V_{s_1-q(1)}, \dots, V_{s_1+q(1)}$. So, that we are left with the graph with vertices $V_1, \dots, V_{s_1-q(1)}=V_{s_1+q(1)}, V_{s_1+q(1)+1}, \dots, V_{2mL}$. Further observe that the backtracking set $S'$ has size $|S'|=|S|-1$ and still satisfies the condition that consecutive indices in $\mathcal S'$ differ by $2L$ and indices in $\mathcal S'$ are odd. By the IH, the number of such classes $\mathcal G'$ is upper bounded by $(2L)^{|S|-1}(C_1(mL-q(1)))^{C_2\Delta(\mathcal G')+C_3}$. We had at most $2L+1$ choices for $q(1)$. Additionally, for labels in $r(\mathcal G)$ and $c(\mathcal G)$ we have $2mL$ and $2mL\ell$ choices to pick which label they are equal to in $\mathcal G'$. Thus, we conclude the number of generalized equivalence classes is no greater than \[(2L+1)(2L+1)^{|S|-1}(C_1(mL-q(1)))^{C_2\Delta(\mathcal G')+C_3}((2mL\ell)^{\ell})^{r(\mathcal G)+c(\mathcal G)/\ell}
\]
And since $\Delta(\mathcal G)-\Delta(\mathcal G')\geq r(\mathcal G)+c(\mathcal G)/\ell$ the result $(1)$ follows. $(2)$ also follows immediately by the inductive hypothesis since the number of $\bf{T_4}$ edges in $\mathcal G$, but not $\mathcal G'$ is upper bounded by $(\Delta(\mathcal G)-\Delta(\mathcal G'))\ell$.

Next suppose $q(1)> 2L$. Consider the generalized multi-labeling of a $(m-2)L$-graph $\mathcal G'$ obtained removing the backtracking sequence $V_{s_1-2L}, V_{s_1-2L+1}, \dots, V_{s_1+2L}$. We are let with vertices $V_1, \dots, V_{s_1-2L}=V_{s_1+2L}, V_{s_1+2L+1}, \dots, V_{2mL-1}$. Observe that the backtracking set $S'$ is $V_{s_1+2L}, V_{s_1+4L}, \dots$, so that $|S'|=|S|-1$. The rest of the argument is the same as case 1) except we don't have the extra factor of $2L+1$ for choosing $q(1)$.

For (3) using (2) there are at most $12\Delta(\mathcal G)\ell$ $\bf{T_4}$ edges in a generalized multi-labeling $\mathcal G$. Suppose that $(V_{k-1}, V_{k})$ is a $\bf{T_1}$ edge and $V_{k}$ is a $n$-vertex. Then by Condition (i) of Definition \ref{def:col_multi-labelling} the edge $(V_{k}, V_{k+1})$ is either $\bf{T_1}$ or $\bf{T_4}$. The number of $\bf{T_1}$ and $\bf{T_3}$ edges in $\mathcal G$ is equal. Then consider the set of edges $E$ of $\mathcal G$ that are not $\bf{T_4}$, and the edge directly following $e\in E$ is not $\bf{T_4}$. $E$ has at least $4mL-24\Delta(\mathcal G)\ell$ such edges. At least $2mL-24\Delta(\mathcal G)\ell$ edges in $E$ are $\bf{T_1}$ (since there are a equal number of $\bf{T_1}$ and $\bf{T_3}$ edges in the graph). Now suppose $(V_{k-1}, V_k)$ is a $\bf{T_1}$ edge in $E$ and $V_k$ is a $n$-vertex, by assumption the edge $(V_{k}, V_{k+1})$ is also $\bf{T_1}$. Hence we conclude there are at least $mL-12\Delta(\mathcal G)\ell$ $\bf{T_1}$ edges that are between vertices $(V_k, V_{k+1})$ where $V_{k+1}$ is a $p$-vertex.
\end{proof}


\bibliographystyle{apalike}
\bibliography{kernel.bib}

\begin{thebibliography}{}

\bibitem[Abbe et~al., 2015]{abbe2015reed}
Abbe, E., Shpilka, A., and Wigderson, A. (2015).
\newblock Reed-muller codes for random erasures and errors.
\newblock In {\em Proceedings of the forty-seventh annual ACM symposium on
  Theory of Computing}, pages 297--306.

\bibitem[Anari et~al., 2018]{anari2018smoothed}
Anari, N., Daskalakis, C., Maass, W., Papadimitriou, C., Saberi, A., and
  Vempala, S. (2018).
\newblock Smoothed analysis of discrete tensor decomposition and assemblies of
  neurons.
\newblock {\em Advances in neural information processing systems}, 31.

\bibitem[Anderson et~al., 2014]{anderson2014more}
Anderson, J., Belkin, M., Goyal, N., Rademacher, L., and Voss, J. (2014).
\newblock The more, the merrier: the blessing of dimensionality for learning
  large gaussian mixtures.
\newblock In {\em Conference on Learning Theory}, pages 1135--1164. PMLR.

\bibitem[Bai and Yin, 1993]{bai_yin}
Bai, Z.~D. and Yin, Y.~Q. (1993).
\newblock {Limit of the Smallest Eigenvalue of a Large Dimensional Sample
  Covariance Matrix}.
\newblock {\em The Annals of Probability}, 21(3):1275 -- 1294.

\bibitem[Belkin and Niyogi, 2003]{belkin2003laplacian}
Belkin, M. and Niyogi, P. (2003).
\newblock Laplacian eigenmaps for dimensionality reduction and data
  representation.
\newblock {\em Neural computation}, 15(6):1373--1396.

\bibitem[Bhaskara et~al., 2014]{bhaskara2014smoothed}
Bhaskara, A., Charikar, M., Moitra, A., and Vijayaraghavan, A. (2014).
\newblock Smoothed analysis of tensor decompositions.
\newblock In {\em Proceedings of the forty-sixth annual ACM symposium on Theory
  of computing}, pages 594--603.

\bibitem[Cheng and Singer, 2013]{doi:10.1142/S201032631350010X}
Cheng, X. and Singer, A. (2013).
\newblock The spectrum of random inner-product kernel matrices.
\newblock {\em Random Matrices: Theory and Applications}, 02(04):1350010.

\bibitem[Deshpande et~al., 2016]{deshp2016sparse}
Deshpande, Y., Montanari, A., et~al. (2016).
\newblock Sparse pca via covariance thresholding.
\newblock {\em Journal of Machine Learning Research}, 17(141):1--41.

\bibitem[Dubova et~al., 2023]{dubova2023universality}
Dubova, S., Lu, Y.~M., McKenna, B., and Yau, H.-T. (2023).
\newblock Universality for the global spectrum of random inner-product kernel
  matrices in the polynomial regime.
\newblock {\em arXiv preprint arXiv:2310.18280}.

\bibitem[Fan and Montanari, 2019]{Fan2019}
Fan, Z. and Montanari, A. (2019).
\newblock The spectral norm of random inner-product kernel matrices.
\newblock {\em Probability Theory and Related Fields}, 173(1):27--85.

\bibitem[Hastie et~al., 2022]{hastie2022surprises}
Hastie, T., Montanari, A., Rosset, S., and Tibshirani, R.~J. (2022).
\newblock Surprises in high-dimensional ridgeless least squares interpolation.
\newblock {\em Annals of statistics}, 50(2):949.

\bibitem[Louart et~al., 2018]{louart2018random}
Louart, C., Liao, Z., and Couillet, R. (2018).
\newblock A random matrix approach to neural networks.
\newblock {\em The Annals of Applied Probability}, 28(2):1190--1248.

\bibitem[{Lu} and {Yau}, 2022]{2022arXiv220506308L}
{Lu}, Y.~M. and {Yau}, H.-T. (2022).
\newblock {An Equivalence Principle for the Spectrum of Random Inner-Product
  Kernel Matrices with Polynomial Scalings}.
\newblock {\em arXiv e-prints}, page arXiv:2205.06308.

\bibitem[Pennington and Worah, 2017]{pennington2017nonlinear}
Pennington, J. and Worah, P. (2017).
\newblock Nonlinear random matrix theory for deep learning.
\newblock {\em Advances in neural information processing systems}, 30.

\bibitem[Rahimi and Recht, 2007]{rahimi2007random}
Rahimi, A. and Recht, B. (2007).
\newblock Random features for large-scale kernel machines.
\newblock {\em Advances in neural information processing systems}, 20.

\bibitem[Rudelson, 2012]{rudelson2012row}
Rudelson, M. (2012).
\newblock Row products of random matrices.
\newblock {\em Advances in Mathematics}, 231(6):3199--3231.

\bibitem[Rudelson and Vershynin, 2009]{rudelson2009smallest}
Rudelson, M. and Vershynin, R. (2009).
\newblock Smallest singular value of a random rectangular matrix.
\newblock {\em Communications on Pure and Applied Mathematics: A Journal Issued
  by the Courant Institute of Mathematical Sciences}, 62(12):1707--1739.

\bibitem[Vershynin, 2010]{vershynin2010introduction}
Vershynin, R. (2010).
\newblock Introduction to the non-asymptotic analysis of random matrices.
\newblock {\em arXiv preprint arXiv:1011.3027}.

\bibitem[Vershynin, 2020]{vershynin2020concentration}
Vershynin, R. (2020).
\newblock {Concentration inequalities for random tensors}.
\newblock {\em Bernoulli}, 26(4):3139 -- 3162.

\bibitem[Weyl, 1912]{Weyl1912}
Weyl, H. (1912).
\newblock Das asymptotische verteilungsgesetz der eigenwerte linearer
  partieller differentialgleichungen (mit einer anwendung auf die theorie der
  hohlraumstrahlung).
\newblock {\em Mathematische Annalen}, 71(4):441--479.

\end{thebibliography}
\newpage
\appendix
\section{Proof of Lemma \ref{lem:lemma_k_x_expansion}}
\label{proof_kernel_expansion}

Consider any kernel function $k(x)$, that is analytic and satisfies $|k(x)|\leq A e^{\al |x|}$ for all $|x|\ge t$ and constants $t>\frac{(\ell+1)}{\alpha},\;0\leq \al \leq 1,\; A>0$ and $\ell \in \N$.
The following lemma gives an estimate for the Taylor coefficients of $k(x)$
\begin{lem}
Let us denote the Taylor expansion of $k(x)$ as
\begin{align}\label{e:Taylor}
    k(x)=\sum_{i\geq 0}b_i x^i.
\end{align}
Then the Taylor coefficients satisfy the following.
\begin{align}
    \quad|b_i|\leq \frac{(e\al)^i}{i^i},\quad \mbox{for all $i\geq \lceil t\alpha\rceil$}.
\end{align}
\end{lem}
\begin{proof}
    The Taylor coefficients can be obtained from the contour integrals
    \begin{align}
        b_i=\frac{k^{(i)}(0)}{i!}
        =\frac{1}{2\pi i}\oint_{|z|=R}\frac{k(z)}{z^{i+1}}d z.
    \end{align}
    where $R$ is a large constant that we will choose later. If $R>t$, then by taking the absolute value on both sides, we have
    \begin{align}
        |b_i|\leq \frac{1}{2\pi}\oint_{|z|=R}\frac{Ae^{\al R}}{R^{i+1}}|dz|=\frac{Ae^{\al R}}{R^i}.
    \end{align}  
If $i\geq \lceil t\alpha\rceil$, we can minimize the right-hand side by taking $R=i/\al$, and
\begin{align}
   |b_i|\leq\frac{A(e\al)^i}{i^i}.
\end{align}
\end{proof}

For $(z_1, z_2, \cdots, z_p)\in \{-1,1\}^p$, since, $k\left(\frac{z_1+z_2+\cdots+z_p}{\sqrt p}\right)$ is a symmetric function in $z_1, z_2,\cdots, z_p$, proceeding as in the proof of Lemma \ref{lem:lemma_k_x_expansion}, we can expand it as
\begin{align}\label{e:decompose}
    k\left(\frac{\sum_{i}z_i}{\sqrt p}\right)
    &=a_0+\frac{a_1}{\sqrt p}\sum_i z_i +\frac{a_2}{p^2}\sum_{i_1<i_2} z_{i_1}z_{i_2}+\cdots +\frac{a_p}{p^{p/2}}z_1z_2\cdots z_p\\
    &=\sum_{d=0}^p\frac{a_d}{p^{d/2}}\sum_{i_1<i_2<\cdots<i_d}z_{i_1}z_{i_2}\cdots z_{i_d}
\end{align}
Let us also observe that the coefficients $a_d$ defined in Lemma \ref{lem:lemma_k_x_expansion}, can also be recovered by solving the equations
\begin{align}\label{e:adexp}
        \frac{a_d}{p^{d/2}}=\E\left[k\left(\frac{\sum_{i}Z_i}{\sqrt p}\right)Z_1Z_2\cdots Z_d\right],
    \end{align}
where $Z_1,Z_2,\cdot, Z_p$ i.i.d. Bernoulli random variables with $\P(Z_i=\pm 1)=1/2$.
    
The following lemma gives the estimates of the coefficients $a_i$.

\begin{lem}
\label{lem:bd_on_coeff}
  Consider a kernel $k(\cdot)$ that is analytic and satisfies \eqref{eq:kernel_bound}. Then, the coefficients $a_d$ in \eqref{e:decompose} are bounded as follows:
  \begin{align}
      |a_d|\leq (3A\sqrt{2e\pi})\sqrt{d+1}\al^d, \quad \mbox{for all $d \ge \lceil t\alpha\rceil$.}
  \end{align}
\end{lem}
\begin{proof}
    By plugging the Taylor expansion \eqref{e:Taylor} into \eqref{e:adexp}, we get
    \begin{align}\begin{split}\label{e:adexp2}
        \frac{a_d}{p^{d/2}}
        &=\sum_{i\geq 0} \frac{1}{p^{i/2}}\E\left[\left(\sum_{j=1}^p Z_j\right)^iZ_1Z_2\cdots Z_d\right]\\
        &=\sum_{i\geq 0} \sum_{j_1,j_2,\cdots, j_i}\frac{b_i}{p^{i/2}}\E\left[Z_{j_1}Z_{j_2}\cdots Z_{j_i}Z_1Z_2\cdots Z_d\right]
    \end{split}\end{align}
The expectation on the right-hand side of \eqref{e:adexp2} is nonzero if $\{1,2,\cdots,d\}\subset\{j_1,j_2,\cdots, j_i\}$ (there are $i!/(i-d)!$ choices) and the remaining indices pair up (there are $(i-d-1)!! p^{(i-d)/2}$ choices). This gives an upper bound of the choices of $j_1,j_2,\cdots,j_i$ as
\begin{align}
    \frac{k!}{(k-i)!} (k-i-1)!! p^{\frac{(k-i)}{2}}.
\end{align}
For all $d \ge \ell$, this gives an upper bound for $a_d$ as
\begin{align*}
    |a_d|&\leq \sum_{i\geq d\atop 2|i-d}|b_i|\frac{k!}{(k-d)!!}
    \leq \sum_{i\geq d\atop 2|i-d}\frac{A(e\al)^i}{i^i}\frac{i!}{(i-d)!!}\\
    &\leq \sum_{i\geq d\atop 2|i-d}\sqrt{2e \pi (i+1)}\frac{A\al^i}{(i-d)!!}
    =\sum_{m\geq 0}\sqrt{2e \pi (2m+d+1)}\frac{A\al^{d+2m}}{(2m)!!}\\
    &\leq 3A\sqrt{2e\pi}\sqrt{d+1}\al^d,
\end{align*}
where we used that $i!\leq \sqrt{2e\pi (i+1)}(i/e)^i$ for $i\geq 0$.
\end{proof}

\begin{proof}[Proof of Lemma \ref{lem:lemma_k_x_expansion}]
Since $k(\cdot)$ is analytic, it admits an expansion in terms of a Taylor series. In other words, there exists constants $\{b_d:d \in [m]\}$ such that
\begin{align}
    k(x) = \sum\limits_{d=0}^{\infty}b_dx^d, \nonumber
\end{align}
for all $x \in \R$. Thus, expanding the terms of the matrix $[\kerK(X)]_{ij}$ for $(i,j) \in [n] \times [n]$ such that $i \neq j$, we get 
\begin{align}
    [\kerK(X)]_{ij}=\sum\limits_{d=0}^{\infty}\frac{b_d}{\sqrt{np^d}}\sum\limits_{k_1,\ldots,k_d \in [p]}(X_{ik_1}X_{ik_2}\cdots X_{ik_d})(X_{jk_1}X_{jk_2}\cdots X_{jk_d}).\nonumber
\end{align}
Since the entries $X_{ij} \in \{\pm 1\}$ for $i \in [n]$ and $j \in [p]$, we have 
\[
X^m_{ij}=\begin{cases}
   1 & \mbox{if $m$ is even,}\\
   X_{ij} & \mbox{otherwise}.
\end{cases}
\]
Now, observe that 
\begin{align}
\label{eq:sum_with_powers}
&\sum\limits_{k_1,\ldots,k_d \in [p]}(X_{ik_1}X_{ik_2}\cdots X_{ik_d})(X_{jk_1}X_{ik_2}\cdots X_{jk_d})\nonumber\\
&\hskip 4em=\sum\limits_{\substack{t_1+\cdots+t_d=d\\t_1,\ldots,t_d \in [d]\\k_1\neq k_2 \neq \ldots\neq k_d\\k_r \in [p]:\,t_r \neq 0,\\r \in [d]}}(X^{t_1}_{ik_1}X^{t_2}_{ik_2}\cdots X^{t_d}_{ik_d})(X^{t_1}_{jk_1}X^{t_2}_{ik_2}\cdots X^{t_d}_{jk_d}).
\end{align}
Let us observe that if $t_r$ is odd for some $r$, then the term $X^{t_r}_{ik_r}$ collapses to $X_{ik_r}$ leaving an excess factor of $p^{(t_r-1)/2}$ in the denominator. Similarly, if $t_r$ is even, then the term $X^{t_r}_{ik_r}$ collapses to one, leaving a factor of $p^{(t_r-2)/2}$ in the denominator. Therefore, the only terms that contribute significantly to the sum in \eqref{eq:sum_with_powers} are the tuples $(t_1,\ldots,t_d)$ where $t_k \in \{1,2\}$ for all $k \in [d]$. It is also easy to observe that all $d > p$ collapses to $(X^{t_1}_{ik_1}X^{t_2}_{ik_2}\cdots X^{t_r}_{ik_r})(X^{t_1}_{jk_1}X^{t_2}_{ik_2}\cdots X^{t_r}_{jk_r})$ for some $r\le p$. Thus, by pairing up the terms with $t_r=2$ and using elementary combinatorics to count such pairs, we can re-express the sum \eqref{eq:sum_with_powers} as
\begin{align}
    &\sum_{d=0}^\infty \frac{b_d}{\sqrt{n p^d}}\sum\limits_{k_1,\ldots,k_d \in [p]}(X_{ik_1}X_{ik_2}\cdots X_{ik_d})(X_{jk_1}X_{ik_2}\cdots X_{jk_d})\nonumber\\
&\hskip 2em=\sum_{d=0}^{p}\sum\limits_{\substack{t=0\\2|(d-t)}}^d\frac{b_d}{\sqrt{np^t}}\frac{d\,!}{\left(\frac{d-t}{2}\right)!\;2^{(d-t)/2}}\sum_{\substack{k_1<\ldots<k_t\\k_1,\ldots,k_t \in [p]}}(X_{ik_1}X_{ik_2}\cdots X_{ik_t})(X_{jk_1}X_{ik_2}\cdots X_{jk_t})\nonumber\\
&\hskip 2em=\sum_{t=0}^{p}\left\{\sum\limits_{\substack{d\ge t\\2|(d-t)}}^d\frac{b_d}{\sqrt{np^t}}\frac{d\,!}{\left(\frac{d-t}{2}\right)!\;2^{(d-t)/2}}\right\}(1+o(1))\sum_{\substack{k_1<\ldots<k_t\\k_1,\ldots,k_t \in [p]}}(X_{ik_1}X_{ik_2}\cdots X_{ik_t})(X_{jk_1}X_{ik_2}\cdots X_{jk_t})\nonumber
\end{align}
It can be verified using integration by parts that if $k(x)=\sum_{d=0}^{\infty}b_dx^d$, then for all $0 \le t \le p$,
\[
\sum\limits_{\substack{d\ge t\\2|(d-t)}}^p\frac{b_d}{\sqrt{np^t}}\frac{d!}{2^{(d-t)/2}\left(\frac{d-t}{2}\right)!} = \sqrt{t!}\,\mathbb E[k(\xi)H_t(\xi)], \quad \mbox{for $\xi \sim \dN(0,1)$,}
\]
where $H_t(\cdot)$ is the $t$-th Hermite polynomial. Hence, \eqref{eq:kernel_expansion} holds. The bound on the coefficients $a_d$ follows from Lemma \ref{lem:bd_on_coeff}.
\end{proof}

\section{Proofs of results from Section \ref{sec:spec_b}}
\label{b_x}

\subsection{Proof of Lemma \ref{lem:row_positive_excess}}
    We shall induct on $L$ to prove the lemma. For the base case, consider $L=2$ and $L=3$. For $L=2$, we must have $d_1=d_2$ by Lemma \ref{lem:l_2_3}. Furthermore, we must also have $r(\mathcal M)=2$. Hence, $r(\mathcal M) +\frac{c(\mathcal M)}{\ell} = 2+\frac{d_1}{c}=1+\sum_{i=1}^\ell\frac{d_1}{2\ell}+1$. Next, for $L=3$, we must have $d_1=d_2=d_3$, and $r(\mathcal M)\leq 3$, and consequently $r(\mathcal M)+\frac{c(\mathcal M)}{\ell}\leq 3 +\frac{d_1}{\ell}\leq \frac{3}{2}+\frac{3d_1}{2\ell}+1$, where the second inequality holds as $d_1\geq \ell$.

    Now, assume by induction that the result holds for $L-2$ and $L-1$ for all $L\geq 4$. If each distinct $n$-label appears twice there are at most $L/2$ distinct $n$-labels. This means that $r(\mathcal M)\leq \frac{L}{2}$, and by Lemma \ref{lem:Nj_2} we have $c(\mathcal M)\leq \sum_{i=1}^L\frac{d_i}{2}$, which gives $r(\mathcal M)+\frac{c(\mathcal M)}{\ell}\leq \frac{L}{2}+\frac{\sum_{i=1}^L d_i}{2\ell}+1$.

    Now, suppose there exists some $n$-vertex $V$ with a $n$-label that appears exactly once. Define the vertices $T$, $U$, $W$, $X$ as in Lemma \ref{lem:graph_induction}. If $T$ and $X$ have different $n$-labels then we follow the procedure (1) in Lemma \ref{lem:graph_induction} to obtain a new graph $(L-1)$-graph with valid row multi-labelling $\mathcal M'$. By the inductive hypothesis $r(\mathcal M')+\frac{c(\mathcal M')}{\ell}\leq \frac{L}{2}+\frac{\sum_{i=1}^\ell d_i}{2\ell}+1$ for the $(L-1)$-graph. The vertices $U$ and $W$ have the same tuples of $p$-labels (up to reordering). If $U$ and $W$ don't have the same $p$-label then there will be some edge between a $p$-label of $U$ or $W$ and the $n$-label of $V$ that appears once, which is in contradiction to condition (iii) of Definition \ref{def:multi-labelling}. Let $d_U=d_W$ be the number of $p$-labels of $p$-vertices $U$ and $W$. We have $c(\mathcal M')=c(\mathcal M)$, while $r(\mathcal M')=r(\mathcal M)-1$. Since, $d_U\geq \ell$ and $1\leq \frac{1}{2}+\frac{d_U}{2\ell}$ the inductive step follows.

    If $T$ and $X$ have the same $n$-labels, we follow procedure $(2)$ of Lemma \ref{lem:graph_induction}. We obtain a new $(L-2)$-graph with multi-labelling $\mathcal M'$. Again by the inductive hypothesis, we have $r(\mathcal M')+\frac{c(\mathcal M')}{\ell}\leq \frac{L}{2}+\frac{\sum_{i=1}^L d_i}{2\ell}+1$ for the $(L-2)$-graph. Define $d_U=d_W$ to be the number of $p$-labels of the vertices $U$ and $W$. Then we have $c(\mathcal M)-d_U\leq c(\mathcal M')$ and $r(\mathcal M)-1= r(\mathcal M')$. Since $1+\frac{d_U}{\ell}= \frac{2}{2}+\frac{2d_U}{2\ell}$, this implies the inductive step.

\subsection{Proof of Lemma \ref{lem:simple_positive_excess}}
We shall prove this lemma by inducting on the values of $L$. For the base case, consider $L=2$ or $L=3$. Suppose $L=2$ or $L=3$ and there are no empty labels in the graph. Then, we can use Lemma \ref{lem:l_2_3}, which is an equivalent statement to show the induction hypothesis.

If $L=2$ and there are empty labels in the graph then both $p$-vertices have empty labels by condition $(iv)$ of Definition \ref{def:simple_labelling}, so $\tilde k = 0$, $\tilde r = 2$, and $\tilde c = 0$, which proves the claim. If $L = 3$ and there are empty labels in the graph then by condition $(iv)$ of Definition \ref{def:simple_labelling}, we conclude there must be exactly two empty labels. However, by condition $(iii)$ of Definition \ref{def:simple_labelling} the $n$-labels of the $n$-vertices preceding and following the only $p$-vertex with non-empty label must be the same, which contradicts condition $(i)$ of Definition \ref{def:simple_labelling}. Thus if $L=3$, there cannot be empty labels in the graph. This proves the base case.

Now for the inductive step suppose an $L$-graph has each $n$-label appearing at least twice. Then $\tilde r \leq \frac{\ell}{2}$. Suppose that there is some non-empty $p$-label that appears only once. Let $V$ denote the $p$-vertex containing such a $p$-label $j$. Let $U$ and $W$ be the $n$-vertices preceding and following $V$ in the $L$-graph, and their $n$-labels be $i_U$ and $i_W$. By Definition \ref{def:simple_labelling}, we know that $i_U\ne i_W$, which implies $(i_U, j)$ and $(i_W, j)$ appear only once in the row multi-labelling. This contradicts condition $(iii)$ of Definition \ref{def:simple_labelling}. Thus, the number of non-empty $p$-labels $\tilde c$ is no greater than $\frac{\ell\cdot \tilde k}{2}$ since each $p$-vertex has $\ell$ labels. Therefore, we can conclude that $\tilde r+\frac{\tilde c}{\ell}\leq \frac{L+\tilde k}{2}$.

Now, suppose there is some $n$-label appearing only once in the $L$ graph, and let $V$ be the $n$-vertex with this label. Let the two $p$-vertices preceding and following it be $U$ and $W$ respectively, and suppose both have non-empty $p$ labels. Then we can use the same argument as in the proof of the Lemma \ref{lem:row_positive_excess}. Note that by condition $(iv)$ of Definition \ref{def:simple_labelling} it is not possible that exactly one of $U$ and $W$ has an empty label. So we consider the case when both $U$ and $W$ have the same label. Let $T$ and $X$ be the $n$-vertices preceding $U$ and following $W$ respectively. Note that by condition $(iv)$ of Definition \ref{def:simple_labelling}, $T$ and $X$ must have the same $n$-label. Thus the $(L-2)$-graph obtained by removing the four edges $(T, U), (U, V), (V, W)$ and $(W, X)$ has a valid $(p,n,\ell)$ simple-labelling. By the inductive hypothesis the $(L-2)$-graph satisfies $\tilde \Delta \geq 0$. In removing $(T, U), (U, V), (V, W)$ and $(W, X)$ the value of $\tilde r$ decreases by $1$, while $\tilde c$ and $\tilde k$ are fixed. So using the inductive hypothesis $\tilde \Delta\geq 0$ also holds for the $L$-graph.

\subsection{Proof of Proposition \ref{prop:invert_mapping_2}}\label{s:B3}
    Note that for any $\mathcal M\in \mathcal C_0$ since $\Delta(\mathcal M)=0$ by Lemma \ref{lem:tree_excess_0} we know that the graph $\mathcal G$ containing the multi-labelling $\mathcal M$ is a tree with each edge appearing twice, and no pair of $n$-vertices or $p$-vertices share a common label in $\mathcal G$. 
    We first verify that indeed $\tilde \Delta(\tilde{\mathcal M})=0$. Observe that, since $\mathcal M$ is a tree with $L$ edges, it contains $L+1$ vertices. Furthermore, by definition $\tilde{\mathcal M}$ is also a tree with $L+1$ vertices. We know that the number of empty labels in $\tilde{\mathcal M}$ is $L-\tilde k(\tilde{\mathcal M})$ and hence by $(4)$ in Lemma \ref{lem:tree_excess_0} the number of empty vertices in the tree for $\tilde {\mathcal M}$ is $\frac{L-\tilde k(\tilde{\mathcal M})}{2}$. Thus, the number of nonempty vertices in $\tilde {\mathcal M}$ is $L+1-\frac{L-\tilde k(\tilde{\mathcal M})}{2}=\frac{L+\tilde k(\tilde{\mathcal M})}{2}+1$. Since each nonempty $p$-vertex has $\ell$ labels, we conclude that $\tilde \Delta(\tilde{\mathcal M})=0$.

    Now, let us consider the set $\phi^{-1}(\tilde{\mathcal M})$ for some $\tilde {\mathcal M}\in\tilde{\mathcal C_0}$. Note that each $\mathcal M\in \phi^{-1}(\tilde{\mathcal M})$ has the same $n$-vertices and $n$-labels as $\tilde{\mathcal M}$. Let $S$ be the set of $p$-vertices mapping to empty labels under $\phi$ in $\mathcal M$ and $\mathcal D$ be the set of multi-labellings $\mathcal M\in \phi^{-1}(\tilde{\mathcal M})$ where we fix the number of $p$-labels corresponding to the $p$-vertices having more than $p$-labels to be $d_1, \dots, d_{|S|}$ with $\ell+1\leq d_r\leq L_1$. Now consider the term
    \[
    \sum_{\mathcal M\in \mathcal D}\prod_{s=1}^L\frac{|a_{d_s}(\mathcal M)|}{(d_s(\mathcal M)!)^{1/2}} = \left(\frac{|a_{\ell}|}{(\ell!)^{1/2}}\right)^{\tilde k(\tilde{\mathcal M})}\sum_{\mathcal M\in D}\prod_{r=1}^{|S|}\frac{|a_{d_r}(\mathcal M)|}{(d_r(\mathcal M)!)^{1/2}}
    \]
    Observe that the labels in each of the vertices in $S$, the labels can be permuted in $d_r(\mathcal M)!$ ways. Moreover, since each vertex in $S$ has degree $2$, therefore the term $(d_r(\mathcal M)!)^{1/2}$ appears twice in the denominator for each $r$. From this observation, we can conclude that 
    \[
    \left(\frac{|a_{\ell}|}{(\ell!)^{1/2}}\right)^{\tilde k(\tilde{\mathcal M})}\sum_{\mathcal M\in D}\prod_{r=1}^{|S|}\frac{|a_{d_r}(\mathcal M)|}{(d_r(\mathcal M)!)^{1/2}} = \left(\frac{|a_{\ell}|}{(\ell!)^{1/2}}\right)^{\tilde k(\tilde{\mathcal M})}\prod_{r=1}^{|S|}|a_{d_r}|
    \]
    Now summing over all possible sets $\mathcal D$, we obtain,
    \[
    \sum_{\mathcal M\in\phi^{-1}(\tilde{\mathcal M})}\prod_{s=1}^L\frac{|a_{d_s}(\mathcal M)|}{(d_s(\mathcal M)!)^{1/2}}=\left(\frac{|a_{\ell}|}{(\ell!)^{1/2}}\right)^{\tilde k(\tilde{\mathcal M})}\left(\sum_{r=\ell+1}^{L_1}a_{r}^2\right)^{\frac{L-\tilde k(\tilde{\mathcal M})}{2}} = \left(\frac{|a_\ell|}{(\ell!)^{1/2}}\right)^{\tilde k(\tilde {\mathcal M})}b^{\frac{L-\tilde k(\tilde {\mathcal M})}{2}} 
    \]

\section{Proof of results from Section \ref{tree}}\label{sec:combo_row}

\begin{lem}\label{lem:Nj_2}
    Consider any $j\in [n]$ and let $N^p_j$ denote the number of appearances of $j$ as a $p$-label in a given \textbf{$(p, n, L_1)$-multi-labelling} of a $L$-graph. Then $N_j^p>0$ implies $N_j^p\geq 2$.
\end{lem}

\begin{proof}
    Suppose we find some $j$ in a row multi-labeling with $N_j^p=1$. Let $V$ denote the $p$-vertex containing the $p$-label $j$. Let $U$ and $W$ be the $n$-vertices preceding and following $V$ in the $\ell$-graph, and their $n$-labels be $i_U$ and $i_W$. By Definition \ref{def:col_multi-labelling}, we know $i_U\ne i_W$, which implies $(i_U, j)$ and $(i_W, j)$ appear only once in the row multi-labelling contradicting condition $(iii)$ of Definition \ref{def:col_multi-labelling}. Therefore, for all $j \in [p]$, if $N^p_j>0$ then $N^p_j\ge 2$.
\end{proof}

\begin{lem}\label{lem:l_2_3}
    If $L=2$ or $L=3$, then for any \textbf{$(p, n, L_1)$-multi-labelling} of the $L$-graph, all $n$-labels are distinct, and all the $p$-vertices have the same tuple of $p$-labels up to reordering. 
\end{lem}

\begin{proof}
    The proof follows by retracing the steps used to prove Lemma A.3 of \cite{Fan2019}. 
\end{proof}

\begin{lem}\label{lem:graph_induction}
    In an \textbf{$(p, n, L_1)$-multi-labelling} of an $L$-graph with $L\geq 4$, suppose an $n$-vertex $V$ is such that its $n$-label appears on no other $n$-vertices. Let the $p$-vertex preceding $V$ be $U$, the $n$-vertex preceding $U$ be $T$, the $p$-vertex following $V$ be $W$, and the $n$-vertex following $W$ be $X$.

    \begin{itemize}
        \item[(1)] If $T$ and $X$ have different $n$-labels then the graph obtained by deleting $V$ and $W$ and connecting $U$ to $X$ is an $(L-1)$-graph with a valid \textbf{$(p, n, L_1)$-multi-labelling}.
        \item[(2)] If $T$ and $X$ have the same $n$-label, then the graph obtained by deleting $U$, $V$, $W$, and $X$ and connecting $T$ to the $p$-vertex after $X$ is an $(L-2)$-graph with a valid \textbf{$(p, n, L_1)$-multi-labelling}.
    \end{itemize}
\end{lem}

\begin{proof}
    The proof follows by retracing the steps used to prove Lemma A.4 of \cite{Fan2019}. 
\end{proof}

\begin{lem}\label{lem:Nj_3row}
Consider any $j\in [p]$ and let $N^c_j$ denote the number of appearances of $j$ as a $p$-label in a given multi-labeling. Consider any $L$-graph $\mathcal G$ with an \textbf{$(p, n, L_1)$-multi-labelling} $\mathcal M$. Suppose $r(\mathcal M)\leq \frac{L}{2}$. Then the following holds:
    $$\sum_{j : N^c_j\geq 3} N^c_j \leq (6 \Delta(\mathcal M) - 6) \ell $$
\end{lem}

\begin{proof}
    Let $m=|\{j : N^c_j = 2\}|$ and $m'= |\{j : N^c_j\geq 3\}|$. Then note that $c(\mathcal M) = m+m'$, since $N^c_j\ne 1$ for all $j$. Furthermore, observe that $2m+3m'\leq \sum_{i=1}^Ld_i$; which implies $2m+3\,(c(\mathcal M)-m)\leq \sum_{i=1}^Ld_i$ and hence $m\geq 3\,c(\mathcal M)-\sum_{i=1}^Ld_i$. Therefore, we have a lower bound on $m$. Next, using $\sum_{j : N^c_j\geq 3} N^c_j\leq \sum_{i=1}^Ld_i-2m$, we have $\sum_{j : N^c_j\geq 3} N^c_j\leq \sum_{i=1}^Ld_i-2(3\,c(\mathcal M)-\sum_{i=1}^Ld_i)=3\sum_{i=1}^Ld_i-6c(\mathcal M)$. Recall that $r(\mathcal M)\leq \frac{L}{2}$. Thus we have that $\sum_{j : N^c_j\geq 3} N^c_j\leq 3L\cdot \ell+3\sum_{i=1}^Ld_i-6\ell\, r(\mathcal M)-6\,c(\mathcal M) = (6\Delta (\mathcal M)-6)\ell$, which proves the statement.
\end{proof}

\begin{lem}\label{lem:P_ii_42}
    Let $P_{i, i'}$ denote the number of appearances of $i, i'$ as $p$-labels of two consecutive $p$-vertices (in either order). Then for any \textbf{$(p, n, L_1)$-multi-labelling} of an $L$-graph we have that $$\sum_{i<i':P_{i, i'}\geq 3}P_{i, i'}\leq 42\Delta(\mathcal M)\ell$$
\end{lem}

\begin{proof}
    We shall prove this lemma by induction on $L$. For $L=2$ or $L=3$ the result is trivial by Lemma \ref{lem:l_2_3} since all $n$-labels are distinct and  $\Delta \geq 0$.

    Suppose the result holds for $L-2$ and $L-1$ for an $L$-graph when $L\geq 4$. Now let us assume that each $n$-label appears twice. If there is some $p$-label with $N_j^c=2$ then the adjacent $n$-labels satisfy $P_{i, i'}\geq 2$ by $(i)$ and $(iii)$ of Definition \ref{def:multi-labelling}. Thus, the number of pairs $i<i'$ such that $P_{i,i'}=1$ is no greater than the number of $p$-labels $j$ with $N_j^p\geq 3$. This is bounded by $6\Delta(\mathcal M) \ell$ by Lemma \ref{lem:Nj_3row}. The number of consecutive $n$-labels $i, i'$ where $i\ne i'$ must be lower bounded by $r(\mathcal M)-1$, which we can re-write as $r(\mathcal M)-1=\frac{L}{2}+\frac{\sum_{i=1}^L d_i}{2\ell}-\frac{c(\mathcal M)}{c}-\Delta(\mathcal M)$. Note that since $N_j^c\geq 2$, for any $j$ we have $\frac{c(\mathcal M)}{\ell}\leq\sum _{i=1}^L \frac{d_i}{2\ell}$, and so $r(\mathcal M)-1\geq \frac{L}{2}-\Delta(\mathcal M)$. Hence we conclude that $$\sum_{i, i' :P_{i, i'}\geq 2} P_{i, i'} \geq \frac{L}{2}-\Delta(\mathcal M)-6\Delta(\mathcal M) \ell$$ Let $P_2$ denote the number of pairs $i, i'$ such that $P_{i, i'}=2$. Then, we have $2P_2+3(\frac{\ell}{2}-\Delta-6\Delta c-P_2)\leq L$, and consequently $P_2\geq \frac{L}{2}-3\Delta(\mathcal M)-18\Delta(\mathcal M) \ell$. Thus of the $L$ pairs of consecutive $n$-vertices we use at least $L-6\Delta(\mathcal M)-36\Delta(\mathcal M)\ell$ for pairs $i$, $i'$ with $P_{i, i'}=2$. This leaves at most $6\Delta (\mathcal M)+36\Delta(\mathcal M) \ell$ remaining pairs, and hence we obtain $\sum_{i, i':P_{i, i'}\geq 3}P_{i, i'}\leq 42\Delta(\mathcal M) \ell$.

    Now, suppose we have an $n$-vertex $V$ whose $n$-label appears exactly once. We consider the $(L-1)$-graph or $(L-2)$-graph obtained through the corresponding procedure of Lemma \ref{lem:graph_induction}. We claim the sum $\sum_{i, i': P_{i, i'}\geq 3} P_{i, i'}$ is the same for both graphs. Assume for contradiction the vertex $V$ had $n$-label $i$ that was part of a consecutive pair $i, i'$ with $P_{i, i'}\geq 3$. Then it is clear that the $n$-label $i$ must appear in some other $n$-vertex contradicting the claim that the $n$-label of $V$ appeared just once. Thus, by the inductive hypothesis, we obtain the result.
\end{proof}

\subsection{Proof of Proposition \ref{prop:unpaired_edges}}
     Following the argument of the Proof of Lemma \ref{lem:row_positive_excess} we can remove the $n$-vertices that appears once in $\mathcal G$ using Lemma \ref{lem:graph_induction}. As before, the result is a $L'$-graph $\mathcal G'$ (where $L'\leq L$) with multi-labelling $\mathcal M'$ satisfying $\Delta(\mathcal M)\geq \Delta(\mathcal M')$ such that $r(\mathcal M')\geq L'/2$. We claim $\mathcal M'$ has all but at most $4\ell(\Delta(\mathcal M)-\Delta(\mathcal M'))$ of the bad edges of $\mathcal M$. 
     
     If we remove two edges using procedure $(1)$ of Lemma \ref{lem:graph_induction} it is clear that the edges are $\bf{T_1}$ and $\bf{T_3}$, since the $n$-vertex removed has only appeared once in the graph. If we remove four edges using procedure $(2)$ of Lemma \ref{lem:graph_induction} then the two edges incident to the vertex $V$ are again $\bf{T_1}$ and $\bf{T_3}$. Suppose $(T, U)$ and $(W, X)$ are $\bf{T_1}$ and $\bf{T_3}$ edges. Then at least one  $p$-label of $U$ or $W$ appears elsewhere in $\mathcal M$. This means that $\frac{d_V+d_U}{2}$ exceeds the number of $p$-labels removed from $\mathcal M$, which in turn means that the excess will decrease by at least $\frac{1}{2\ell}$. Hence for any two bad edges removed we decrease $\Delta(\mathcal M)$ by at least $\frac{1}{2\ell}$. Therefore, it suffices to bound the number of bad edges in $\mathcal G$.

     Let $(U, V)$ denote some bad edge, and without loss of generality, we assume $V$ is the $p$-vertex. Let $W$ denote the other $n$-vertex adjacent to $V$. Note that the $n$-labels $i_U$ and $i_W$ are distinct from each other. Let $J$ be the tuple of $p$-labels for vertex $V=(j_1, \dots, j_{d_V})$. There must exist some edge between $n$-vertex $U'$ with label $i_{U'}=i_U$ and $p$-vertex $V'$ with labels $J'=(j'_1, \dots, j'_{d_V'})$ containing $j_1$ since the sub-edge between labels $(j_1, i_U)$ must appear elsewhere in the graph. Also, observe that $V$ is not the same vertex as $V'$ as this would imply $U$ and $U'$ are consecutive $n$-vertices with the same label. Now consider the following two cases.
    
    If $J'$ is a permutation of $J$ then we can find $n$-vertex $U''$ and $p$-vertex $V''$ so that $(U', V')$ and $(U'', V'')$ are a good pair. (Otherwise $(U, V)$ and $(U', V')$ are a good pair). Furthermore, $V'$ is not the same vertex $V$ or $V''$ as this would imply either $U$ or $U'$ and $U''$ are consecutive $n$-vertices with the same label. This means $N^c_{j_1}\geq 3$ and hence the total number of such bad edges is bounded by $12\Delta(\mathcal M')\ell$ by Lemma \ref{lem:Nj_3row}.

    Now, let us assume that $J'$ is not a permutation of $J$. Let $W'$ denote the other $n$-vertex incident to $J'$ with label $i_{W'}$. Then we consider the following two subcases.
    
    Case $1$ $(i_{W}=i_{W'})$: Consider some $j$ such that $j\in J$ and $j\notin J'$ (The same argument works if $j\in J'$ and $j\notin J$). Then, $j$ must appear in some other tuple of $p$-labels $J''$ of vertex $V''$, since currently sub-edges $(j, i_U)$ and $(j, i_W)$ have only occurred once $(i_U\ne i_{W})$. If the two $n$-vertices adjacent to $V''$ have $n$-labels equal to $i_{U}$ and $i_{W}$ (in either order), then $P_{i_U, i_W}=3$, and hence the number of such bad edges must be bounded $84\Delta(\mathcal M')\ell$ using Lemma \ref{lem:P_ii_42}. If the $n$-labels of the two vertices adjacent to $V''$ do not have labels $i_U$ and $i_W$, then $j$ must appear in some third tuple of $p$-labels, as at least one of $(j, i_U)$ and $(j, i_W)$ has only appeared once. The number of bad edges this generates is again bounded by $12\Delta(\mathcal M')\ell$ by Lemma \ref{lem:Nj_3row}. Thus the number of bad edges is no greater than $96\Delta(\mathcal M')\ell$ if $i_W=i_{W'}$.

    Case 2 $(i_W\ne i_{W'})$: Recall that $j_1\in J$ and $j_1\in J'$. Then note that the sub-edges $(j_1, i_W)$ must appear elsewhere in the graph. Hence there is some other vertex $V''$ containing $j_1$ as a $p$-label. This implies $N_{j_1}^c\geq 3$ and consequently the number of such bad edges is bounded by $12\Delta(\mathcal M')\ell$ by Lemma \ref{lem:Nj_3row}.

    Thus we conclude the total number of bad edges is bounded by $120\Delta(\mathcal M')\ell+4\ell(\Delta(\mathcal M)-\Delta(\mathcal M'))\le 120\Delta(\mathcal M)\ell$.

\subsection{Proof of Proposition \ref{prop:t4bound}}
    As in Proposition \ref{prop:unpaired_edges}, consider the graph $\mathcal M'$ obtained by removing $n$-vertices that have $n$-labels that appear once in $\mathcal G$. First, let us provide a bound on the number of $\bf{T_4}$ edges in $\mathcal G$ but not $\mathcal G'$. Consider an $n$-vertex $V$ appearing once in $\mathcal G$ and let $U$ and $W$ be the $p$-vertices preceding and following it. Then, by definition $(U, W)$ is an $\bf{T_1}$ edge and $(V, W)$ is an $\bf{T_3}$ edge because $U$ and $W$ must have the same tuple of labels. Thus using procedure $(1)$ of Lemma \ref{lem:graph_induction} we remove no $\bf{T_4}$ edges. 
    
    Now let $T$ be the $n$-vertex that precedes $U$ and $X$ be the vertex that follows $W$. If $(T, U)$ is an $\bf{T_1}$ edges then $(X, W)$ is an $\bf{T_3}$ edge (as, $U$ and $W$ have the same tuple of labels). On the other hand, if $(T, U)$ is not a $\bf{T_1}$ edge then at least one $p$-label $j$ of vertices $U$ and $W$ appears elsewhere in the multi-labeling $\mathcal M$. Suppose $U$ and $W$ both have $d$ $p$-labels. If we remove $(T, U)$ and $(X, W)$ using procedure $(2)$ of Lemma \ref{lem:graph_induction}, we shall be removing at most $d-1$ $p$-labels from $\mathcal M$. Thus for every removed non $\bf{T_1}$ edge, we shall decrease the excess of the resulting graph by at least $\frac{1}{2\ell}$. Hence the number of $\bf{T_4}$ edges in $\mathcal M$ that are not in $\mathcal M'$ is upper bounded by $4(\Delta(\mathcal M)-\Delta(\mathcal M'))\ell$.

    Observe that, by Proposition \ref{prop:unpaired_edges} the number of $\bf{T_4}$ edges that are also bad edges in $\mathcal M'$ is no more than $120\Delta(\mathcal M')\ell$. 

    Since $\Delta(\mathcal M)\geq \Delta(\mathcal M')$, we can conclude the number of $\bf{T_4}$ edges is no bigger than $120\Delta(\mathcal M)\ell$.

\subsection{Proof of Lemma \ref{lem:tree_excess_0}} First consider an $L$-graph and a multilabeling $\mathcal M$ such that $\Delta(\mathcal M)=0$.
    $(1)$: Since $\Delta(\mathcal M)=0$, by Proposition \ref{prop:t4bound} the number of $\bf{T_4}$ edges in $\mathcal G$ is $0$. Hence, $\mathcal G$ only contains only type $\bf{T_1}$ and $\bf{T_3}$ edges. Recall, as defined in Definition \ref{def:edge_types}, that an edge $e_0$ is of type $\bf{T_3}$ between the vertices $(U, V)$ if only if there is a $\bf{T_1}$ edge $e_1$ between vertices $(U', V')$ so that the labels of $(U, V)$ are the same as the labels of $(U', V')$. There may not be another edge $e_2$ in $\mathcal G$ between $(U'', V'')$ so that the label of $U$ equals the label $U''$ and the label of $V$ equals the label of $V''$ as this would imply either $e_0$ or $e_2$ is of type $\bf{T_4}$. By condition $(iii)$ in Definition \ref{def:multi-labelling} for each $\bf{T_1}$ there exists a $\bf{T_3}$ edge so that we can pair all edges into good pairs. Therefore Assertion $(1)$ holds.

    \bigskip

    Proof of $(2)$: Now suppose we pair the corresponding $\bf{T_1}$ and $\bf{T_3}$ edges. Then the remaining graph $\mathcal G'$ has the same vertices and edges as the graph obtained by keeping the vertices of $\mathcal G$ and the $\bf{T_1}$ edges. Note that such a graph is connected since each new $\bf{T_1}$ edge adds an additional edge to the existing connected component. Furthermore, such a graph $\mathcal G'$ is a tree that has distinct $n$-labels for each $n$-vertex and has no two tuples of $p$-labels that share a common $p$-label since visiting any vertex that has a label that has previously been in the graph would result in an edge in $\mathcal G'$ that is not $\bf{T_1}$.

    \bigskip

    Proof of $(3)$: Suppose there is a $p$-vertex $V$ in $\mathcal G'$ with degree $1$. Then there are two edges incident to $V$ in $\mathcal G$ that are paired up. Call the two $n$-vertices incident to these edges $U$ and $U'$. Then $U$ and $U'$ have the same label in $\mathcal M$. This contradicts $(i)$ in Definition \ref{def:multi-labelling}.

    \bigskip

    Proof of $(4)$: First note the degree of $V$ in $\mathcal G'$ is at least $2$ by $(3)$. Suppose the degree of $V$ is at least $3$ in $\mathcal G'$. Then consider the multi-labelling $\mathcal M^*$ obtained be replacing all vertices mapped to $V$ from $\mathcal G$ by a vertex $V^*$ where we simply remove an arbitrary $p$-label of $V$ to get the tuple of $p$-labels of $V^*$. Note $\mathcal M^*$ is still a valid multi-labelling since by $(2)$ the removed label does not appear elsewhere in $\mathcal G$. Note that $c(\mathcal M)=c(\mathcal M^*)+1$. Let $\sum_{s=1}^L d_s$ and $\sum_{s=1}^Ld_s^*$ denote the number of $p$-labels in both graphs. Then $\sum_{s=1}^L d_s\geq \sum_{s=1}^Ld_s^*+3$ since $\deg V\geq 3$ in $\mathcal G'$. This means $\Delta(\mathcal M)>\Delta(\mathcal M^*)\geq 0$, which is a contradiction. 

    Next, consider a multi-labeling $\mathcal M$ satisfying claims (1)-(6). We shall show that $\Delta(\mathcal M)=0$. Observe that since $\mathcal M$ is a tree with $L$ edges, it has $L+1$ vertices. Let $\mathcal S$ by the set of $p$-vertices in $\mathcal G$ with more than $\ell$ $p$-labels and let $\mathcal S^c$ be the set of all other $p$-vertices in $\mathcal G$. Let $c(\mathcal S)$ and $c(\mathcal S^c)$ denote the total number of distinct $p$-labels amongst the $p$-vertices in $\mathcal S$ and $\mathcal S^c$, respectively. Let $d_s$ denote the number of $p$-labels of the $s$th vertex in $\mathcal G$. Then by $(2)$ and $(4)$ we have that $$\sum_{\substack{s=1 \\ s\in \mathcal S}}^L \frac{d_s}{2\ell} = \frac{c(\mathcal S)}{\ell}.$$ The $p$-vertices not in $\mathcal S$ have $\ell$ $p$-labels. Hence $$\sum_{\substack{s=1\\ s\in \mathcal S^c}}^L  d_s= |\mathcal S^c|\ell=(L-|\mathcal S|)\ell.$$ There are $L+1-|\mathcal S|/2$ vertices in the tree $\mathcal G'$ for $\mathcal G$ that do not belong to $\mathcal S$ since each vertex in $\mathcal S$ has degree two by $(4)$. Furthermore, each $p$-vertex in $\mathcal S^c$ has $\ell$ $p$-labels. Therefore, $$\Delta(\mathcal M)=1+\frac{L}{2}+\sum_{\substack{s=1 \\ s\in \mathcal S}}^L \frac{d_s}{2\ell}+\frac{L-|S|}{2}-L-1+\frac{|\mathcal S|}{2}-\frac{c(\mathcal S)}{\ell}=0.$$

\subsection{Proof of Proposition \ref{prop:deltapositivebound}}
  We upper bound the number of isomorphism classes in $\mathcal N_{k, L, L_1}$ mapped to $H$. By Proposition \ref{prop:unpaired_edges}, for any isomorphism class in $\mathcal N_{k, L, L_1}$, there exist at most $120k\ell$ bad edges. Furthermore, by Proposition \ref{prop:t4bound} there are at most $120k\ell$ $\bf{T_4}$ edges. Let us pick some $G\in \Psi_k^{-1}(H)$. Then, we can see that Step 1 of $\Psi_k(G)$ removes no more than $240k\ell$ edges.
  
  Thus, amongst the edges of $H$, at least $L-720k\ell$ belong to the good chain $A$ produced after Step 3 of the construction of $\Psi_k(G)$. This is because, we have removed at least one edge in Step 1 for each pair of edges removed in Step 3. Hence, we have no more than $(2L)^{720k\ell}$ ways to select the edges of $H$ that are not in the good chain $A$. For each such edge not in the good chain, we can select the first occurrence of its $n$-label in at most $L$ ways and the first occurrence of its $p$-labels in no more than $\sum_{i=\ell}^{L_1} {2L_1 L\choose i}\leq L_1\cdot \frac{(2L_1L)^{L_1}}{L_1!}\leq L_1\cdot (2eL)^{L_1}\leq (4eL)^{L_1}$ ways. Here, we have used that $L_1!\geq(\frac{L_1}{e})^{L_1}$.

  The remaining edges in $H$ belong to the good chain $A$. Note that $A$ was created in step $(2)$ by gluing together no more than $241k\ell$ good chains $A_r$. We can select the endpoints of these good chains $A_r$ in at most $(2L)^{482k\ell}$ ways, and we can pick the order of these chains in at most $(2L)^{482k\ell} \cdot 2^{482k\ell}$ ways. Note that, here $(2L)^{482k\ell}$ comes from the trivial bound that $2L$ is greater than the number of good chains, and the factor $2^{482k\ell}$ comes from the fact that we may flip the orientation of the good chains.

  Now, we are left to control the number of $p$-vertices updated in step $(4)$. Let $V$ be such a $p$-vertex that is updated in step $(4)$. Consider the set of at least $3$ $p$-vertices in $G$ that map to $V$ after step $(3)$. Then removing a $p$-labelling from $V$ corresponds to decreasing $\sum_{s=1}^Ld_s$ by at least $3$ in $G$, while $c(G)$ only decreases by $1$. Thus each such update corresponds to a decrease of $\Delta(G)$ by at least $\frac{1}{2\ell}$, so in total, we make no more than $2k\ell+L_1$ updates. We can select the vertices in $H$ that were updated in no more than $L^{2k\ell+L_1}$ ways which leads to
\begin{align}
    |\Psi_k^{-1}(H)| \leq (2L)^{720k\ell}\cdot (4eL)^{964k\ell L_1}\cdot(2L)^{482k\ell}\cdot 2^{482k\ell}\cdot L^{2k\ell+L_1}\leq (C_1 L)^{C_2kL_1}.
\end{align} 

\subsection{Proof of Lemma \ref{lem:psi_property}}
    For $(1)$, note that $r(\mathcal M_0)\leq r(\mathcal M)$, since under $\Psi_k$ we did not create any new $n$-vertices in the graph. We remove $n$-vertices in step $(1)$ when we remove at most $241k\ell$ bad edges. Hence, we have removed no more than $241k\ell$ distinct $n$-labels from the graph, which proves the claim.

    For $(2)$. Recall that we modify at most $241k\ell$ $p$-vertices in step $(1)$ by removing bad edges. In step $(4)$ we modify at most $2k\ell+L_1$ $p$-vertices through removal of $p$-labels for $p$-vertices with degree greater than $4$ and more than $\ell$ $p$-labels. Hence, at most $C_4k\ell+L_1$ (assuming $k>0$) $p$-vertices have a different number of labels in $\mathcal M$ compared to $\mathcal M_0$. In steps $(4)$ and $(5)$, we replace all modified $p$ -vertices with $p$ -vertices that have $\ell$ labels. Hence $(2)$ follows.

\section{Proofs of Results in Section \ref{a_x} and Section \ref{sec:combo_col}}
\label{s:proof_comb_nonbacktracking}
\begin{lem}\label{lem:Nj_2_col}
    Consider any $j\in [n]$ and let $N^r_j$ and $N^c_j$ denote the number of appearances of $j$ as a $n$-label and $p$-label, respectively, in a given column multilabeling. Then $N_j^r>0$ implies $N_j^r\geq 2$ and $N_j^c>0$ implies $N_j^c\geq 2$.
\end{lem}

\begin{proof}
    The proof is similar to Lemma A.1 of \cite{Fan2019}.
\end{proof}

\begin{proof}[Proof of Lemma \ref{lem:col_positive_excess}]
   By Lemma \ref{lem:Nj_2_col}, we have $\tilde r\leq \frac{L}{2}$ and $\frac{\tilde c}{\ell} \leq \frac{dL}{2\ell}$. Thus, we can conclude that $
    \tilde r+\frac{\tilde c}{\ell}\leq \frac{L}{2}+\frac{dL}{2\ell}$.
\end{proof}

\begin{lem}\label{lem:Nj_3_col}
     Consider any $L$-graph with $(p, n, d)$-column multilabeling. Then the following holds:
    $$\sum_{j : N^c_j\geq 3} N_j^c+\sum_{j : N^r_j\geq 3} N_j^r \leq 12\Delta \ell$$
\end{lem}

\begin{proof}
    Let $c_2=|\{j : N_j^c = 2\}|$, $c_3= |\{j : N_j^c\geq 3\}|$, $r_2=|\{j : N_j^r = 2\}|$, and $r_3= |\{j : N_j^r\geq 3\}|$. We have that $c_2+c_3=\tilde c$ and $r_2+r_3=\tilde r$. Note that $2c_2+3c_3\leq dL$ and $2r_2+3r_3\leq L$ or equivalently, $2c_2+3(\tilde c-c_2)\leq dL$ and $2r_2+3(\tilde r-r_2)\leq L$. This means that $c_2\geq 3\tilde c - dL$ and $r_2\geq 3\tilde r-L$, which gives a lower bound on the number of labels that satisfy $N_j^r=2$ or $N_j^c=2$. Therefore, we have $\sum_{j : N_j^c\geq 3} N_j^c+\sum_{j : N_j^r\geq 3} N_j^r \leq dL-2(3\tilde c-dL)+L-2(3\tilde r-L)=3dL-6\tilde c+3L-6\tilde r$. Since $\tilde r\leq\frac{L}{2}$, we also have $3dL-6\tilde c\leq 6\Delta \ell$ and similarly since $\frac{\tilde c}{\ell}\leq \frac{dL}{2\ell}$ we have $3L-6\tilde r\leq 6\Delta$. Combining these observations, the claim holds.
\end{proof}

\subsection{Proof of Proposition \ref{prop:col_t4}}

    Take any $\bf{T_4}$ edge in the graph and suppose it is incident to the $n$-vertex $U$ and the $p$-vertex $V$ ($U$ and $V$ can appear in either order). Let $W$ be the $n$-vertex preceding or following vertex $U$ so that $(W, V)$ is an edge in the $L$-graph.

    Now suppose that the $n$ labels $i_1$ and $i_2$ of $U$ and $W$ and $p$ labels $(j_1, \dots, j_a)$ of the vertex $V$ satisfy $N_{i_1}^r=N_{i_2}^r=N_{j_1}^c=N_{j_2}^c=\ldots = N_{j_d}^c= 2$. Recall by Definition \ref{def:col_multi-labelling}, we have $j_1\ne j_2\ne \dots \ne j_d$ and each of $j_1, \dots, j_d$ is a $p$-label of exactly one other $p$-vertex. Let $V^*=\{V_1, V_2, \dots\}$ be the set of vertices other than $V$ that have a $p$-label in the set $\{j_1, \dots, j_d\}$. We argue $|V^*|=1$. For each $V_k\in V^*$, let $U_k$ and $W_k$ be the $n$-vertices preceding and following $V_k$ in $\mathcal G$. Note that the $n$-labels of $U_k$ and $W_k$ must $i_1$ and $i_2$ (in either order) as otherwise one of the sub-edges $(i_1, j_1)\dots, (i_1, j_d), (i_2, j_1), \dots, (i_2, j_d)$ would appear more than once in the graph. But, $N_{i_1}^r=N_{i_2}^r=2$ so the size of the set $V^*$ can be at most $1$. Otherwise, we would have $N_{i_1}^r\geq 3$ or $N_{i_2}^r\geq 3$ or both. 
    
    This shows that there is some other vertex $V'$ with the tuple of $p$ labels equal to $(j_1, \dots, j_d)$. Let $U'$ and $W'$ be the $n$-vertices preceding and following $V'$. Note that the following equalities may not hold \[U=U',\; U=W',\; U'=W,\; U'=W',\] since this would imply that $V$ and $V'$ are consecutive $p$-vertices with the same $p$-label. However, if we assume without loss of generality that the vertex $V$ precedes $V'$ in the graph, we conclude that $(U, V)$ and $(V, W)$ are $\bf{T_1}$ edges and $(U', V')$ and $(V', W')$ are $\bf{T_3}$ edges (or vice versa if $V'$ precedes $V$). This contradicts the fact that $(U, V)$ is $\bf{T_4}$. 
    
    Therefore, a $\bf{T_4}$ edge can only occur between $n$-vertex $U$ and $p$-vertex $V$ only if at least one of $N_{i_1}^r, N_{i_2}^r, N_{j_1}^c, N_{j_2}^c, \ldots, N_{j_d}^c$ is greater than or equal to $3$. The number of such edges is bounded by $12\Delta(\mathcal L)\ell$ by Lemma \ref{lem:Nj_3_col} and so the number of $\bf{T_4}$ edges is bounded by $12 \Delta(\mathcal L) \ell$. 

\begin{lem}\label{lem:unpaired_col_boxes}
    Consider any $L$-graph with $(p, n, d)$-column multilabeling. Let $\mathcal V$ denote the largest set of $p$-vertices such that any two $p$-vertices $V, V'\in \mathcal V$ have no common $p$-label. Then $0\leq L/2-|\mathcal V|\leq 18\Delta\ell$.
\end{lem}

\begin{proof}
   It is clear that $L/2\geq \mathcal V/2$ since $N_j^c\geq 2$, for all $j$ appearing as a $p$-label. Suppose $V$ is a $p$-vertex and there is no other $p$-vertex $V'$ with the same set of $p$-labels as $V$. Then the number of such vertices is upper bounded by $24\Delta \ell$ because one of the two edges incident to $V$ is $\bf{T_4}$. 
   
   For the remaining $p$-vertices, we can group them so that the $p$-vertices in each group have the same set of $p$-labels. By Lemma \ref{lem:Nj_3_col} there are no more than $12\Delta\ell$ vertices contained in groups of size $3$ or more. Then the set $\mathcal U$ containing one $p$-vertex from each group has size at least $(L-36\Delta\ell)/2$. This implies $|\mathcal V|\geq |\mathcal U|\geq L/2-18\Delta\ell$.
\end{proof}

\subsection{Proof of Proposition \ref{prop:T3choice_col}}
    Let $e^{(r)}$ be any $\bf{T_3}$ edge of $\mathcal G$ between some vertices $U$ and $V$. Let the order of $\bf{T_3}$ be $\tau>1$ (otherwise, the statement is trivial). Let $e^{(s_1)}, \dots, e^{(s_\tau)}$ be the set of $\bf{T_1}$ edges incident to a vertex that have the same label as $U$ and are unpaired at $r$. Furthermore, let $U_{s_k}$ and $V_{s_k}$ denote the vertices incident to edge $e^{(s_k)}$. Without loss of generality, we may assume that all $U_{s_k}$'s have the same labels as $U$.
    
    We claim that there must be a $\bf{T_4}$ between the following pairs of edges (inclusive): $$\left(e^{(s_2)}, e^{(s_3)}\right), \left(e^{(s_3)}, e^{(s_4)}\right), \dots, \left(e^{(s_{\tau-1})}, e^{(s_{\tau})}\right), \left(e^{(s_{\tau})}, e^{(r)}\right)$$ 
    
    For each pair, we can find a `cycle' of edges starting at vertex $U_{s_k}$ and ending at vertex $U_{s_{k+1}}$. Assume for contradiction that such a cycle consists solely of $\bf{T_1}$ and $\bf{T_3}$ edges. We note that such a cycle may not contain a $\bf{T_1}$ directly followed by a $\bf{T_3}$ edges as this would result in either two consecutive $n$-vertices or $p$-vertices with the same labels. Since each of the edges $e^{(s_k)}$ with $k\in\{2, 3, \dots, \tau\}$ is $\bf{T_1}$, we can conclude that the edges $e^{(s_k+1)}$ are also $\bf{T_1}$ (since they are not $\bf{T_4}$). Continuing this argument we also conclude that each of the edges $e^{(s_k+2)}, e^{(s_k+3)}, \dots, e^{(s_{k+1})}$ are all $\bf{T_1}$. However, this is a contradiction since either the edge $e^{(s_{k+1}-1)}$ or $e^{(s_{k+1})}$ is between vertices $(V', U')$ (in that order) with $U'$ having the same label as $U$. Furthermore, since the first time the label of $U$ appeared in $\mathcal G$ was edge $e^{(s_1)}$, the edge $(V', U')$ is not $\bf{T_1}$.

    All these arguments allows us to conclude that the order $\tau$ of any $\bf{T_3}$ edge is bounded by two times the number of $\bf{T_4}$ edges plus one, which by Proposition \ref{prop:col_t4} is no greater than $24\Delta(\mathcal L) \ell+1$.

\end{document}